\documentclass[a4paper,fleqn]{cas-sc}

\usepackage{amsfonts,amssymb,amsthm,bm}
\usepackage{arydshln}
\usepackage{algorithm,algpseudocode}
\usepackage[numbers,sort&compress]{natbib}
\hypersetup{hypertexnames=false}

\newtheorem{theorem}{Theorem}
\newtheorem{lemma}[theorem]{Lemma}
\newtheorem{remark}{Remark}[theorem]
\begin{document}
\let\WriteBookmarks\relax
\renewcommand{\floatpagefraction}{.7}
\def\textpagefraction{.001}

\shorttitle{RBF-FD on moving manifolds}
\shortauthors{M. Lowery et al.}

\title[mode=title]{A high-order, meshless, Lagrangian--Eulerian RBF-FD method for advection--diffusion--reaction on moving manifolds}
\tnotemark[1]
\tnotetext[1]{ML and VS were supported by National Science Foundation grant
DMS 2505986. VS was also supported by the Air Force Office of Scientific
Research (AFOSR) LRIR grant FA9550-25-1-0042 and National Science Foundation
grant SHF 2403379. GW was supported by National Science Foundation grants
2309712 and 2505987.}

\author[1]{Matthew Lowery}
\ead{mlowery@cs.utah.edu}

\author[2]{Grady B. Wright}
\ead{gradywright@boisestate.edu}

\author[1]{Varun Shankar}
\cormark[1]
\ead{shankar@cs.utah.edu}

\affiliation[1]{organization={Kahlert School of Computing, University of Utah},
  city={Salt Lake City},
  state={UT},
  country={USA}}

\affiliation[2]{organization={Department of Mathematics, Boise State University},
  city={Boise},
  state={ID},
  country={USA}}

\cortext[cor1]{Corresponding author.}

\begin{abstract}
We present a high-order radial basis function-generated finite difference
(RBF-FD) method for partial differential equations on moving manifolds
$\mathcal M(t)\subset\mathbb R^3$ of co-dimension one. Our method builds on the tangent-plane
formulation of surface RBF-FD and combines Lagrangian and Eulerian treatments:
the manifold and material derivative are evolved in a Lagrangian fashion,
while the remaining surface differential operators are reconstructed on the
instantaneous point cloud and stabilized, when necessary, by a quasi-analytical
hyperviscosity formulation. Lagrangian marker drift is handled by adaptive
rearrangement using a compact global parametric model of the moving surface,
which also supplies accurate normals and geometry-based quadrature. After
rearrangement, we reconstruct the multistep history by backward
semi-Lagrangian tracing and interpolation before resuming the Lagrangian time
discretization. We exploit temporal coherence through three update strategies:
defect correction for the local RBF-FD weights, a curvature-based update of the
hyperviscosity coefficients between spectral recomputations, and a global
defect-correction iteration that reuses an incomplete LU (ILU) factorization
before preconditioned generalized minimal residual (GMRES) iterations. Finally, we enforce the prescribed global mass balance
through a scalar projection based on the evolving surface quadrature. Numerical
experiments demonstrate high-order convergence, conservation to roundoff in
source-free problems, stable long-time integration, and substantial savings
from the proposed update strategies.
\end{abstract}

\begin{keywords}
radial basis function finite differences \sep moving manifolds \sep surface partial differential equations \sep Lagrangian--Eulerian methods \sep hyperviscosity \sep high-order methods
\end{keywords}

\maketitle

\section{Introduction}
\label{sec:intro}

Partial differential equations on evolving surfaces arise in surfactant transport on fluid interfaces, phase separation and signaling on biological membranes, pattern formation on growing tissues, surface diffusion in materials, and atmospheric transport
\cite{Stone1990,XuZhaoLowengrub2006,ElliottStinner2010,RaetzRoeger2014,BarreiraElliottMadzvamuse2011,EilksElliott2008,Pudykiewicz2006,DziukElliott2013}.
These problems couple two sources of numerical difficulty: the PDE lives on a lower-dimensional manifold, and that manifold changes in time. A useful discretization must therefore track the geometry, approximate surface differential operators accurately on the evolving surface, control deterioration of the spatial discretization, and preserve the global balances built into the PDE.

The main existing approaches differ in how they represent the surface. Evolving surface finite element methods use surface-fitted meshes and provide a mature variational framework for parabolic and geometric evolution equations
\cite{DziukElliott2007EvolvingSurfaces,DziukElliott2007SurfaceFiniteElements,dziuk_fully_2012,DziukElliott2013,LubichMansourVenkataraman2013BDFEvolving,elliott_evolving_2015,kovacs_convergent_2019,elliott_unified_2021}.
Closest-point methods extend surface quantities into a narrow band
\cite{RuuthMerriman2008,MacdonaldRuuth2010,MarzMacdonald2012}, while TraceFEM restricts a bulk finite element space to the evolving surface
\cite{olshanskii_finite_2009,olshanskii_stabilized_2014,olshanskii_eulerian_2014,olshanskii_trace_2017_static,OlshanskiiXu2017TraceFEM,LehrenfeldOlshanskiiXu2018TraceFEM,grande_analysis_2018}.
Meshfree methods instead work directly from point samples of the manifold. These include GMLS methods
\cite{LiangZhao2013,TraskKuberry2020,GrossTraskKuberryAtzberger2020,JonesBoslerKuberryWright2023,hangelbroek2026sla}, global kernel collocation
\cite{fuselier2013highorder,ChenLing2020Extrinsic,ChenLing2020KernelHeatTransport}, and local radial basis function-generated finite difference (RBF-FD) methods
\cite{piret2012orthogonal,shankar2015rbffdSurfaces,lehto2017rbfcompact,shankar_rbf-loi_2018,petras_rbf-fd_2018,PetrasLingPiretRuuth2019LSRBFCPM,ShankarWrightNarayan2020,JonesBoslerKuberryWright2023,wright2023mgm}.
RBF-FD gives sparse, high-order differentiation formulas on scattered surface nodes without mesh connectivity, which makes it a natural candidate for surfaces that deform in time.

Most RBF-FD work on general manifolds treats a fixed surface. High-order formulations use extrinsic derivatives, restricted polynomial spaces, closest-point extensions, and tangent-plane constructions
\cite{piret2012orthogonal,shankar2015rbffdSurfaces,lehto2017rbfcompact,shankar_rbf-loi_2018,petras_rbf-fd_2018,PetrasLingPiretRuuth2019LSRBFCPM,ShankarWrightNarayan2020,reeger_sphere,reeger_smooth,reeger_closed,wright2023mgm,JonesBoslerKuberryWright2023}, while semi-Lagrangian RBF methods handle transport on the sphere and on more general manifolds
\cite{shankar_mesh-free_2018,petras_meshfree_2022}.
Moving the surface changes the problem substantially. The nodes move, the tangent spaces and the RBF-FD weights change, and the global sparse operators change as well. A Lagrangian point cloud can also stretch or cluster until it no longer resolves the solution well. Finally, conservation becomes closely tied to geometry in that as the surface area changes, the quadrature weights representing that area must change with it. A high-order moving-surface method has to deal with all of these effects at once.

We consider the conservative surface advection--diffusion--reaction equation
\begin{equation}
  D_t c + c\,\nabla_{\mathcal M}\!\cdot \bm v
  =
  \nu \Delta_{\mathcal M} c + \lambda c + f,
  \qquad \bm x\in\mathcal M(t),
  \label{eq:moving-surface-pde}
\end{equation}
where $c=c(\bm x,t)$ is a scalar field on the evolving manifold $\mathcal M(t)$, $\bm v$ is the material velocity, and
\[
D_t c = \partial_t c+\bm v\cdot\nabla c
\]
is the material derivative. The geometric term $c\,\nabla_{\mathcal M}\!\cdot\bm v$ accounts for local changes in surface area
\cite{DziukElliott2007EvolvingSurfaces,DziukElliott2013,ChenLing2020KernelHeatTransport}.
For a closed surface,
\[
  \frac{d}{dt}\int_{\mathcal M(t)} c\,dS
  =
  \int_{\mathcal M(t)} (\lambda c+f)\,dS.
\]
Thus, when $\lambda=f=0$, the total scalar mass must remain constant even while $\mathcal M(t)$ expands, contracts, rotates, or deforms. We use this material structure directly: we move the computational nodes with $\bm v$, discretize $D_t c$ along those material trajectories, and build the surface differential operators on the resulting point cloud at each time level.

This work makes four main contributions. First, we develop a high-order Lagrangian--Eulerian RBF-FD discretization for conservative PDEs on evolving manifolds. The unknowns move in a Lagrangian manner with the surface, while tangent-plane RBF-FD supplies the spatial operators on each instantaneous geometry. This removes the need for an Eulerian RBF-FD advection operator for the material transport. Second, we enforce the global conservation law through evolving surface quadrature and a scalar conservation projection onto the discrete mass balance dictated by the PDE; for source-free advection--diffusion this preserves $M_h^n$ to roundoff. Third, we extensively exploit temporal coherence in the evolving discretization at both the local and global levels: we update local RBF-FD weights by defect correction, update the quasi-analytical hyperviscosity parameters without repeated full spectral calculations, and reuse a frozen ILU factorization inside a global defect-correction iteration before invoking preconditioned GMRES
\cite{shankar_hyperviscosity-based_2018,ShankarWrightNarayan2020}. Fourth, we make the Lagrangian point cloud self-correcting: a geometric surface model supplies normals and quadrature information, monitors node quality, triggers rearrangement when necessary, and supports semi-Lagrangian recovery of the BDF history on the new node set. Crucially, the semi-Lagrangian step is only utilized on rearrangement.

These pieces are tightly coupled. The same geometric model that gives tangent-plane normals also gives the surface quadrature used for conservation. Rearrangement restores spatial quality but changes the material nodes, so we reconstruct the short BDF history before continuing the Lagrangian solve. Smooth motion, on the other hand, changes the local and global operators only incrementally, so a defect correction scheme proves effective at updating them without incurring the cost of a complete rebuild at every time step. The resulting method remains local, sparse, high-order, and conservative while allowing the surface and the point cloud to evolve.

We test each part of the method directly. Manufactured problems on contracting, expanding, rotating, anisotropically deforming, and nonzero-genus surfaces show high-order spatial convergence. The surface quadrature and conservation projection enforce the prescribed mass balance to roundoff in the conservative tests and generally improve the relative error. Separate ablations quantify the effects of normal estimation, adaptive hyperviscosity, local and global defect correction, long-time integration, and rearrangement with history recovery. Comparisons with evolving-surface finite element, TraceFEM, and existing RBF-based methods show competitive or improved accuracy at comparable spatial resolution or under refinement~\cite{DziukElliott2007EvolvingSurfaces,OlshanskiiXu2017TraceFEM,PetrasLingPiretRuuth2019LSRBFCPM,ChenLing2020KernelHeatTransport}.
We finish with conservative transport on a biconcave membrane driven by a three-dimensional fluid--structure interaction trajectory, where both the geometry and the material velocity come from an external simulation rather than a manufactured surface map.

The remainder of the paper is organized as follows.
Section~\ref{sec:overview} reviews tangent-plane RBF-FD and hyperviscosity on manifolds.
Section~\ref{sec:methods} presents the moving-surface discretization, adaptive operator updates, surface quadrature and conservation correction, and node rearrangement with history recovery.
Section~\ref{sec:global-bdf-linear-algebra} describes the linear algebra for the implicit BDF systems. Section~\ref{sec:tangent-plane-error-estimates} presents error estimates for our method. Section~\ref{sec:numerical-experiments} presents the numerical experiments.

\section{Overview}
\label{sec:overview}

We first summarize the two stationary-surface ingredients used throughout the
moving-surface method: tangent-plane RBF-FD approximation and manifold hyperviscosity.  The next
section applies these constructions to the surface at each time level and
develops the updates required as the geometry evolves.

\subsection{RBF-FD on the tangent plane}
\label{subsec:overview-tangent-rbffd}

Let $\mathcal M\subset\mathbb R^3$ be a smooth surface sampled at nodes
$X=\{\bm x_i\}_{i=1}^N$.  Tangent-plane RBF-FD replaces each surface stencil
by a two-dimensional local approximation in the tangent space at its center~\cite{JonesBoslerKuberryWright2023,reeger_smooth,wright2023mgm}.
%\cite{reeger_sphere,reeger_smooth,reeger_closed}.  
For a center $\bm x_i$, let
\[
  \mathcal P_i=\{\bm x_{i_1},\ldots,\bm x_{i_{n_s}}\},
  \qquad \bm x_{i_1}=\bm x_i,
\]
be its nearest-neighbor stencil.  Let $\bm n_i$ be the unit normal at
$\bm x_i$ and choose orthonormal tangent vectors $\bm t_{i,1}$ and
$\bm t_{i,2}$.  We write
\[
  T_i=[\bm t_{i,1},\bm t_{i,2}]\in\mathbb R^{3\times2}
\]
for a basis of $T_{\bm x_i}\mathcal M$.  The stencil points are represented in
centered, scaled tangent coordinates by
\begin{equation}
  \bm\zeta_{i_j}
  =
  \frac{T_i^T(\bm x_{i_j}-\bm x_i)}{\rho_i},
  \qquad j=1,\ldots,n_s,
  \label{eq:tp-coordinates}
\end{equation}
where $\rho_i>0$ is a local length scale comparable to the stencil diameter.
Its role is to nondimensionalize the local coordinates; the factors of
$\rho_i$ are restored when the physical surface derivatives are formed.

On the projected stencil we use a polyharmonic spline (PHS)
$\phi(r)=r^{2m+1}$ augmented by the polynomial space $\Pi_\ell(\mathbb R^2)$.  Its
dimension is
\[
  M=\dim\Pi_\ell(\mathbb R^2)=\frac{(\ell+1)(\ell+2)}{2}.
\]
We orthogonalize a basis $\{p_\beta\}_{\beta=1}^M$ of this space on the
projected stencil and form
\begin{equation}
  \mathcal A_i
  =
  \begin{bmatrix}
    \Phi_i & P_i\\
    P_i^T & 0
  \end{bmatrix},
  \qquad
  (\Phi_i)_{jk}=\phi(\|\bm\zeta_{i_j}-\bm\zeta_{i_k}\|),
  \qquad
  (P_i)_{j\beta}=p_\beta(\bm\zeta_{i_j}).
  \label{eq:tp-augmented-system}
\end{equation}
The polynomial degree $\ell$, PHS order $m$, and stencil size $n_s$ determine
the spatial accuracy; we denote the resulting target order by $\xi$. We choose the PHS order $2m+1$ as $\ell$ if $\ell$ is odd and $\ell-1$ if $\ell$ is even~\cite{shankar_hyperviscosity-based_2018}.

For a differential functional $\mathcal D$ in the scaled tangent coordinates,
the local RBF-FD weights satisfy
\begin{equation}
  \mathcal A_i
  \begin{bmatrix}
    \widetilde w_{i,\mathcal D}\\
    \mu_{i,\mathcal D}
  \end{bmatrix}
  =
  \begin{bmatrix}
    \mathcal D\phi(\|\bm\zeta-\bm\zeta_{i_1}\|)|_{\bm\zeta=0}\\
    \vdots\\
    \mathcal D\phi(\|\bm\zeta-\bm\zeta_{i_{n_s}}\|)|_{\bm\zeta=0}\\
    \mathcal D p_1(0)\\
    \vdots\\
    \mathcal D p_M(0)
  \end{bmatrix}.
  \label{eq:tp-local-weights}
\end{equation}
Here $\widetilde w_{i,\mathcal D}\in\mathbb R^{n_s}$ contains the RBF-FD
weights and $\mu_{i,\mathcal D}\in\mathbb R^M$ contains the multipliers that
enforce polynomial reproduction.  In particular,
\begin{equation}
  \sum_{j=1}^{n_s}
  (\widetilde w_{i,\mathcal D})_j p(\bm\zeta_{i_j})
  =
  \mathcal Dp(0),
  \qquad p\in\Pi_\ell(\mathbb R^2).
  \label{eq:tp-polynomial-reproduction}
\end{equation}

We next convert the scaled-coordinate weights to surface derivatives.  For
$\mathcal D=\partial_{\zeta_1}^2+\partial_{\zeta_2}^2$, the physical
Laplace--Beltrami weights are
\begin{equation}
  w_{i,\Delta_{\mathcal M}}
  =
  \rho_i^{-2}\widetilde w_{i,\mathcal D}.
  \label{eq:tp-laplacian-row}
\end{equation}
If $\widetilde w_{i,1}$ and $\widetilde w_{i,2}$ are the weights for
$\partial_{\zeta_1}$ and $\partial_{\zeta_2}$, then the weights for any Cartesian component $\alpha\in\{x,y,z\}$ of the surface gradient are
\begin{equation}
  w_{i,\partial_{\mathcal M,\alpha}}
  =
  \rho_i^{-1}
  \left[
    (\bm e_\alpha\cdot\bm t_{i,1})\widetilde w_{i,1}
    +(\bm e_\alpha\cdot\bm t_{i,2})\widetilde w_{i,2}
  \right].
  \label{eq:tp-gradient-row}
\end{equation}
Here $\bm e_x,\bm e_y,\bm e_z$ are the Cartesian coordinate vectors.

These local weights define the global RBF-FD operators.  For a nodal vector
$U=(u(\bm x_1),\ldots,u(\bm x_N))^T$, we set
\begin{equation}
  L_{i,i_j}=(w_{i,\Delta_{\mathcal M}})_j,
  \qquad
  (G_\alpha)_{i,i_j}=(w_{i,\partial_{\mathcal M,\alpha}})_j,
  \qquad j=1,\ldots,n_s,
  \label{eq:tp-global-operator-entries}
\end{equation}
with all remaining entries in row $i$ equal to zero.  Hence
\begin{equation}
  (LU)_i
  =
  \sum_{j=1}^{n_s}(w_{i,\Delta_{\mathcal M}})_j U_{i_j},
  \qquad
  (G_\alpha U)_i
  =
  \sum_{j=1}^{n_s}(w_{i,\partial_{\mathcal M,\alpha}})_j U_{i_j}.
  \label{eq:tp-global-operator-action}
\end{equation}
Thus $L$ approximates $\Delta_{\mathcal M}$ and
$G_x,G_y,G_z$ approximate the Cartesian components of $\nabla_{\mathcal M}$.
Section~\ref{sec:methods} constructs these operators on the evolving node set
at each time level.

\subsection{Hyperviscosity on manifolds}
\label{subsec:overview-hyperviscosity}

RBF-FD approximations of first-order surface derivatives can contain spurious
growing modes.  We use the quasi-analytical manifold hyperviscosity
construction of \cite{ShankarWrightNarayan2020}, which augments each Cartesian
surface derivative by a high-order Laplace--Beltrami term.  For a positive
integer $k$, define
\begin{equation}
  H_k z=L^kz,
  \label{eq:hv-poly-laplacian}
\end{equation}
and consider
\[
  G_\alpha z+\gamma_{\alpha,k}H_k z,
  \qquad \alpha\in\{x,y,z\}.
\]
The coefficient $\gamma_{\alpha,k}$ is estimated from the responses of the
discrete first- and high-order operators to a high-frequency probe. Let $h$ denote the characteristic nodal spacing used in the hyperviscosity
model, set $\kappa=2/h$, and define
\begin{equation}
  \varphi_i
  =
  \exp\!\left(\mathrm{i}\kappa(x_i+y_i+z_i)\right),
  \label{eq:hv-probe}
\end{equation}
where $\mathrm{i}^2=-1$.  Let
\[
  P_i=I-\bm n_i\bm n_i^T,
  \qquad
  \bm a=(1,1,1)^T.
\]
Since $\nabla_{\mathcal M}=P\nabla$, the exact surface gradient of the probe at
$\bm x_i$ is
\begin{equation}
  \nabla_{\mathcal M}\varphi(\bm x_i)
  =
  \mathrm{i}\kappa\varphi_i P_i\bm a.
  \label{eq:hv-exact-surface-gradient}
\end{equation}
We denote its $\alpha$th Cartesian component, sampled over all nodes, by
$g_{\alpha,\mathcal M}$.

For each $G_\alpha$, let
\begin{equation}
  \tau_\alpha
  \approx
  \max_{\lambda\in\sigma(G_\alpha)}\operatorname{Re}\lambda
  \label{eq:hv-tau-definition}
\end{equation}
be the estimated largest real part of its spectrum.  The measured growth
exponent is
\begin{equation}
  q_\alpha
  =
  \frac{
    \log\|G_\alpha\varphi-g_{\alpha,\mathcal M}\|_2
    -\log\tau_\alpha-\log\|\varphi\|_2
  }{\log\kappa}.
  \label{eq:hv-growth-exponent}
\end{equation}
The response of the discrete poly-Laplacian to the same probe is
\begin{equation}
  \eta_i
  =
  \left|
    \operatorname{Re}
    \left(
      \frac{(L^k\varphi)_i}
      {(-1)^k3^k\kappa^{2k}\varphi_i}
    \right)
  \right|,
  \qquad
  \overline\eta=\frac1N\sum_{i=1}^N\eta_i.
  \label{eq:hv-poly-laplacian-response}
\end{equation}
Define
\begin{equation}
  r_\alpha
  =
  \tau_\alpha
  2^{q_\alpha-2k}
  h^{2k-q_\alpha}.
  \label{eq:raw-gamma-components}
\end{equation}
The componentwise hyperviscosity coefficients are
\begin{equation}
  \boxed{
  \gamma_{\alpha,k}
  =
  3^{-k}(-1)^{1-k}
  \frac{r_\alpha}{\overline\eta}.}
  \label{eq:gamma-full}
\end{equation}
For target spatial order $\xi$, we choose
\begin{equation}
  k
  =
  \max\left\{
    k_{\min},
    \left\lceil\frac{\xi+\max\{q_x,q_y,q_z\}}{2}\right\rceil
  \right\},
  \label{eq:hv-power}
\end{equation}
which gives $2k-\max\{q_x,q_y,q_z\}\geq\xi$.  On a fixed surface these spectral
quantities can be computed once.  On a moving surface they change with the
geometry; Section~\ref{subsec:hyperviscosity-and-update} develops a geometric
update that reduces how often the spectral quantities must be recomputed.

\section{A Lagrangian--Eulerian RBF-FD method for moving surfaces}
\label{sec:methods}
We discretize \eqref{eq:moving-surface-pde} using a Lagrangian--Eulerian
approach.  The manifold and material derivative are treated in a Lagrangian
fashion: the nodes move with $\bm v$, and between rearrangements each nodal
index follows the same material point.  At each time level, we approximate the
remaining surface differential operators on the current embedded point cloud
using the tangent-plane RBF-FD construction of
Section~\ref{subsec:overview-tangent-rbffd}.

The method couples material time stepping, surface differentiation, geometric
quadrature, and occasional node rearrangement.  We state the particular
choices used here and, where relevant, the requirements needed to replace
them by alternatives of comparable accuracy.

Algorithm~\ref{alg:moving-surface-step} summarizes one time step.  The
subsections below define each operation and the criteria used for operator
updates, stabilization, conservation, and rearrangement.

\begin{algorithm}[t]
\caption{One time step of the moving-surface RBF-FD method}
\label{alg:moving-surface-step}
\textbf{Given:} Material history
$\{X^{n+1-j},C^{n+1-j}\}_{j=1}^q$, a geometric representation of
$\mathcal M(t)$, and the data needed to advance the target mass.\par
\textbf{Output:} Updated solution history through $t_{n+1}$.
\begin{enumerate}
  \item Advance the material nodes to $X^{n+1}$ and approximate the nodal
  surface velocity $V^{n+1}$ from their trajectories; see
  Section~\ref{subsec:bdf-lagrangian-eulerian}.

  \item Evaluate the surface geometry and normals, and construct the quadrature
  weights on $X^{n+1}$; see Sections~\ref{subsec:global-parametric-sbf-models}
  and~\ref{subsec:mass-conservation-projection}.

  \item Form the tangent-plane RBF-FD operators $L^{n+1}$ and
  $G_\alpha^{n+1}$, $\alpha\in\{x,y,z\}$, using direct local solves or the
  defect-correction update in Section~\ref{subsec:moving-tangent-plane-stencils}.

  \item Form the discrete surface divergence from
  \eqref{eq:discrete-surface-divergence}.  Compute or geometrically update the
  hyperviscosity coefficients as described in
  Section~\ref{subsec:hyperviscosity-and-update}, and form the stabilized
  divergence \eqref{eq:stabilized-surface-divergence}.

  \item Advance the target mass $M_\star^{n+1}$ according to
  Section~\ref{subsec:mass-conservation-projection}, using
  \eqref{eq:exact-discrete-mass-target} for a prescribed discrete mass,
  constant mass for a source-free conservative problem, or
  \eqref{eq:discrete-mass-target} when a mass rate is supplied.

  \item Solve the BDF system \eqref{eq:bdf-update} for the unprojected value
  $\widetilde C^{n+1}$, then apply the conservation projection
  \eqref{eq:mass-conservation-projection} to obtain $C^{n+1}$.

  \item If the rearrangement criterion based on
  \eqref{eq:spacing-uniformity} is satisfied, reconstruct the BDF history as
  described in Section~\ref{subsec:lagrangian-rearrangement-history}:
  \begin{enumerate}
    \item redistribute the nodes using \eqref{eq:fps-remap};
    \item trace the new material labels backward using \eqref{eq:sl-backtrace};
    \item interpolate the solution history using \eqref{eq:history-transfer};
    \item recompute quadrature weights and surface operators on the rearranged
    nodes, and apply \eqref{eq:mass-conservation-projection} to each
    reconstructed history value.
  \end{enumerate}
\end{enumerate}
\end{algorithm}

\subsection{Material time discretization}
\label{subsec:bdf-lagrangian-eulerian}
\paragraph{Material trajectories}
Let $
  X^n=\{\bm x_j^n\}_{j=1}^N$,$
  \bm x_j^n=\Phi_{t_n}(\bm x_j^0)$
where $\Phi_t$ is the flow map generated by $\bm v$.  We write
$C_j^n\approx c(\bm x_j^n,t_n)$ and collect the nodal values in
$C^n\in\mathbb R^N$.  Since $j$ remains a material label between
rearrangements, backward differentiation formulas (BDF) applied to the
sequence $\{C^n\}$ directly approximate $D_t c$.

The PDE discretization also needs the surface velocity at the nodes.  We
approximate it from the nodal trajectories using the same BDF startup
sequence as the scalar:
\[
  V^1=\frac{X^1-X^0}{\Delta t},\qquad
  V^2=\frac{3X^2-4X^1+X^0}{2\Delta t},
\]
and, after startup,
\[
  V^{n+1}
  =
  \frac{11X^{n+1}-18X^n+9X^{n-1}-2X^{n-2}}{6\Delta t}.
\]
Here $V^{n+1}\in\mathbb R^{N\times 3}$, with its columns denoted in order as as $v_x^{n+1}$, $v_y^{n+1}$, and $v_z^{n+1}$.

\paragraph{Surface divergence}
At $t_{n+1}$, let $
  L^{n+1}$, $G_x^{n+1}$, $G_y^{n+1}$, $G_z^{n+1}$
denote the RBF-FD approximations of $\Delta_{\mathcal M}$ and the Cartesian
components of $\nabla_{\mathcal M}$ on $X^{n+1}$.  The unstabilized surface
divergence is
\begin{equation}
  d_v^{n+1}
  =
  G_x^{n+1}v_x^{n+1}
  +G_y^{n+1}v_y^{n+1}
  +G_z^{n+1}v_z^{n+1}.
  \label{eq:discrete-surface-divergence}
\end{equation}
Section~\ref{subsec:hyperviscosity-and-update} adds a coordinate-wise
correction $h_v^{n+1}$ and uses
\begin{equation}
  \widetilde d_v^{n+1}=d_v^{n+1}+h_v^{n+1}.
  \label{eq:stabilized-surface-divergence}
\end{equation}

\paragraph{Backward differentiation formula}
For order $q\in\{1,2,3\}$, define the BDF coefficients
$\{a_j^{(q)}\}_{j=0}^q$ by
\[
  \frac{1}{\Delta t}\sum_{j=0}^q a_j^{(q)}C^{n+1-j}
  \approx D_t c(t_{n+1}).
\]
The BDF1--3 coefficient vectors are
\[
  (1,-1),\qquad
  \left(\frac32,-2,\frac12\right),\qquad
  \left(\frac{11}{6},-3,\frac32,-\frac13\right),
\]
and we extrapolate the explicit geometric term with
\[
  \widehat C^{n+1}
  =
  \begin{cases}
    C^n, & q=1,\\
    2C^n-C^{n-1}, & q=2,\\
    3C^n-3C^{n-1}+C^{n-2}, & q=3.
  \end{cases}
\]
Let $\mathcal F^{n+1}$ denote the prescribed explicit source, including the
reaction contribution when reaction and forcing are supplied together.  The
unprojected BDF step is
\begin{equation}
  \left(a_0^{(q)}I-\Delta t\,\nu L^{n+1}\right)\widetilde C^{n+1}
  =
  -\sum_{j=1}^q a_j^{(q)}C^{n+1-j}
  +\Delta t\left[
    \mathcal F^{n+1}
    -\widehat C^{n+1}\circ\widetilde d_v^{n+1}
  \right],
  \label{eq:bdf-update}
\end{equation}
where $\circ$ denotes pointwise multiplication.  We use BDF1 and BDF2 for
startup and BDF3 thereafter.  Diffusion is implicit; the geometric-divergence
term is explicit through $\widehat C^{n+1}$.  The same moving-surface construction can be paired with another multistep or
implicit--explicit time discretization, provided the method advances values on
material labels and supplies the required solution history after
rearrangement.

\subsection{Updating RBF-FD operators on the moving surface}
\label{subsec:moving-tangent-plane-stencils}

\paragraph{Local systems on the current surface}
At each time level we form the tangent-plane RBF-FD systems of
Section~\ref{subsec:overview-tangent-rbffd} on the current point cloud.  For
stencil center $i$, let $\mathcal A_i^{n+1}$ denote the current augmented
RBF-FD matrix and let $B_i^{n+1}$ collect the right-hand sides needed for the
Laplacian and gradient rows.  The local solve gives the scaled-coordinate
weights, which are converted to physical derivatives using
\eqref{eq:tp-laplacian-row}--\eqref{eq:tp-gradient-row}.

The straightforward approach factors each $\mathcal A_i^{n+1}$ anew.  For
smooth surface motion, consecutive local systems are close, so we reuse the
previous factorization whenever the stencil membership is unchanged.

\paragraph{Local defect correction}
For an unchanged stencil, initialize
\begin{equation}
  W_i^{(0)}=(\mathcal A_i^n)^{-1}B_i^{n+1},
  \label{eq:defect-initial-guess}
\end{equation}
and iterate
\begin{equation}
  W_i^{(r+1)}
  =
  W_i^{(r)}
  +(\mathcal A_i^n)^{-1}
  \left(B_i^{n+1}-\mathcal A_i^{n+1}W_i^{(r)}\right),
  \qquad r=0,\ldots,r_{\max}-1.
  \label{eq:defect-correction}
\end{equation}
We stop the iteration when
\begin{equation}
  \frac{\|B_i^{n+1}-\mathcal A_i^{n+1}W_i^{(r+1)}\|}
       {\|B_i^{n+1}\|}
  \leq \varepsilon_{\mathrm{dc}}.
  \label{eq:defect-acceptance}
\end{equation}
If the stencil membership changes, the old factors are unavailable, or the
criterion \eqref{eq:defect-acceptance} is not met, we solve the current system
directly and retain its factors.  Defect correction changes only the linear
algebra used to obtain the current RBF-FD weights.  Direct local solves or
other factor-reuse strategies may be used instead without changing the spatial
discretization.

\subsection{Moving-surface hyperviscosity update}
\label{subsec:hyperviscosity-and-update}

\paragraph{Stabilization of the surface divergence}
At selected time levels we compute the manifold hyperviscosity parameters
directly from the current operators using
Section~\ref{subsec:overview-hyperviscosity}.  We apply the correction to the
three velocity components that enter
\eqref{eq:discrete-surface-divergence}.  Define
\begin{equation}
  H_k^{n+1}z=(L^{n+1})^kz,
  \label{eq:moving-poly-laplacian}
\end{equation}
and
\begin{equation}
  h_v^{n+1}
  =
%  \sum_{\ell=1}^3
%  \gamma_{\ell,k}^{n+1}H_k^{n+1}v_\ell^{n+1}.
  H_k^{n+1}\left(\gamma_{x,k}^{n+1}v_x^{n+1}+\gamma_{y,k}^{n+1}v_y^{n+1}+\gamma_{z,k}^{n+1}v_z^{n+1}\right)
  \label{eq:hyperviscosity-term}
\end{equation}
We evaluate $H_k^{n+1}$ by $k$ sparse matrix--vector products; we never form
$(L^{n+1})^k$.  Equation~\eqref{eq:bdf-update} then contains the correction
$-\widehat C^{n+1}\circ h_v^{n+1}$.

Subsequent steps use the stabilized divergence
$\widetilde d_v^{n+1}$ only through \eqref{eq:bdf-update}.  Another stable
discretization of $\nabla_{\mathcal M}\!\cdot\bm v$ may therefore be
substituted here without changing the time discretization or conservation
projection.

\paragraph{Geometric update of the coefficients}
Direct evaluation of the coefficients requires three largest-real-part
eigenvalue estimates and the growth exponents in
\eqref{eq:hv-growth-exponent}.  We avoid repeating these spectral calculations
at every time step by updating the raw contributions $r_\ell$ from the
changing geometry.  Holding the most recently computed
$q_\alpha$ fixed, the exact ratio would be
\begin{equation}
  r_\alpha^{n+1}
  =
  r_\alpha^n
  \left(\frac{\tau_\alpha^{n+1}}{\tau_\alpha^n}\right)
  \left(\frac{h^{n+1}}{h^n}\right)^{2k-q_\alpha^n}.
  \label{eq:raw-ratio-exact}
\end{equation}
We estimate the scale change from
\begin{equation}
  K_\alpha^n
  =
  \operatorname{rms}\!\left(L^nX_\alpha^n\right),
  \qquad \alpha=x,y,z,
  \label{eq:coordinate-curvature-proxy}
\end{equation}
where $X_\alpha^n$ is the $\alpha$ coordinate function of the point cloud.
Since $L^nX^n$ approximates $\Delta_{\mathcal M}X$, $K_x^n$, $K_y^n$, $K_z^n$ provide coordinate-wise geometric scales.

Under a uniform scaling by $s$,
$h^{n+1}/h^n=s$, $\tau_\alpha^{n+1}/\tau_\alpha^n\approx s^{-1}$, and
$K_\alpha^{n+1}\approx s^{-1}K_\alpha^n$.  Substitution in
\eqref{eq:raw-ratio-exact} gives the geometric update
\begin{equation}
  \widetilde r_\alpha^{\,n+1}
  =
  r_\alpha^n
  \left(
    \frac{K_\alpha^n}{K_\alpha^{n+1}}
  \right)^{2k-q_\alpha^n-1}.
  \label{eq:raw-gamma-update}
\end{equation}
For a general smooth deformation we apply the same relation componentwise.
We recompute $\overline\eta^{\,n+1}$ from $L^{n+1}$ using the probe in
\eqref{eq:hv-poly-laplacian-response}, and set
\begin{equation}
  \widetilde\gamma_{\alpha,k}^{\,n+1}
  =
  3^{-k}(-1)^{1-k}
  \frac{\widetilde r_\alpha^{\,n+1}}
       {\overline\eta^{\,n+1}},
  \qquad \alpha\in\{x,y,z\}.
  \label{eq:gamma-predictor}
\end{equation}

\paragraph{Recomputation criterion}
We use the geometrically updated coefficient vector when
\[
  \frac{
    \|\widetilde{\bm\gamma}_k^{\,n+1}-\bm\gamma_k^n\|_2
  }{
    \max(\|\bm\gamma_k^n\|_2,\epsilon)
  }
  \leq \varepsilon_\gamma,
\]
provided $K_x^{n+1}$, $K_y^{n+1}$, and $K_z^{n+1}$ are finite and positive, and the prescribed
maximum number of consecutive geometric updates has not been reached.
Otherwise, we recompute the eigenvalue estimates, growth exponents, and
coefficients from \eqref{eq:gamma-full}.  Whenever these quantities are
recomputed, $k$ is selected by \eqref{eq:hv-power}; the previous value serves
as the lower bound, so $k$ can increase but does not decrease.  We use
$\varepsilon_\gamma=0.05$ and recompute the spectral quantities after at most
five geometric coefficient updates.

\subsection{Global parametric geometry models and normals}
\label{subsec:global-parametric-sbf-models}

\paragraph{Geometric representation}
The method requires a geometric representation of the moving surface from
which we can evaluate the surface map, parameter tangents, normals, and surface
Jacobian.  Rearrangement additionally requires interpolation in the material
coordinates.  These quantities may come from an analytic parameterization or
from a numerical geometric model.  For the latter, we use global radial basis
function models on fixed material coordinates.

\paragraph{Sphere-like surfaces}
For a surface homeomorphic to the sphere, let $\bm u\in\mathbb S^2$ be a
material coordinate and $\bm X(\bm u,t)\in\mathbb R^3$ the embedded surface.
Choose control sites $\{\bm u_j\}_{j=1}^M\subset\mathbb S^2$ and define
\[
  d_{\mathbb S}(\bm u,\bm u_j)=\|\bm u-\bm u_j\|_2.
\]
With a polyharmonic spline (PHS) kernel $\phi(r)=r^{2m+1}$, we approximate the
geometry by the spherical basis function (SBF) model
\begin{equation}
  \bm X_M(\bm u,t)
  =
  \sum_{j=1}^M \bm a_j(t)\,
  \phi\!\left(d_{\mathbb S}(\bm u,\bm u_j)\right),
  \label{eq:spherical-sbf-geometry}
\end{equation}
with $
  \bm X_M(\bm u_i,t)=\bm X(\bm u_i,t)$, 
  $i=1,\ldots,M.$
The control sites are fixed in material space, so we factor the interpolation
matrix once and update only its right-hand side as the surface moves.  We use
a deterministic farthest-point subset for the controls; another well-spaced
control set gives the same construction.

\paragraph{Torus-like surfaces}
For material coordinates $(\theta,\varphi)\in\mathbb T^2=[0,2\pi)^2$, define
\[
  \bm q(\theta,\varphi)
  =
  (\cos\theta,\sin\theta,\cos\varphi,\sin\varphi)\in\mathbb R^4
\]
and the periodic chordal distance
\begin{equation}
  d_{\mathbb T}\bigl((\theta,\varphi),(\theta_j,\varphi_j)\bigr)
  =
  \left\|
    \bm q(\theta,\varphi)-\bm q(\theta_j,\varphi_j)
  \right\|_2.
  \label{eq:toroidal-distance}
\end{equation}
The corresponding toroidal basis function (TBF) model is
\begin{equation}
  \bm X_M(\theta,\varphi,t)
  =
  \sum_{j=1}^M \bm a_j(t)\,
  \phi\!\left(
    d_{\mathbb T}\bigl((\theta,\varphi),(\theta_j,\varphi_j)\bigr)
  \right).
  \label{eq:toroidal-basis-geometry}
\end{equation}
The four-dimensional embedding only enforces periodicity of the material
coordinates.

\paragraph{Normals and scalar interpolation}
For either parameterization, differentiating $\bm X_M$ gives the tangents.  If
$(\alpha_1,\alpha_2)$ is a local material chart, we use
\begin{equation}
  \bm n_M(\alpha,t)
  =
  \frac{
    \partial_{\alpha_1}\bm X_M(\alpha,t)
    \times
    \partial_{\alpha_2}\bm X_M(\alpha,t)
  }{
    \left\|
      \partial_{\alpha_1}\bm X_M(\alpha,t)
      \times
      \partial_{\alpha_2}\bm X_M(\alpha,t)
    \right\|_2
  }.
  \label{eq:sbf-normal-model}
\end{equation}
The same basis gives a scalar interpolant:
\begin{equation}
  I_M^n g(\alpha)
  =
  \sum_{j=1}^M b_j^n\,
  \phi\!\left(d(\alpha,\alpha_j)\right),
  \qquad
  I_M^n g(\alpha_i)=g(\alpha_i),
  \label{eq:global-sbf-scalar-interpolant}
\end{equation}
where $d=d_{\mathbb S}$ or $d_{\mathbb T}$.  The local RBF-FD operators remain
those of Section~\ref{subsec:overview-tangent-rbffd}; the global model supplies
geometry and, when needed, interpolation of the solution history.

For a reduced geometry model, we choose enough controls that its normal error
stays below the target PDE error.  With material fill distance
\[
H_M=\max_{\alpha_i}\min_{1\leq j\leq M}d(\alpha_i,\alpha_j),
\]
normal error $O(H_M^{q_g})$, and RBF-FD error $O(h^\xi)$, we use the balancing
criterion
\begin{equation}
  H_M^{q_g}\lesssim c_{\rm geom} h^\xi,
  \label{eq:sbf-control-balance}
\end{equation}
where $c_{\rm geom}$ is a constant parameter. For the degree-seven PHS-SBF geometry interpolant used here with $\phi(r) = r^7$, we take $q_g =8$  and $c_{\rm geom} = 0.1$. We select the number \(M\) of geometry control
sites so that
\[
H_M^8 \le 0.1h^\xi.
\]
\subsection{Geometry-derived quadrature and mass projection}
\label{subsec:mass-conservation-projection}

\paragraph{Material quadrature}
The mass projection requires a surface quadrature rule that evolves with the
geometry.  Let $\widehat{\mathcal M}$ be the fixed material manifold with
reference measure $d\mu(a)$, and let $
  \bm X_M(\cdot,t):\widehat{\mathcal M}\rightarrow\mathcal M(t)
$
be the geometry map.  For material labels $a_i\in\widehat{\mathcal M}$, choose
a positive quadrature rule
\begin{equation}
  \int_{\widehat{\mathcal M}}\widehat g(a)\,d\mu(a)
  \approx
  \sum_{i=1}^N\widehat\omega_i\,\widehat g(a_i).
  \label{eq:material-quadrature}
\end{equation}
If
\[
  dS_t(\bm X_M(a,t))=J_M(a,t)\,d\mu(a),
\]
then the physical weights are
\begin{equation}
  w_i^n=\widehat\omega_iJ_i^n,
  \qquad
  J_i^n=J_M(a_i,t_n).
  \label{eq:geometry-quadrature-weights}
\end{equation}
For a two-parameter chart $\alpha=(\alpha_1,\alpha_2)$,
\begin{equation}
  J_M(\alpha,t)
  =
  \left\|
    \partial_{\alpha_1}\bm X_M(\alpha,t)
    \times
    \partial_{\alpha_2}\bm X_M(\alpha,t)
  \right\|_2,
  \label{eq:chart-surface-jacobian}
\end{equation}
while for $\alpha=\bm u\in\mathbb S^2$ and an orthonormal basis
$\bm e_1,\bm e_2\in T_{\bm u}\mathbb S^2$,
\begin{equation}
  J_M(\bm u,t)
  =
  \left\|
    D\bm X_M(\bm u,t)\bm e_1(\bm u)
    \times
    D\bm X_M(\bm u,t)\bm e_2(\bm u)
  \right\|_2.
  \label{eq:spherical-surface-jacobian}
\end{equation}

The reference weights $\widehat\omega_i$ depend on how the material sites are
generated.  Equal weights are natural for quasi-uniform sampling with respect
to $d\mu$.  If the initial nodes are uniform in physical area, we instead use
\begin{equation}
  \widehat\omega_i
  =
  \frac{|\mathcal M(0)|}{N J_i^0},
  \qquad
  J_i^0=J_M(\alpha_i,0),
  \label{eq:initial-area-uniform-material-weights}
\end{equation}
which gives $w_i^0=|\mathcal M(0)|/N$.

\paragraph{Direct surface quadrature}
When a convenient material quadrature is unavailable, we use vertex areas on
a surface triangulation $\mathcal T^n$:
\begin{equation}
  w_i^n
  =
  \frac13
  \sum_{\tau\in\mathcal T^n:\,i\in\tau}|\tau|.
  \label{eq:vertex-area-quadrature}
\end{equation}
The triangulation is needed only when these weights are computed.
We use \eqref{eq:vertex-area-quadrature} for this case.  Any consistent
positive surface quadrature may be used in the conservation step below.

After rearrangement, we construct quadrature weights for the new labels rather
than transporting the old vector of weights.  For material quadrature this
means recomputing
\[
  \widetilde w_i^n
  =
  \widetilde{\widehat\omega}_i
  J_M(\widetilde \alpha_i,t_n),
\]
and for direct surface quadrature it means recomputing the rule on the new
point cloud.

\paragraph{Discrete mass balance}
Define
\begin{equation}
  Q_h^n[g]=\sum_{i=1}^Nw_i^ng_i^n,
  \qquad
  \|g\|_{h,n}^2=Q_h^n[g^2].
  \label{eq:geometry-quadrature-functional}
\end{equation}
For an unprojected BDF solution $\widetilde C^{n+1}$,
\[
  M_h^{n+1}(\widetilde C)
  =Q_h^{n+1}[\widetilde C^{n+1}]
  =(w^{n+1})^T\widetilde C^{n+1}.
\]
For \eqref{eq:moving-surface-pde}, the transport theorem gives
\[
  M(t)=\int_{\mathcal M(t)}c\,dS,
  \qquad
  \frac{dM}{dt}
  =
  \int_{\mathcal M(t)}(\lambda c+f)\,dS.
\]
We evolve the target mass directly from the physical balance law; the full
pointwise residual does not enter the mass update.

\paragraph{Target mass}
We consider three constructions for the target mass.  Let
\[
  M_h^0=Q_h^0[C^0].
\]
If an exact discrete mass $M_{\mathrm{ex},h}(t_n)$ is available, we preserve
the initial quadrature offset and set
\begin{equation}
  M_\star^n
  =
  M_h^0
  +M_{\mathrm{ex},h}(t_n)
  -M_{\mathrm{ex},h}(t_0).
  \label{eq:exact-discrete-mass-target}
\end{equation}
For a source-free conservative problem, $M_\star^n=M_h^0$.  If the physical
mass rate is available, define
\begin{equation}
  R_h^n
  =
  \begin{cases}
    R(t_n), & \text{for a supplied scalar mass-rate function},\\
    Q_h^n[s^n], & \text{for a supplied nodal balance source }s^n,
  \end{cases}
  \label{eq:discrete-mass-rate}
\end{equation}
with the physical choice
$R(t)=\int_{\mathcal M(t)}(\lambda c+f)\,dS$ for
\eqref{eq:moving-surface-pde}.  We accumulate the rate with
\begin{align}
  I_h^0 &= 0, \nonumber\\
  I_h^1 &= \frac{\Delta t}{2}(R_h^0+R_h^1), \nonumber\\
  I_h^2 &= \frac{\Delta t}{3}(R_h^0+4R_h^1+R_h^2), \nonumber\\
  I_h^n &= I_h^{n-1}
  +\frac{\Delta t}{12}(5R_h^n+8R_h^{n-1}-R_h^{n-2}),
  \qquad n\geq3,
  \label{eq:multistep-mass-integral}
\end{align}
and set
\begin{equation}
  M_\star^n=M_h^0+I_h^n.
  \label{eq:discrete-mass-target}
\end{equation}
The startup uses trapezoidal and Simpson quadrature; the subsequent update is
the third-order endpoint-inclusive Adams--Moulton rule.  A different mass-rate
integrator can be used provided it matches the accuracy required by the time
discretization.

\paragraph{Mass projection}
Given $M_\star^{n+1}$, we project by a constant shift:
\begin{equation}
  C^{n+1}
  =
  \widetilde C^{n+1}
  +\delta^{n+1}\bm 1,
  \qquad
  \delta^{n+1}
  =
  \frac{M_\star^{n+1}-(w^{n+1})^T\widetilde C^{n+1}}
       {(w^{n+1})^T\bm 1}.
  \label{eq:mass-conservation-projection}
\end{equation}
The constant shift in \eqref{eq:mass-conservation-projection} minimizes the
correction in the quadrature-weighted discrete $L^2$ norm subject to the
prescribed mass.  We apply the same projection to any
history values created by rearrangement, using the quadrature rule at the
corresponding time level.  The constant shift leaves the RBF-FD operators and implicit diffusion system
unchanged.  Other conservative projections may be used with the same
quadrature and target mass.

\subsection{Lagrangian rearrangement and semi-Lagrangian history reconstruction}
\label{subsec:lagrangian-rearrangement-history}

A material point cloud inherits the deformation of the surface flow.  Over
time, nodes can cluster or separate enough to degrade the local RBF-FD
stencils.  We monitor the node distribution and occasionally replace the material labels
by a better-spaced set.  The BDF history must then be reconstructed on the new
labels so that the multistep scheme can continue without restarting.

\paragraph{Rearrangement criterion and redistribution}
We monitor the spacing-uniformity ratio
\begin{equation}
  Q(X)
  =
  \frac{\max_i\delta_i}{\min_i\delta_i},
  \qquad
  \delta_i=\min_{j\neq i}\|\bm x_i-\bm x_j\|_2.
  \label{eq:spacing-uniformity}
\end{equation}
We rearrange the nodes when $Q(X^n)$ exceeds a prescribed tolerance, or when a
short extrapolation predicts that the tolerance will be exceeded within the
next few time steps.  A different geometric measure may be used here; the subsequent rearrangement
procedure only needs a criterion for loss of point-set quality.

Suppose rearrangement occurs at time $t_n$.  Let $\Theta$ be a material parameter
domain for $\bm X(\bm\theta,t)$.  We choose new material parameters
\[
  \widetilde\Theta^n
  =\{\widetilde{\bm\theta}_j^n\}_{j=1}^N
\]
that are quasi-uniform on $\mathcal M(t_n)$.  We draw candidate material coordinates according to the surface-area density
\begin{equation}
  J(\bm\theta,t_n)
  =
  \left\|
    \partial_{\theta_1}\bm X(\bm\theta,t_n)
    \times
    \partial_{\theta_2}\bm X(\bm\theta,t_n)
  \right\|_2
  \label{eq:surface-area-density}
\end{equation}
and apply farthest-point sampling in the embedded metric:
\begin{equation}
  \widetilde{\bm\theta}_{k+1}^n
  \in
  \arg\max_{\bm\theta\in\Theta_{\rm cand}}
  \min_{1\leq j\leq k}
  \left\|
    \bm X(\bm\theta,t_n)
    -\bm X(\widetilde{\bm\theta}_j^n,t_n)
  \right\|_2.
  \label{eq:fps-remap}
\end{equation}
The new cloud is
\[
  \widetilde X^n
  =
  \{\bm X(\widetilde{\bm\theta}_j^n,t_n)\}_{j=1}^N.
\]
We use farthest-point sampling for redistribution.  Any resampling method that
produces a sufficiently regular point set may be used instead.

\paragraph{Backward characteristic tracing}
The next BDF3 step requires solution values at $t_n,t_{n-1},t_{n-2}$ on the
same material labels.  We obtain the older locations by tracing each new label
backward along the material flow.  Let
\[
  D_\theta\bm X
  =
  \begin{bmatrix}
    \partial_{\theta_1}\bm X & \partial_{\theta_2}\bm X
  \end{bmatrix},
  \qquad
  G_\theta=(D_\theta\bm X)^TD_\theta\bm X.
\]
When the motion is specified in ambient coordinates, the material-coordinate
velocity is
\begin{equation}
  \bm a(\bm\theta,t)
  =
  G_\theta^{-1}(D_\theta\bm X)^T
  \left[
    \bm v(\bm X(\bm\theta,t),t)-\partial_t\bm X(\bm\theta,t)
  \right].
  \label{eq:material-coordinate-velocity}
\end{equation}
If the material-coordinate velocity is already available, we use it directly.
For $r=1,2$, the backtraced labels satisfy
\begin{equation}
  \frac{d\bm\theta_j}{d\sigma}
  =
  \bm a(\bm\theta_j(\sigma),\sigma),
  \qquad
  \bm\theta_j(t_n)=\widetilde{\bm\theta}_j^n,
  \qquad
  \sigma\in[t_{n-r},t_n],
  \label{eq:sl-backtrace}
\end{equation}
with
\[
  \widetilde{\bm x}_j^{n-r}
  =
  \bm X(\bm\theta_j(t_{n-r}),t_{n-r}).
\]
We use fourth-order Runge--Kutta substepping for these characteristic
problems.  Any trajectory integrator whose error is small relative to the
spatial and temporal discretization errors may be used instead.

\paragraph{History interpolation}
Let $I_h^{n-r}$ be a high-order surface interpolant built from
$(X^{n-r},C^{n-r})$.  For $r=0,1,2$, set
\begin{equation}
  \widetilde C_j^{n-r}
  =
  I_h^{n-r}C^{n-r}
  \bigl(\widetilde{\bm x}_j^{n-r}\bigr).
  \label{eq:history-transfer}
\end{equation}
We may use either the local tangent-plane RBF interpolant or the global SBF/TBF
interpolant \eqref{eq:global-sbf-scalar-interpolant}; we use the former for parsimony.  Any surface interpolant of sufficient accuracy may be used for this history
reconstruction.

The new history
\[
  (\widetilde X^n,\widetilde C^n),\qquad
  (\widetilde X^{n-1},\widetilde C^{n-1}),\qquad
  (\widetilde X^{n-2},\widetilde C^{n-2})
\]
replaces the old one.  We recompute or update the differentiation operators on
$\widetilde X^n$, recompute the quadrature weights on the new labels, and
apply \eqref{eq:mass-conservation-projection} to each interpolated history
value.  The BDF discretization then continues on the new material labels without a restart.

Algorithm~\ref{alg:moving-surface-step} now provides the top-level view of the
complete time step.  The remaining operation is the global sparse solve in
\eqref{eq:bdf-update}.  Section~\ref{sec:global-bdf-linear-algebra}
describes the defect-correction iteration and preconditioner reuse used for
that system.

%\section{Linear algebra for moving-surface BDF systems}
\section{Solving the moving-surface BDF systems}
\label{sec:global-bdf-linear-algebra}

The discretization above defines a sequence of sparse BDF systems.  The
linear algebra used to solve these systems is an additional acceleration
layer, separate from the local RBF-FD stencil construction.  At each time step
we solve
\begin{equation}
  A^{n+1}C^{n+1}=b^{n+1},
  \qquad
  A^{n+1}=I-\beta_q\Delta t\,\nu L^{n+1}.
  \label{eq:global-bdf-system}
\end{equation}
Here \(\beta_q\) is the diffusion coefficient obtained after dividing the
BDF\(q\) formula by its leading coefficient; for the BDF1, BDF2, and BDF3
steps used here,
\[
  \beta_1=1,\qquad \beta_2=\frac{2}{3},\qquad
  \beta_3=\frac{6}{11}.
\]
The right-hand side \(b^{n+1}\) contains the previous BDF values, forcing,
reaction, explicit geometric-divergence term, and explicit hyperviscosity
term from \eqref{eq:bdf-update}.  Although \eqref{eq:global-bdf-system} is a
global sparse problem rather than a local stencil problem, consecutive
matrices are close when the surface moves smoothly.  We exploit this by using
a global analogue of the local defect correction in
\eqref{eq:defect-correction}.

Let
\[
  M_f=P_f^TL_fU_f\approx A^f
\]
denote a frozen incomplete factorization of a recent BDF matrix \(A^f\), with
the permutation convention \(L_fU_f\approx P_fA^f\).  The initial guess is the
explicit BDF extrapolate
\[
  C^{(0)}=\widehat C^{n+1}.
\]
Before launching a Krylov iteration, we apply a small fixed number of global
defect sweeps
\begin{equation}
  r^{(s)}=b^{n+1}-A^{n+1}C^{(s)},\qquad
  C^{(s+1)}
  =
  C^{(s)}+M_f^{-1}r^{(s)},
  \qquad s=0,\ldots,s_{\max}-1.
  \label{eq:global-defect-sweep}
\end{equation}
The residual is always formed with the current matrix \(A^{n+1}\); only the
approximate inverse is frozen.  The defect stage is stopped early if
\[
  \frac{\|r^{(s)}\|_2}{\|b^{n+1}\|_2}
  \leq \varepsilon_{\rm lin}.
\]
Thus these sweeps are not a change in the discretization.  They are cheap
pre-Krylov corrections using triangular solves with \(L_f\) and \(U_f\).

The corrected vector is then used as the initial guess for left-preconditioned
GMRES,
\begin{equation}
  M_f^{-1}A^{n+1}C^{n+1}
  =
  M_f^{-1}b^{n+1}.
  \label{eq:frozen-ilu-gmres}
\end{equation}
The linear tolerance is chosen below the expected spatial error, using
\[
  \varepsilon_{\rm lin}
  =
  \max\!\left(10^{-10},\min\!\left(10^{-8},0.05h^\xi\right)\right)
\]
unless a stricter user-prescribed tolerance is supplied.  The frozen ILU is
reused across time steps with the same BDF coefficient \(\beta_q\).  It is
refreshed at startup, when \(\beta_q\) changes, when GMRES fails, when the
previous GMRES solve reaches a prescribed iteration threshold, or after an
optional maximum number of reuses.  If a solve fails with the current frozen
factors, the factors are rebuilt for \(A^{n+1}\) and the defect-plus-GMRES
solve is repeated.  The result is a second defect-correction mechanism, now
at the global sparse-matrix level: local RBF-FD defect correction exploits
smooth stencil motion, while the global solve exploits smooth motion of the
assembled BDF matrix.

\section{Convergence of tangent-plane RBF-FD}
\label{sec:tangent-plane-error-estimates}

We now establish convergence of the tangent-plane operators from
Section~\ref{subsec:overview-tangent-rbffd}.  We use the notation from that
section and take $\ell=\xi+1$
so that polynomial reproduction through degree \(\ell\) gives the target
order \(\xi\) for the Laplace--Beltrami operator.  The corresponding surface
gradient has one additional order.  This choice matches the spatial-order
parameter \(\xi\) used throughout the numerical results.

For a stencil centered at \(\bm x_i\), let
 \(w_{i,\Delta_{\mathcal M}}\) and
\(w_{i,\partial_{\mathcal M,\alpha}}\) be the physical weights in
\eqref{eq:tp-laplacian-row} and \eqref{eq:tp-gradient-row}.  We assume the
scaled stencils remain uniformly bounded,
\begin{equation}
  \max_{i,j}\|\bm\zeta_{i_j}\|_2\leq C_{\rm st},
  \label{eq:theory-scaled-stencil-bound}
\end{equation}
and define the dimensionless stability factors
\begin{equation}
  \Lambda_\Delta
  =\max_i
    \rho_i^2\|w_{i,\Delta_{\mathcal M}}\|_1,
  \qquad
  \Lambda_\nabla
  =\max_{i,\alpha}
    \rho_i\|w_{i,\partial_{\mathcal M,\alpha}}\|_1.
  \label{eq:theory-stability-factors}
\end{equation}
These are the usual stability quantities for polynomial-exact numerical
differentiation formulas~\cite{DavydovSchaback2018}.  The reproduction
property \eqref{eq:tp-polynomial-reproduction}, together with the rescalings
\eqref{eq:tp-laplacian-row}--\eqref{eq:tp-gradient-row}, implies exactness of
the physical weights on \(\Pi_{\xi+1}(\mathbb R^2)\) in the physical tangent
coordinates \(\bm y=\rho_i\bm\zeta\).  For each point cloud, define the global spatial scale
\[
  h=\max_{1\leq i\leq N}\rho_i.
\]
Thus, \(h\to0\) along a refinement sequence means that the largest local
stencil scale tends to zero.  For quasi-uniform point sets and the scaling used
here, \(h\) is comparable to the surface fill distance.

\begin{lemma}[Exact tangent-plane identities]
\label{lem:exact-tangent-plane-identities}
Let \(T_i\) span the exact tangent plane at \(\bm x_i\).  For a sufficiently
small neighborhood, write the surface as the Monge graph
\begin{equation}
  \Psi_i(\bm y)
  =\bm x_i+T_i\bm y+\bm n_i g_i(\bm y),
  \qquad g_i(0)=0,\qquad \nabla g_i(0)=0.
  \label{eq:theory-exact-monge-chart}
\end{equation}
If \(u_i=u\circ\Psi_i\), then
\begin{equation}
  \partial_{\mathcal M,\alpha}u(\bm x_i)
  =\bm e_\alpha^TT_i\nabla u_i(0),
  \qquad
  \Delta_{\mathcal M}u(\bm x_i)=\Delta u_i(0).
  \label{eq:theory-exact-tangent-identities}
\end{equation}
\end{lemma}

\begin{proof}
The metric induced by \eqref{eq:theory-exact-monge-chart} is
\(G=I+\nabla g_i\nabla g_i^T\).  Hence \(G(0)=I\) and
\(\partial_aG(0)=0\).  Substitution into the standard coordinate formulas for
\(\nabla_{\mathcal M}\) and \(\Delta_{\mathcal M}\) gives
\eqref{eq:theory-exact-tangent-identities}; see also
\cite{JonesBoslerKuberryWright2023}.
\end{proof}

\begin{theorem}[Convergence with exact normals]
\label{thm:exact-normal-tangent-rbffd}
Let \(\mathcal M\subset\mathbb R^3\) be a compact co-dimension one manifold that is closed and
\(C^{\xi+2}\).  Consider a sequence of point clouds \(X_h\subset\mathcal M\)
with \(h\to0\) (in the sense described above), and construct the tangent-plane RBF-FD operators using the
exact surface normals.  Suppose the formulas reproduce
\(\Pi_{\xi+1}(\mathbb R^2)\) and that there is a constant \(C_\Lambda\),
independent of \(h\), such that
\begin{equation}
  \Lambda_\nabla\leq C_\Lambda,
  \qquad
  \Lambda_\Delta\leq C_\Lambda
  \label{eq:theory-exact-uniform-stability}
\end{equation}
for every point cloud in the refinement sequence.  Then, for
\(u\in C^{\xi+2}(\mathcal M)\) and
\(U=(u(\bm x_1),\ldots,u(\bm x_N))^T\),
\begin{align}
  \max_{\alpha\in\{x,y,z\}}
  \|G_\alpha U-\partial_{\mathcal M,\alpha}u|_{X_h}\|_\infty
  &\leq C h^{\xi+1}\|u\|_{C^{\xi+2}(\mathcal M)},
  \label{eq:exact-normal-global-gradient}\\
  \|LU-\Delta_{\mathcal M}u|_{X_h}\|_\infty
  &\leq C h^\xi\|u\|_{C^{\xi+2}(\mathcal M)}.
  \label{eq:exact-normal-global-laplacian}
\end{align}
The constant \(C\) is independent of \(h\) and \(u\).  Hence the tangent-plane
gradient and Laplace--Beltrami approximations converge with orders
\(\xi+1\) and \(\xi\), respectively.
\end{theorem}

\begin{proof}
Fix a stencil centered at \(\bm x_i\), let \(\Psi_i\) be the Monge
parameterization in \eqref{eq:theory-exact-monge-chart}, and define
\[
  u_i(\bm y)=u(\Psi_i(\bm y)).
\]
Lemma~\ref{lem:exact-tangent-plane-identities} identifies a surface-gradient
component or the Laplace--Beltrami operator at \(\bm x_i\) with the
corresponding Euclidean differential operator applied to \(u_i\) at
\(\bm y=0\).  Denote this Euclidean operator by \(\mathcal D\), and let
\(s\in\{1,2\}\) be its differential order.

Let \(p_i\) be the degree-\((\xi+1)\) Taylor polynomial of \(u_i\) at the
origin.  Since \(s\leq\xi+1\),
\[
  \mathcal D p_i(0)=\mathcal D u_i(0).
\]
After restoring the physical scaling in \eqref{eq:tp-laplacian-row}--\eqref{eq:tp-gradient-row}, polynomial
reproduction gives
\[
  \mathcal D p_i(0)
  =\sum_{j=1}^{n_s}w_{ij}p_i(\bm y_{i_j}),
  \qquad
  \bm y_{i_j}=T_i^T(\bm x_{i_j}-\bm x_i).
\]
Therefore the local error is
\[
  E_i
  =\sum_{j=1}^{n_s}w_{ij}
  \bigl[p_i(\bm y_{i_j})-u_i(\bm y_{i_j})\bigr].
\]
Taylor's theorem and \eqref{eq:theory-scaled-stencil-bound} imply
\[
  |u_i(\bm y_{i_j})-p_i(\bm y_{i_j})|
  \leq
  C\rho_i^{\xi+2}\|u\|_{C^{\xi+2}(\mathcal M)}.
\]
Thus 
\[
  |E_i|
  \leq
  C\rho_i^{\xi+2}\|w_i\|_1
  \|u\|_{C^{\xi+2}(\mathcal M)}.
\]
For an order-\(s\) operator, set
\(\Lambda_i=\rho_i^s\|w_i\|_1\).  By
\eqref{eq:theory-exact-uniform-stability}, \(\Lambda_i\leq C_\Lambda\), so
\[
  |E_i|
  \leq
  C\rho_i^{\xi+2-s}\|u\|_{C^{\xi+2}(\mathcal M)}.
\]
Since \(\rho_i\leq h\), taking the maximum over all stencils gives
\eqref{eq:exact-normal-global-gradient}--
\eqref{eq:exact-normal-global-laplacian}.
\end{proof}

We next allow the normal used to construct the tangent plane to be
approximate.  Let
\begin{equation}
  P_i=I-\bm n_i\bm n_i^T,
  \qquad
  \widetilde P_i=I-\widetilde{\bm n}_i\widetilde{\bm n}_i^T,
  \qquad
  \varepsilon_h=\max_i\|P_i-\widetilde P_i\|_2.
  \label{eq:theory-projector-error}
\end{equation}
The projector error is independent of the orientation of the normal and equals
the sine of the angle between the exact and approximate tangent planes.
In the following theorem, \(G_\alpha\), \(L\), \(\rho_i\), and the stability
factors are understood to be constructed from the approximate tangent planes.

\begin{theorem}[Convergence with approximate normals]
\label{thm:approximate-normal-tangent-rbffd}
Let the assumptions of Theorem~\ref{thm:exact-normal-tangent-rbffd} hold, but
construct the tangent planes from approximate unit normals
\(\widetilde{\bm n}_i\).  Assume \(\varepsilon_h\leq\varepsilon_0<1\) for all
sufficiently small \(h\).  Also assume that there is a constant
\(C_\Lambda\), independent of \(h\), such that the approximate-plane weights
satisfy
\begin{equation}
  \max_{i,\alpha}
  \rho_i\|w_{i,\partial_{\mathcal M,\alpha}}\|_1
  \leq C_\Lambda,
  \qquad
  \max_i
  \rho_i^2\|w_{i,\Delta_{\mathcal M}}\|_1
  \leq C_\Lambda
  \label{eq:theory-approximate-uniform-stability}
\end{equation}
for every point cloud in the refinement sequence.  Then
\begin{align}
  \max_{\alpha\in\{x,y,z\}}
  \|G_\alpha U-\partial_{\mathcal M,\alpha}u|_{X_h}\|_\infty
  &\leq C\bigl(h^{\xi+1}+\varepsilon_h\bigr)
  \|u\|_{C^{\xi+2}(\mathcal M)},
  \label{eq:approximate-normal-global-gradient}\\
  \|LU-\Delta_{\mathcal M}u|_{X_h}\|_\infty
  &\leq C\bigl(h^\xi+\varepsilon_h\bigr)
  \|u\|_{C^{\xi+2}(\mathcal M)}.
  \label{eq:approximate-normal-global-laplacian}
\end{align}
In particular, if the normal approximation converges with order \(q_g\),
meaning $\varepsilon_h=O(h^{q_g})$,
then the surface gradient and Laplace--Beltrami errors are $O\!\left(h^{\min(\xi+1,q_g)}\right)$ and $O\!\left(h^{\min(\xi,q_g)}\right)$,
respectively.
\end{theorem}

\begin{proof}
Fix a stencil centered at \(\bm x_i\).  Since the approximate tangent plane
makes an angle \(O(\varepsilon_h)\) with the exact tangent plane, for
\(\varepsilon_h\) sufficiently small the surface can be written locally as a
graph over the approximate plane,
\[
  \widetilde\Psi_i(\bm y)
  =\bm x_i+\widetilde T_i\bm y
  +\widetilde{\bm n}_i\widetilde g_i(\bm y),
\]
where \(\widetilde T_i\) is an orthonormal basis of that plane.  At the stencil
center,
\[
  \|\nabla\widetilde g_i(0)\|_2\leq C\varepsilon_h,
\]
and the derivatives of \(\widetilde g_i\) needed below remain uniformly
bounded by the smoothness of \(\mathcal M\). Set \(\widetilde u_i=u\circ\widetilde\Psi_i\).  The induced metric is
\[
  \widetilde G_i
  =I+\nabla\widetilde g_i\nabla\widetilde g_i^T,
\]
so \(\widetilde G_i(0)^{-1}=I+O(\varepsilon_h^2)\).  The coordinate formula for
the surface gradient therefore gives
\[
  \left\|
    \nabla_{\mathcal M}u(\bm x_i)
    -\widetilde T_i\nabla\widetilde u_i(0)
  \right\|_2
  \leq C\varepsilon_h\|u\|_{C^1(\mathcal M)}.
\]
Similarly, the local-coordinate formula for \(\Delta_{\mathcal M}\) differs
from the Euclidean Laplacian by metric and first-derivative terms whose
coefficients are \(O(\varepsilon_h)\).  Hence
\[
  \left|
    \Delta_{\mathcal M}u(\bm x_i)
    -\Delta\widetilde u_i(0)
  \right|
  \leq C\varepsilon_h\|u\|_{C^2(\mathcal M)}.
\]

The Taylor argument from
Theorem~\ref{thm:exact-normal-tangent-rbffd}, applied in the approximate-plane
coordinates, gives for a derivative of order \(s\in\{1,2\}\),
\[
  \left|
    \widetilde{\mathcal D}\widetilde u_i(0)
    -\sum_{j=1}^{n_s}w_{ij}\widetilde u_i(\bm y_{i_j})
  \right|
  \leq
  C\Lambda_i\rho_i^{\xi+2-s}
  \|u\|_{C^{\xi+2}(\mathcal M)},
\]
where \(\Lambda_i=\rho_i^s\|w_i\|_1\).  By
\eqref{eq:theory-approximate-uniform-stability},
\(\Lambda_i\leq C_\Lambda\).  Combining this differentiation error with the
\(O(\varepsilon_h)\) geometric error yields
\[
  |E_i|
  \leq
  C\left(\rho_i^{\xi+2-s}+\varepsilon_h\right)
  \|u\|_{C^{\xi+2}(\mathcal M)}.
\]
Since \(\rho_i\leq h\), taking the maximum over all stencils proves
\eqref{eq:approximate-normal-global-gradient}--
\eqref{eq:approximate-normal-global-laplacian}.
\end{proof}

The final result applies the preceding estimate uniformly to the evolving
surface.  Let \(X(t)\) be the current point cloud and define
\begin{equation}
  h=\sup_{0\leq t\leq T}\max_i\rho_i(t),
  \qquad
  \varepsilon_h
  =\sup_{0\leq t\leq T}\max_i
   \|P_i(t)-\widetilde P_i(t)\|_2.
  \label{eq:theory-moving-errors}
\end{equation}

\begin{theorem}[Evolving surfaces]
\label{thm:moving-surface-tangent-rbffd}
Let \(\{\mathcal M(t)\}_{0\leq t\leq T}\) be a family of compact, closed,
\(C^{\xi+2}\) surfaces of co-dimension one with a uniform Monge radius and uniformly bounded
\(C^{\xi+2}\) chart norms.  Assume
\eqref{eq:theory-scaled-stencil-bound} and
\eqref{eq:theory-stability-factors} hold uniformly in time and that
\(\varepsilon_h\leq\varepsilon_0<1\).  Then, for
\[
  \|u\|_{\xi+2,\infty;T}
  :=\sup_{0\leq t\leq T}
  \|u(\cdot,t)\|_{C^{\xi+2}(\mathcal M(t))},
\]
we have
\begin{align}
  \sup_{0\leq t\leq T}\max_{\alpha\in\{x,y,z\}}
  \|G_\alpha(t)U(t)
  -\partial_{\mathcal M(t),\alpha}u|_{X(t)}\|_\infty
  &\leq C\bigl(h^{\xi+1}+\varepsilon_h\bigr)
  \|u\|_{\xi+2,\infty;T},
  \label{eq:moving-surface-gradient-estimate}\\
  \sup_{0\leq t\leq T}
  \|L(t)U(t)-\Delta_{\mathcal M(t)}u|_{X(t)}\|_\infty
  &\leq C\bigl(h^\xi+\varepsilon_h\bigr)
  \|u\|_{\xi+2,\infty;T}.
  \label{eq:moving-surface-laplacian-estimate}
\end{align}
Hence the tangent-plane operators converge uniformly in time whenever
\(h\to0\) and \(\varepsilon_h\to0\).  If the geometric model satisfies
\begin{equation}
  \varepsilon_h\leq C_g H_M^{q_g},
  \label{eq:theory-geometry-normal-rate}
\end{equation}
then the errors are
\(O(h^{\xi+1}+H_M^{q_g})\) and
\(O(h^\xi+H_M^{q_g})\), respectively.  Thus the geometry criterion
\[
  H_M^{q_g}\lesssim c_{\rm geom}h^\xi
\]
used in Section~\ref{subsec:global-parametric-sbf-models} preserves the target
\(O(h^\xi)\) accuracy of the Laplace--Beltrami operator.  Rearrangement and
changes in stencil membership do not alter the estimate provided the uniform
geometric and stability assumptions continue to hold.
\end{theorem}

\begin{proof}
Apply Theorem~\ref{thm:approximate-normal-tangent-rbffd} at each fixed time.
The uniform chart, stencil, and stability assumptions make its constant
independent of \(t\).  Taking the supremum over \([0,T]\) gives
\eqref{eq:moving-surface-gradient-estimate}--
\eqref{eq:moving-surface-laplacian-estimate};
\eqref{eq:theory-geometry-normal-rate} gives the final statement.
\end{proof}

\begin{remark}
The estimates isolate tangent-plane error and assume the stencil nodes lie on
\(\mathcal M(t)\).  If the node positions themselves approximate the surface,
the corresponding point-location error contributes an additional geometric
term.
\end{remark}
\section{Numerical experiments}
\label{sec:numerical-experiments}

We test the moving-surface tangent-plane RBF-FD method on
advection--diffusion--reaction problems of the form
\begin{equation}
  D_t c + c\nabla_{\mathcal M(t)}\!\cdot\bm v
  =
  \nu\Delta_{\mathcal M(t)}c+\lambda c+f
  \qquad\text{on }\mathcal M(t),
  \label{eq:results-pde}
\end{equation}
where \(D_t\) denotes the material derivative along the prescribed surface
velocity \(\bm v\).  For the manufactured problems, we substitute the exact
solution into \eqref{eq:results-pde} to obtain \(f\); the artificial
hyperviscosity does not enter the manufactured forcing.  We conclude with a
non-manufactured biconcave-membrane problem driven by a three-dimensional
fluid--structure interaction simulation.

For a nodal approximation \(C_h(T)\), we report the relative final-time
\(\ell_2\) error
\begin{equation}
  E_h(T)
  =
  \frac{\|C_h(T)-c(X(T),T)\|_2}
       {\|c(X(T),T)\|_2}.
  \label{eq:results-relative-error}
\end{equation}
Unless stated otherwise, the convergence studies use
\[
  N=576,\;1024,\;1600,\;2500,\;3600,\;4900
\]
nodes and final time \(T=0.12\).  The exact manufactured solution supplies the
two startup values required by BDF3.  For each order parameter \(\xi\), we
choose \(\Delta t\) so that
\(\Delta t^3=\mathcal O(h^\xi)\), then adjust it so that \(T/\Delta t\) is an
integer.  We solve the global BDF systems using the extrapolated history as the
initial guess, a frozen incomplete LU preconditioner, at most four
defect-correction sweeps, and GMRES only when those sweeps do not reach the
requested residual tolerance.  All convergence plots use logarithmic axes
against \(\sqrt N\).  We compute the observed order as the negative
least-squares slope of \(\log E_h\) against \(\log\sqrt N\) over the stated
asymptotic range.

Appendix~\ref{app:supporting-experiments} collects supporting tests for
computed normals, mass projection, long-time integration, and marker
rearrangement.  Here we focus on convergence, computational cost, comparison
with published methods, and the fluid-driven membrane problem.

\subsection{Manufactured-solution convergence}
\label{subsec:manufactured-convergence}

The convergence suite spans radial contraction, radial expansion, rigid
ambient rotation, anisotropic deformation, and a genus-one surface.  Together
these problems test changing surface area, moving tangent frames, nonuniform
deformation, and surface topology.  Theorem~\ref{thm:exact-normal-tangent-rbffd}
gives orders \(\xi+1\) and \(\xi\) for the surface gradient and
Laplace--Beltrami operator with exact normals.
Theorem~\ref{thm:approximate-normal-tangent-rbffd} retains these orders when
the tangent-plane error is \(O(h^\xi)\) or smaller, and
Theorem~\ref{thm:moving-surface-tangent-rbffd} makes the estimates uniform in
time.  Since the Laplace--Beltrami term sets the lower spatial order in
\eqref{eq:results-pde}, \(O(h^\xi)\) is the natural benchmark for the
manufactured PDE solutions.  Our choice \(\Delta t^3=O(h^\xi)\) balances the
third-order temporal error with this spatial rate.  On a quasi-uniform
two-dimensional surface, \(h\asymp N^{-1/2}\), so \(O(h^\xi)\) appears as
slope \(-\xi\) when plotted against \(\sqrt N\).  The fitted rates below meet
or exceed this benchmark.

\subsubsection{Radial geometric flows}
\label{subsec:radial-geometric-flows}

The first two tests use spheres undergoing purely normal motion,
\[
  \mathcal M(t)=\{R(t)\bm y:\|\bm y\|=1\},
  \qquad
  \bm v(R(t)\bm y,t)=R'(t)\bm y,
\]
for which
\[
  \nabla_{\mathcal M(t)}\!\cdot\bm v=\frac{2R'(t)}{R(t)}.
\]
We prescribe the exact concentration in material coordinates as
\[
  c(R(t)\bm y,t)
  =
  e^{-t}\left[0.8+0.15y_1+0.05(y_1^2-y_2^2)\right].
\]
Because the nonconstant terms are spherical harmonics of degrees one and two,
\[
  \Delta_{\mathcal M(t)}c
  =
  \frac{e^{-t}}{R(t)^2}
  \left[-0.30y_1-0.30(y_1^2-y_2^2)\right],
\]
which gives the forcing without numerical differentiation.

For the contracting sphere we use the mean-curvature-flow (MCF) radius
\[
  R(t)=\sqrt{1.4^2-4t}.
\]
Figure~\ref{fig:mcf-sphere-convergence} shows the expected hierarchy with
\(\xi\).  Each fitted rate meets or exceeds the predicted \(O(h^\xi)\) rate for the
rate-limiting Laplace--Beltrami discretization.  Figure~\ref{fig:mcf-sphere-residual}
compares the exact and numerical concentrations and their pointwise error at
the final time.

\begin{figure}[pos=htbp]
\centering
\includegraphics[width=0.55\textwidth]{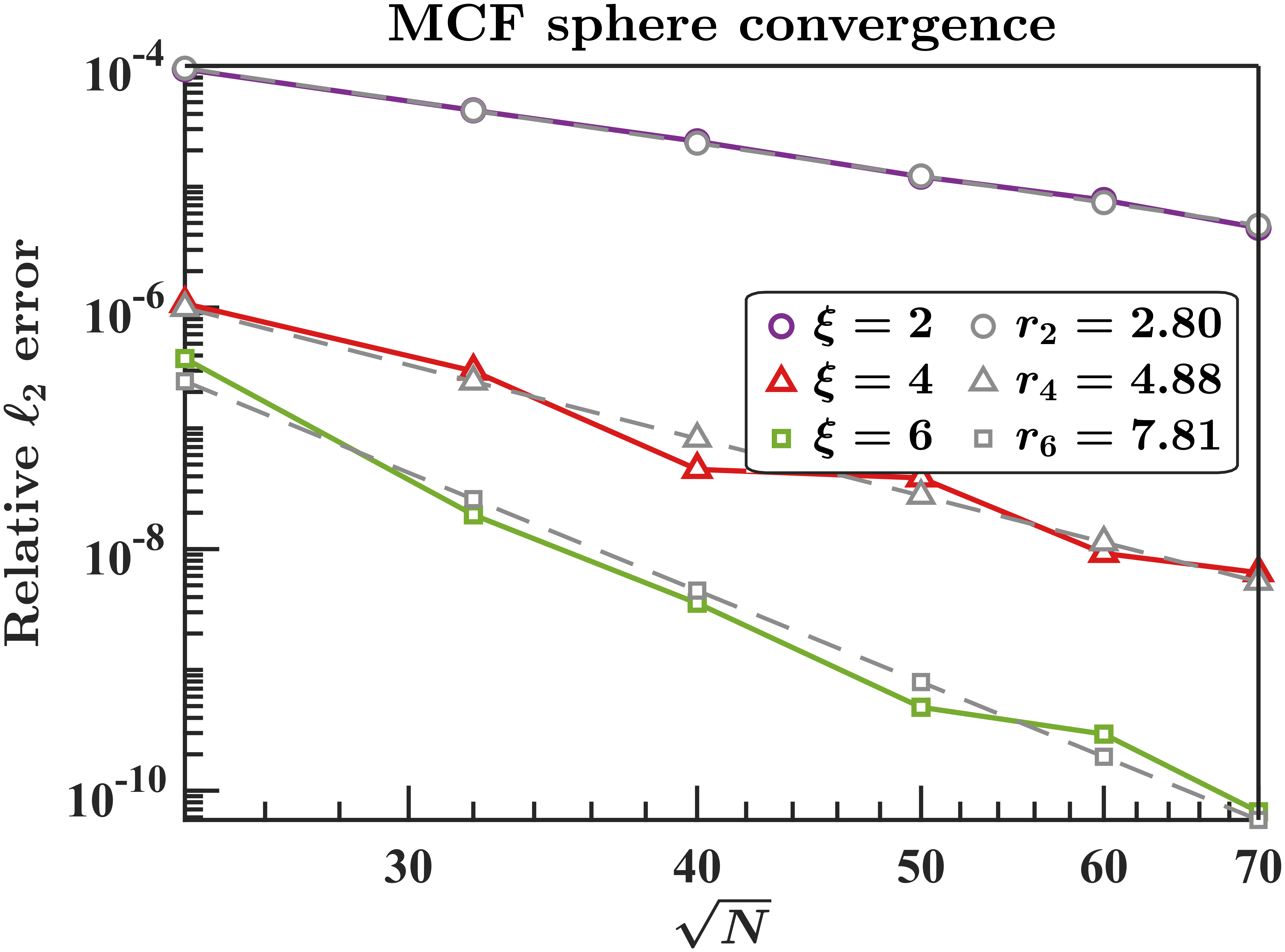}
\caption{Relative \(\ell_2\) error versus \(\sqrt N\) for the contracting
mean-curvature-flow sphere.}
\label{fig:mcf-sphere-convergence}
\end{figure}

% \begin{figure}[!htpb]
% \centering
% \includegraphics[width=0.78\textwidth]{figures/snapshot_visuals/moving_surface_adr_tp_manufactured_solution_snapshots_notitle_mcf_sphere.png}
% \caption{Evolution of the manufactured concentration on the contracting
% mean-curvature-flow sphere at \(t=0,0.06,0.12\), sampled at \(N=1600\)
% material sites.}
% \label{fig:mcf-sphere-showcase}
% \end{figure}

\begin{figure}[pos=htbp]
\centering
\includegraphics[width=0.78\textwidth]{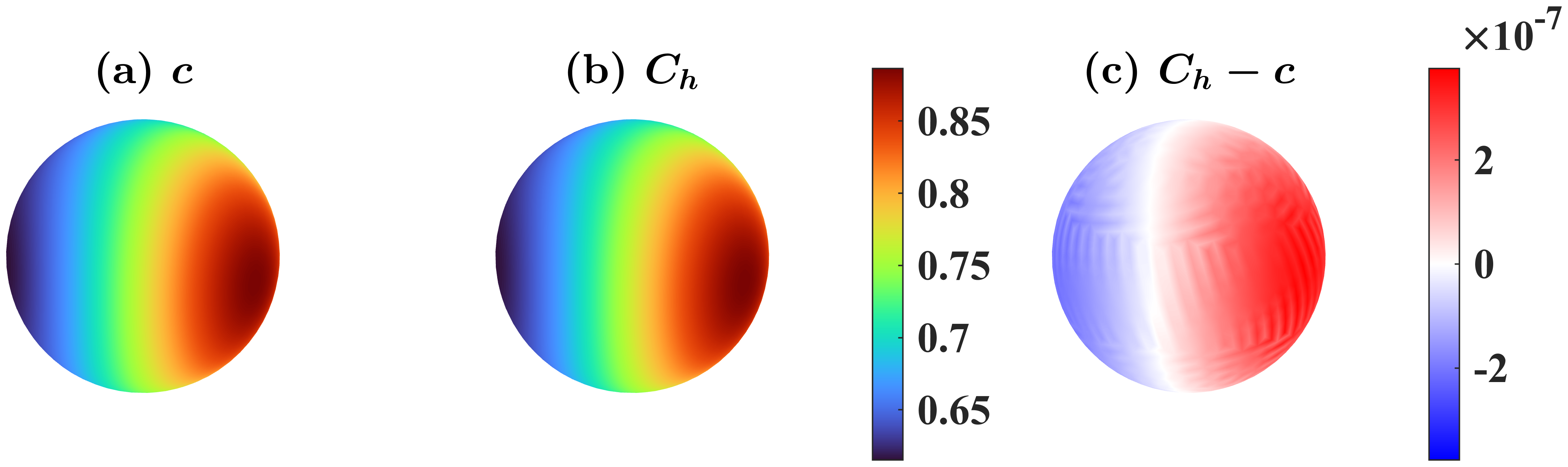}
\caption{Exact concentration, numerical concentration, and pointwise error at
\(T=0.12\) for the contracting sphere with \(\xi=6\) and \(N=1600\).}
\label{fig:mcf-sphere-residual}
\end{figure}

The expanding test uses the inverse-mean-curvature-flow (IMCF) radius
\[
  R(t)=e^{t/2}.
\]
This reverses the sign of the geometric compression while leaving the
material-coordinate solution unchanged.  Figure
\ref{fig:imcf-sphere-convergence} again meets or exceeds the same \(O(h^\xi)\) benchmark, now under surface
expansion.  Figure~\ref{fig:imcf-sphere-showcase} shows the
corresponding concentration evolution.

\begin{figure}[pos=htbp]
\centering
\includegraphics[width=0.55\textwidth]{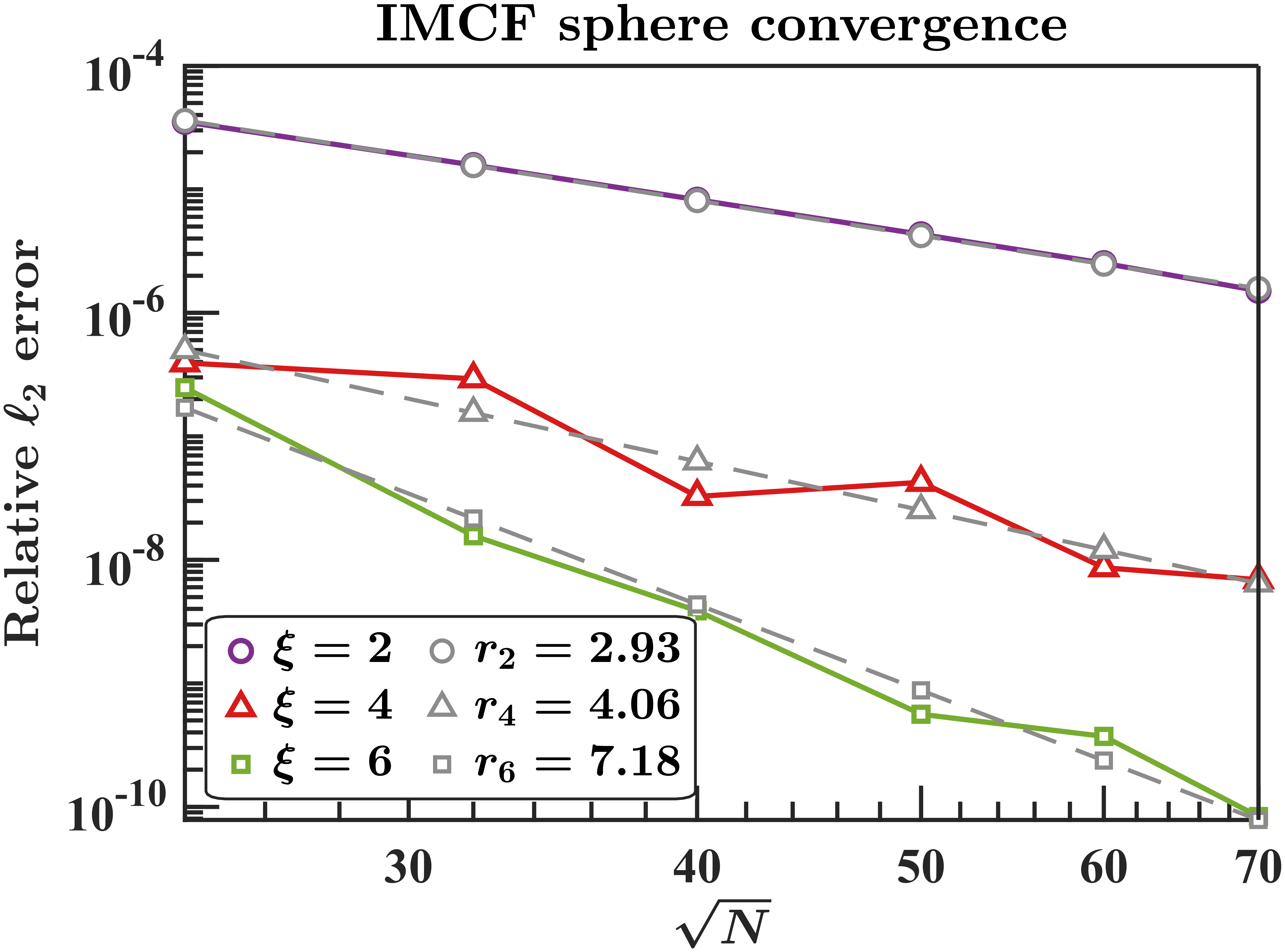}
\caption{Relative \(\ell_2\) error versus \(\sqrt N\) for the expanding
inverse-mean-curvature-flow sphere.}
\label{fig:imcf-sphere-convergence}
\end{figure}

\begin{figure}[pos=htbp]
\centering
\includegraphics[width=0.68\textwidth]{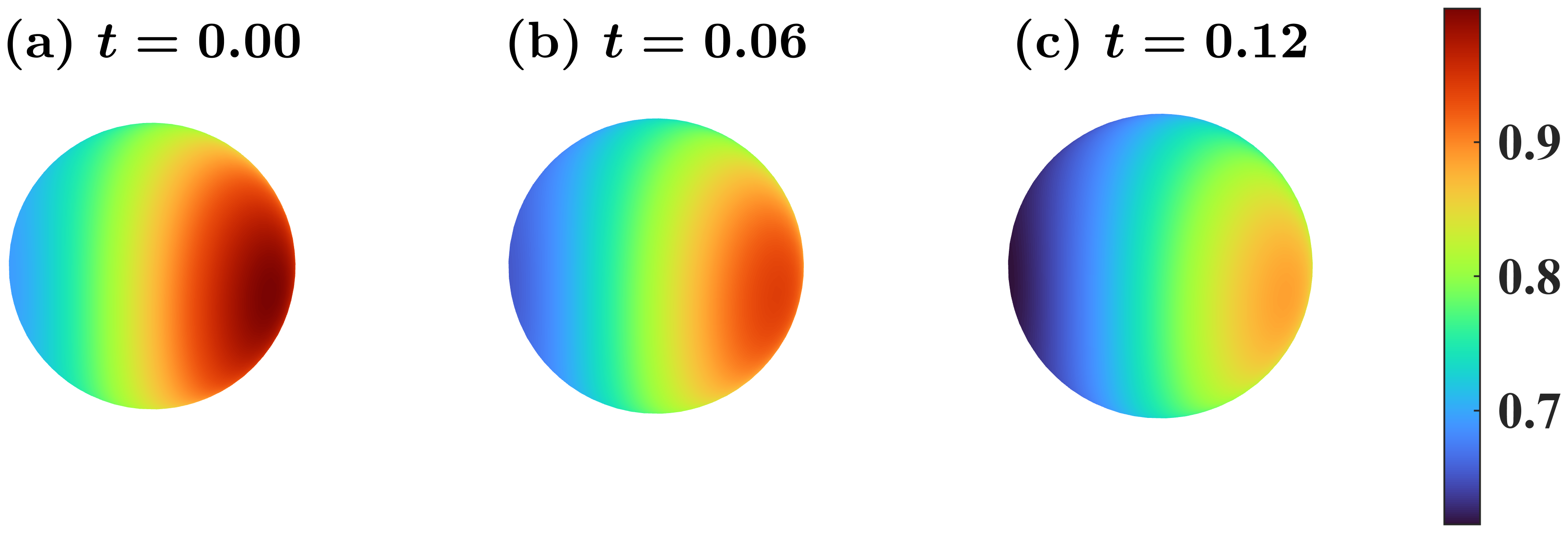}
\caption{Evolution of the manufactured concentration on the expanding
inverse-mean-curvature-flow sphere at \(t=0,0.06,0.12\), sampled at
\(N=1600\) material sites.}
\label{fig:imcf-sphere-showcase}
\end{figure}

% \begin{figure}[!htpb]
% \centering
% \includegraphics[width=0.78\textwidth]{figures/snapshot_visuals/moving_surface_adr_tp_error_residual_notitle_imcf_sphere.png}
% \caption{Exact concentration, numerical concentration, and pointwise error at
% \(T=0.12\) for the expanding sphere with \(\xi=6\) and \(N=1600\).}
% \label{fig:imcf-sphere-residual}
% \end{figure}

\subsubsection{Rotating breathing sphere}
\label{subsec:rotating-breathing-sphere}

The third test combines radial breathing with rigid ambient rotation:
\[
  \bm x(t)=R(t)Q_z(2t)\bm y,
  \qquad
  R(t)=1+0.15\sin(1.5t),
\]
where \(Q_z\) denotes rotation about the \(z\)-axis.  We use the same
concentration in material coordinates as in the radial tests.  The radial
component changes the intrinsic metric and surface divergence, while the
rotation changes the embedded coordinates and tangent frames without changing
the intrinsic geometry.  This directly tests the tangent-plane construction
under rigid motion.  Figure~\ref{fig:rotating-breathing-sphere-convergence}
shows high-order convergence, with every fitted rate meeting or exceeding
\(\xi\).  Thus rigid rotation does not degrade the predicted \(O(h^\xi)\)
moving-surface rate.

\begin{figure}[pos=htbp]
\centering
\includegraphics[width=0.55\textwidth]{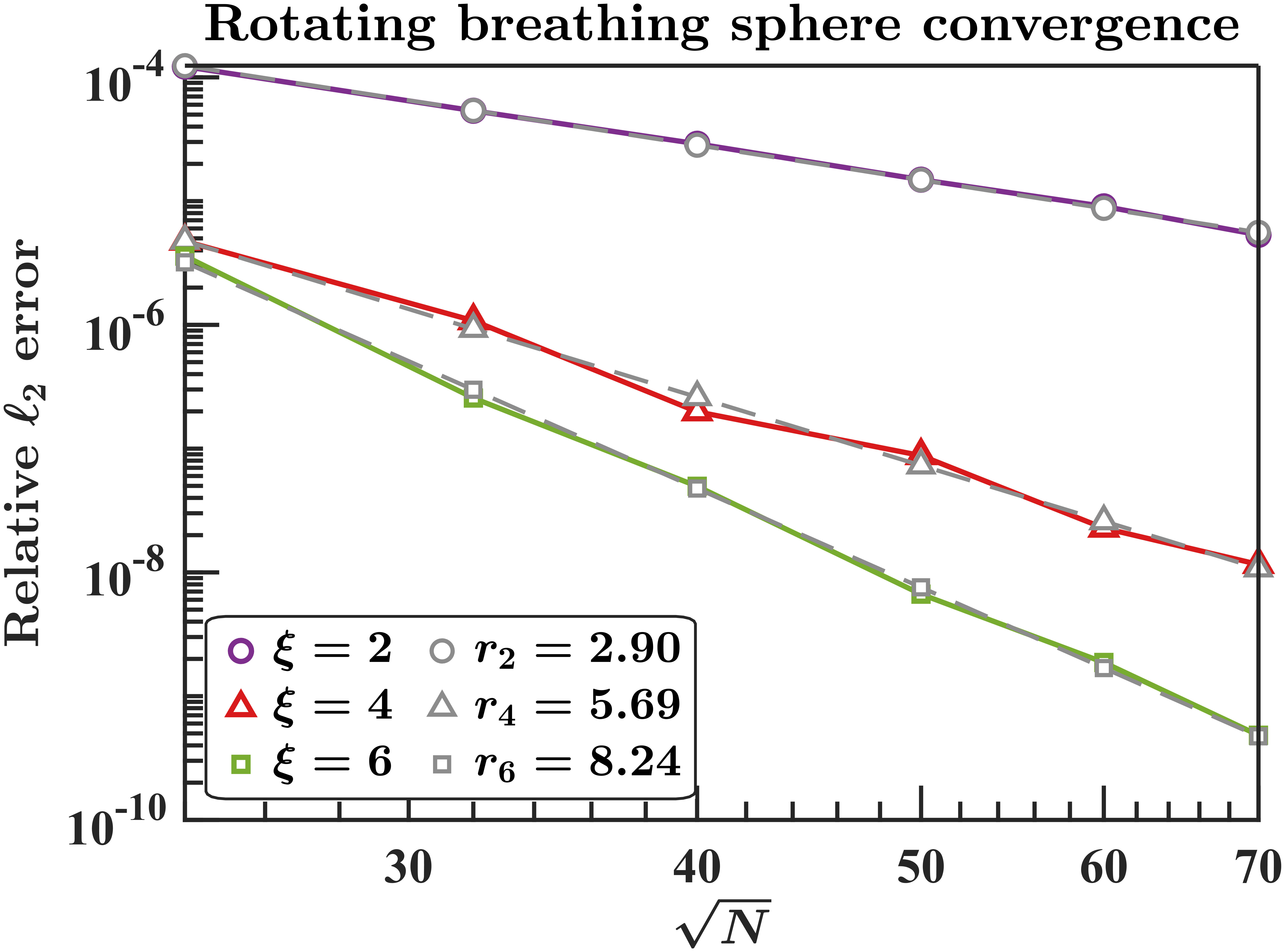}
\caption{Relative \(\ell_2\) error versus \(\sqrt N\) for the rotating
breathing sphere.}
\label{fig:rotating-breathing-sphere-convergence}
\end{figure}

\subsubsection{Anisotropically deforming ellipsoid}
\label{subsec:anisotropic-ellipsoid}

The anisotropic test maps the unit sphere according to
\[
  \bm x(t)=\left(a(t)y_1,b(t)y_2,c(t)y_3\right),
  \qquad \|\bm y\|=1,
\]
where
\[
\begin{aligned}
  a(t)&=1.30+0.08\sin t,\\
  b(t)&=0.90+0.06\cos(1.5t),\\
  c(t)&=0.70+0.05\sin(2t).
\end{aligned}
\]
We choose an exact concentration that is linear in the embedded coordinates,
\[
  c(\bm x(t),t)
  =
  e^{-t}\left(0.8+\bm\alpha\cdot\bm x(t)\right),
  \qquad
  \bm\alpha=(0.15,-0.08,0.05),
\]
and the Laplace--Beltrami term follows from
\[
  \Delta_{\mathcal M(t)}(\bm\alpha\cdot\bm x)
  =
  \bm\alpha\cdot\Delta_{\mathcal M(t)}\bm x.
\]
Unlike the radial tests, the three embedded coordinate directions evolve at
different rates.  Figure~\ref{fig:anisotropic-ellipsoid-convergence} shows
high-order convergence under this anisotropic deformation, again meeting or exceeding the predicted \(O(h^\xi)\) rate.  This is a stronger geometric
test than uniform scaling because the local metric and tangent frames evolve
nonuniformly.  It also tests the coordinate-wise hyperviscosity update.
Figure~\ref{fig:anisotropic-ellipsoid-residual} shows the final-time pointwise
error.

\begin{figure}[pos=htbp]
\centering
\includegraphics[width=0.55\textwidth]{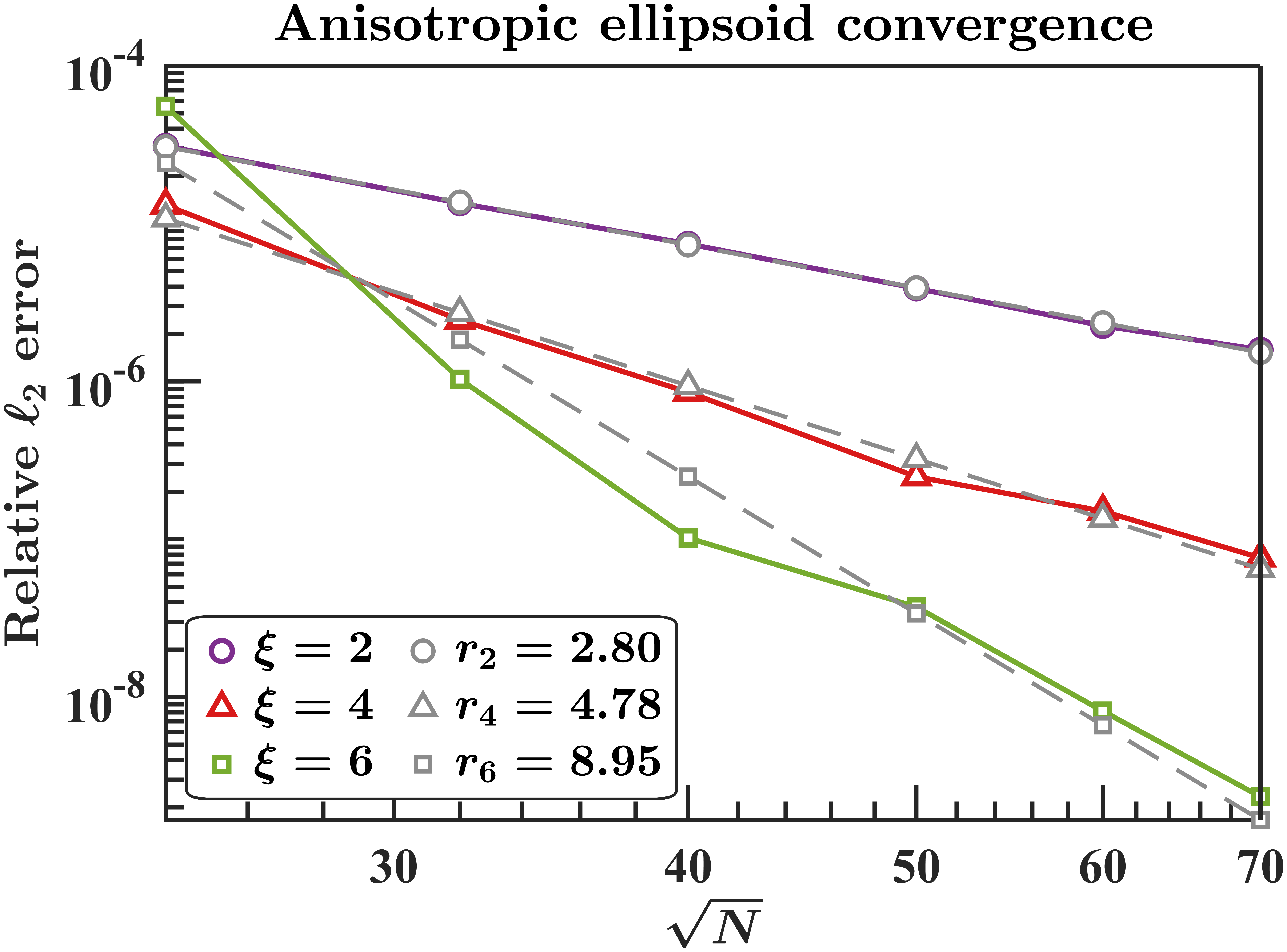}
\caption{Relative \(\ell_2\) error versus \(\sqrt N\) for the anisotropically
deforming ellipsoid.}
\label{fig:anisotropic-ellipsoid-convergence}
\end{figure}

% \begin{figure}[!htpb]
% \centering
% \includegraphics[width=0.78\textwidth]{figures/snapshot_visuals/moving_surface_adr_tp_manufactured_solution_snapshots_notitle_anisotropic_ellipsoid.png}
% \caption{Evolution of the manufactured concentration on the anisotropically
% deforming ellipsoid at \(t=0,0.06,0.12\), sampled at \(N=1600\) material
% sites.}
% \label{fig:anisotropic-ellipsoid-showcase}
% \end{figure}

\begin{figure}[pos=htbp]
\centering
\includegraphics[width=0.78\textwidth]{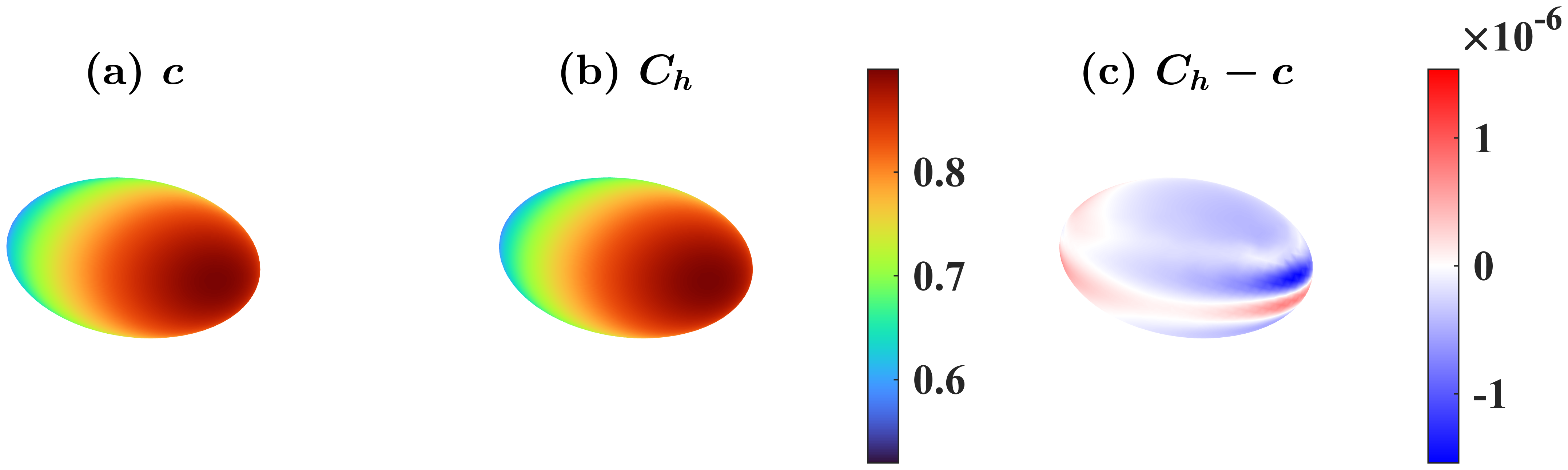}
\caption{Exact concentration, numerical concentration, and pointwise error at
\(T=0.12\) for the anisotropically deforming ellipsoid with \(\xi=6\) and
\(N=1600\).}
\label{fig:anisotropic-ellipsoid-residual}
\end{figure}

\subsubsection{Breathing torus}
\label{subsec:breathing-torus}

The final manufactured surface has nonzero genus and time-dependent major and
minor radii:
\[
  \bm x(\theta,\phi,t)
  =
  \left(
    (A(t)+B(t)\cos\theta)\cos\phi,\,
    (A(t)+B(t)\cos\theta)\sin\phi,\,
    B(t)\sin\theta
  \right),
\]
with
\[
  A(t)=1.35+0.08\sin t,
  \qquad
  B(t)=0.38+0.04\cos(1.3t).
\]
We take the exact solution to be
\[
  c(\theta,\phi,t)
  =
  e^{-t}
  \left[
    0.8+0.12\cos\theta+0.08\sin\phi
    +0.05\cos(\theta-2\phi)
  \right],
\]
and
\[
  \nabla_{\mathcal M(t)}\!\cdot\bm v
  =
  \frac{B'(t)}{B(t)}
  +
  \frac{A'(t)+B'(t)\cos\theta}
       {A(t)+B(t)\cos\theta}.
\]
The forcing uses the analytic toroidal Laplace--Beltrami operator.

The torus exposes a clear resolution threshold at high order.  We report the
\(\xi=4\) series for \(N\ge1024\) and the \(\xi=6\) series for \(N\ge1600\);
coarser clouds do not adequately resolve the corresponding stencils, so we
exclude them from the fitted rates.  Beyond this threshold,
Figure~\ref{fig:breathing-torus-convergence} recovers rates at least as large as \(\xi\), in agreement with the
moving-surface theorem.  The pre-asymptotic threshold
reflects stencil resolution rather than a loss of the asymptotic order.
Figure~\ref{fig:breathing-torus-residual} shows the final-time pointwise
error.

\begin{figure}[pos=htbp]
\centering
\includegraphics[width=0.55\textwidth]{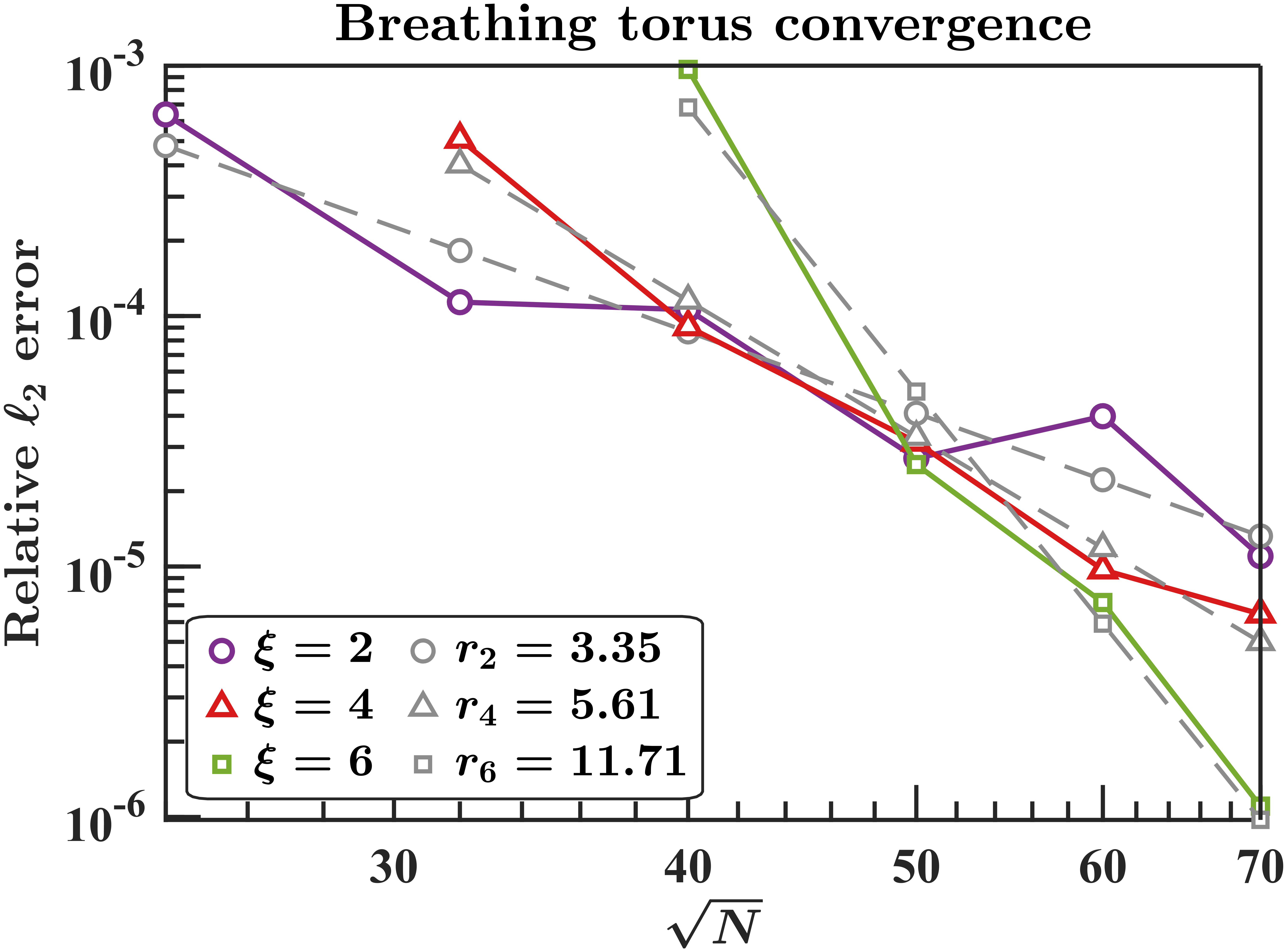}
\caption{Relative \(\ell_2\) error versus \(\sqrt N\) for the breathing
torus.  The high-order fits use only the resolved ranges stated in the text.}
\label{fig:breathing-torus-convergence}
\end{figure}

% \begin{figure}[!htpb]
% \centering
% \includegraphics[width=0.78\textwidth]{figures/snapshot_visuals/moving_surface_adr_tp_manufactured_solution_snapshots_notitle_breathing_torus.png}
% \caption{Evolution of the manufactured concentration on the breathing torus
% at \(t=0,0.06,0.12\), sampled at \(N=1600\) material sites.}
% \label{fig:breathing-torus-showcase}
% \end{figure}

\begin{figure}[pos=htbp]
\centering
\includegraphics[width=0.78\textwidth]{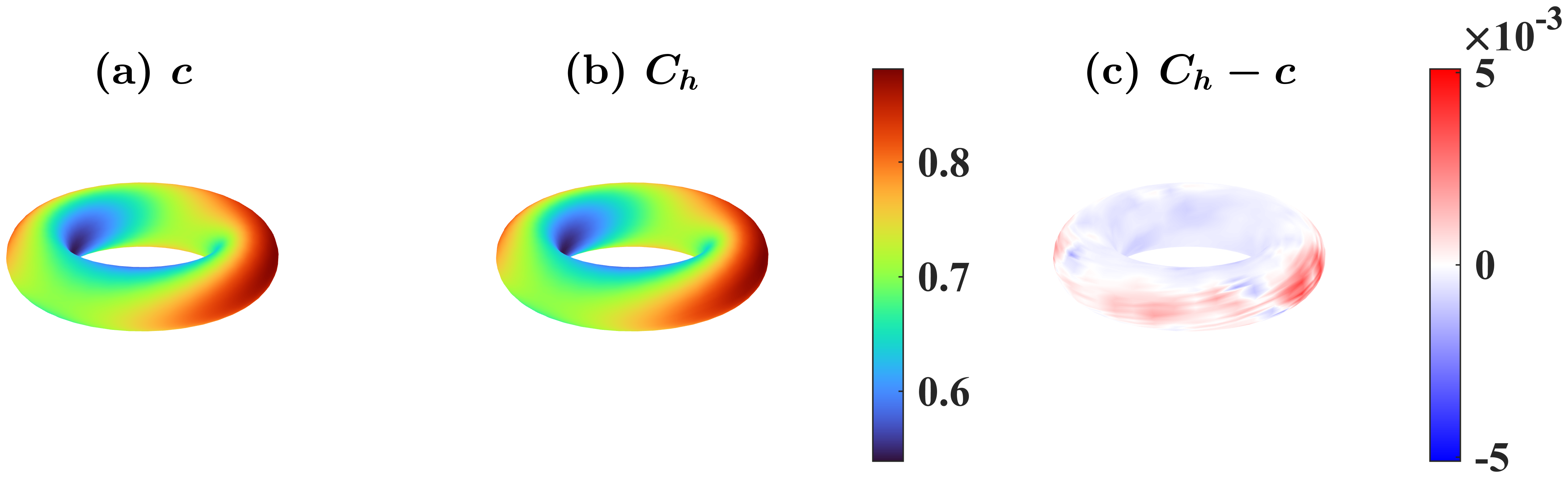}
\caption{Exact concentration, numerical concentration, and pointwise error at
\(T=0.12\) for the breathing torus with \(\xi=6\) and \(N=1600\).}
\label{fig:breathing-torus-residual}
\end{figure}

Across all five problems, the observed rates meet or exceed the
\(O(h^\xi)\) benchmark supplied by
Theorem~\ref{thm:moving-surface-tangent-rbffd}.  This agreement persists under
changing area, rigid motion, anisotropic deformation, and nonzero genus.  The
approximate-normal result,
Theorem~\ref{thm:approximate-normal-tangent-rbffd}, further shows why the
computed-normal tests in Appendix~\ref{app:supporting-experiments} retain the
same order when the tangent-plane error is \(O(h^\xi)\) or smaller.  The only
visible pre-asymptotic limitation occurs on the torus, where the high-order
stencils require a minimum spatial resolution before entering this regime.

\subsection{Cost of updating the moving-surface discretization}
\label{subsec:accelerated-update-timings}

The moving surface changes the local RBF-FD systems, the hyperviscosity
coefficients, and the global BDF system at every time level.  Recomputing all
three from scratch discards the strong temporal coherence of the discretization.
We therefore measure the savings from the three update strategies introduced in
Sections~\ref{subsec:moving-tangent-plane-stencils},
\ref{subsec:hyperviscosity-and-update}, and
\ref{sec:global-bdf-linear-algebra}.

\paragraph{Timing protocol}
We use the same five geometries, six spatial resolutions, and
$\xi=2,4,6$ for each comparison, giving 90 cases.  For a measured cost $t$, we
report the speedup
\[
  S=\frac{t_{\rm base}}{t_{\rm accel}},
\]
so $S>1$ indicates a reduction in cost.  The local RBF-FD and hyperviscosity
columns in Table~\ref{tab:combined-acceleration-speedups} use the complete
per-timestep cost because both updates affect work outside the global solve.
The global column isolates the sparse linear-solve cost.

\paragraph{Local RBF-FD weights}
Defect correction reduces the number of local factorizations by reusing the
previous stencil factors whenever the point configuration changes only mildly.
For these MATLAB calculations, however, the local update alone gives speedups
between $0.898$ and $0.990$.  The small local systems are already inexpensive,
so the cost of checking stencil membership and the defect-correction residual
is comparable to the saved factorization work.  Thus local reuse preserves the
current RBF-FD discretization at lower factorization count, but it is not by
itself a wall-clock acceleration here.

\paragraph{Hyperviscosity coefficients}
The geometric update of the hyperviscosity coefficients produces the largest
isolated savings.  Replacing repeated spectral calculations by the update in
\eqref{eq:gamma-predictor} reduces the complete timestep cost for every surface,
with speedups from $1.149$ on the breathing torus to $1.902$ on the anisotropic
ellipsoid.  This is the dominant source of the combined improvement on the four
genus-zero surfaces.

\paragraph{Global BDF solves}
The global defect-correction iteration also gives a consistent gain.  Across
the reference calculations, it removes all 2484 GMRES iterations required by
the baseline and reduces the sparse-solve cost for every geometry, with
speedups between $1.111$ and $1.874$.  The smaller effect on the complete
timestep reflects the fact that the global solve is only one part of the
moving-surface calculation.

\begin{table}[t]
\centering
\caption{Average per-timestep performance of the three acceleration
strategies and their combination.  The local and hyperviscosity columns report
full-timestep speedups, while the global column reports linear-solve speedup.
The final three columns give the baseline and accelerated full-timestep costs and their combined speedup.}
\scriptsize
\setlength{\tabcolsep}{2.5pt}
\label{tab:combined-acceleration-speedups}
\begin{tabular}{l c c c c c c}
\toprule
Geometry & local & hypervisc. & global solve &
baseline (ms) & all updates (ms) & speedup \\
\midrule
MCF sphere & \(0.941\) & \(1.846\) & \(1.874\) & \(368.3\) & \(236.3\) & \(1.559\) \\
IMCF sphere & \(0.990\) & \(1.798\) & \(1.281\) & \(355.8\) & \(215.0\) & \(1.655\) \\
Rotating breathing sphere & \(0.981\) & \(1.852\) & \(1.191\) & \(400.8\) & \(247.0\) & \(1.623\) \\
Anisotropic ellipsoid & \(0.960\) & \(1.902\) & \(1.211\) & \(446.5\) & \(262.4\) & \(1.701\) \\
Breathing torus & \(0.898\) & \(1.149\) & \(1.111\) & \(467.8\) & \(494.3\) & \(0.947\) \\
\bottomrule
\end{tabular}
\end{table}

\paragraph{Combined cost}
Across all test problems, the mean timestep cost falls from $407.8$ ms to $291.0$
ms, for an overall speedup of $1.402$.  The four genus-zero surfaces achieve
speedups between $1.559$ and $1.701$.  The breathing torus is the sole exception:
there the combined cost increases slightly, giving a speedup of $0.947$.
Figure~\ref{fig:combined-ellipsoid-timing} shows the dependence on resolution
and $\xi$ for the anisotropically deforming ellipsoid.  Taken together, these
results show that the hyperviscosity and global updates provide robust savings,
while the benefit of local defect correction depends on whether factor reuse
outweighs the cost of the local acceptance tests.

\begin{figure}[pos=htbp]
\centering
\includegraphics[width=0.62\textwidth]{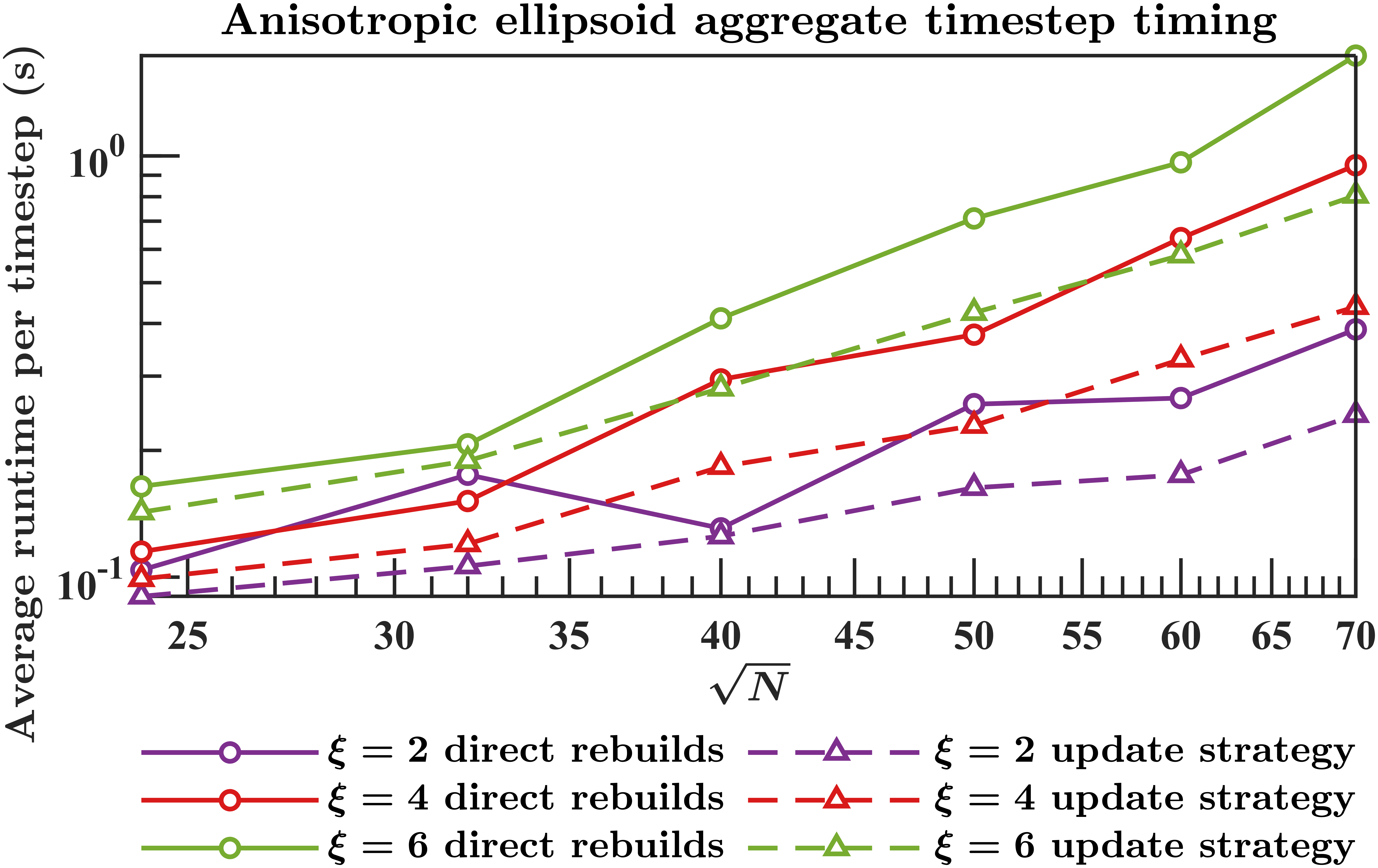}
\caption{Average runtime per timestep versus $\sqrt N$ for the
anisotropically deforming ellipsoid, with $\xi=2,4,6$ shown together.}
\label{fig:combined-ellipsoid-timing}
\end{figure}

\subsection{Comparisons with published methods}
\label{subsec:literature-comparisons}
We compare against two classes of evolving-surface methods.  The finite-element
and TraceFEM benchmarks use second-order spatial discretizations, so we ask how
coarse a \(\xi=6\) RBF-FD cloud can be while attaining the published error.
For the meshfree kernel comparisons, we match the PDE and the reported norm. Our method conserves mass on every problem.

\subsubsection{Evolving-surface FEM and TraceFEM}
\label{subsubsec:fem-literature-comparisons}

For each benchmark in Table~\ref{tab:literature-thresholds}, we increase \(N\)
until the \(\xi=6\) solution first falls below the published error.
Dziuk--Elliott \cite{DziukElliott2007EvolvingSurfaces} report absolute
\(L^\infty(0,T;L^2)\) errors, while Olshanskii--Xu
\cite{OlshanskiiXu2017TraceFEM} report absolute
\(L^2(0,T;L^2)\) errors; we retain the published norms.  None of these rigid,
scaled, or mildly deformed point clouds requires rearrangement.

\begin{table}[t]
\centering
\caption{Smallest \(\xi=6\) point clouds that attain the published
manufactured-solution errors of Dziuk--Elliott
\cite{DziukElliott2007EvolvingSurfaces} and Olshanskii--Xu
\cite{OlshanskiiXu2017TraceFEM}.}
\scriptsize
\setlength{\tabcolsep}{3pt}
\label{tab:literature-thresholds}
\begin{tabular}{l c c c c c c}
\toprule
Benchmark & \(N\) & \(h/h_{\rm lit}\) &
\(\Delta t/\Delta t_{\rm lit}\) &
lit. error & our error & our/lit. \\
\midrule
Dziuk--Elliott sphere & \(196\) & \(7.00\) & \(2.45\) &
\(1.225{\times}10^{-3}\) & \(1.054{\times}10^{-3}\) & \(0.860\) \\
Dziuk--Elliott ellipsoid & \(196\) & \(4.60\) & \(1.06\) &
\(2.295{\times}10^{-3}\) & \(1.087{\times}10^{-3}\) & \(0.474\) \\
Olshanskii--Xu translating sphere & \(256\) & \(3.54\) & \(0.157\) &
\(1.040{\times}10^{-2}\) & \(7.837{\times}10^{-3}\) & \(0.754\) \\
Olshanskii--Xu rotating sphere & \(384\) & \(2.89\) & \(0.837\) &
\(7.360{\times}10^{-3}\) & \(9.204{\times}10^{-4}\) & \(0.125\) \\
Olshanskii--Xu shrinking sphere & \(256\) & \(7.09\) & \(0.314\) &
\(3.038{\times}10^{-3}\) & \(6.735{\times}10^{-4}\) & \(0.222\) \\
\bottomrule
\end{tabular}
\end{table}

Our \(\xi=6\) discretization reaches all five published error levels on
substantially coarser point clouds.  We interpret this as an
accuracy-versus-resolution comparison rather than a wall-clock comparison,
since the methods use different approximation spaces and linear solvers.  Appendix~\ref{app:literature-controls} gives a resolution- and
timestep-matched \(\xi=2\) control together with the Olshanskii--Xu
non-manufactured mass test.

\subsubsection{Meshfree kernel methods}
\label{subsubsec:kernel-literature-comparisons}

Petras et.\ al.~\cite{PetrasLingPiretRuuth2019LSRBFCPM} report relative final-time
\(L^\infty\) errors for the oscillating ellipsoid
\[
  \bm x(\theta,\phi,t)
  =
  \left(
    \sqrt{1+\sin(2t)}\cos\theta\cos\phi,\,
    \sin\theta\cos\phi,\,
    \sin\phi
  \right),
\]
with exact solution \(c=e^{-6t}x_1x_2\), diffusion coefficient \(\nu=1\),
and \(\Delta t_{\rm lit}=2\Delta x^2\).  We choose \(N\) so that
\(h=\sqrt{|\mathcal M(0)|/N}\) matches \(\Delta x\).  For the two
\(\Delta x=0.2\) cases, we use a smaller timestep to accommodate the BDF3
startup history; all remaining cases use the published timestep.

\begin{table}[t]
\centering
\caption{Relative final-time \(L^\infty\) errors for the oscillating
ellipsoid of Petras--Ling--Piret--Ruuth
\cite{PetrasLingPiretRuuth2019LSRBFCPM}.}
\scriptsize
\setlength{\tabcolsep}{3pt}
\label{tab:petras-ling-ellipsoid}
\begin{tabular}{c c c c c c c c}
\toprule
\(\Delta x\) & \(T\) & \(N\) & \(h/\Delta x\) &
\(\Delta t/\Delta t_{\rm lit}\) & lit. error & our error & our/lit. \\
\midrule
\(0.20\) & \(0.08\) & \(314\) & \(1.000\) & \(0.333\) &
\(8.89{\times}10^{-2}\) & \(1.069{\times}10^{-4}\) & \(0.00120\) \\
\(0.20\) & \(0.16\) & \(314\) & \(1.000\) & \(0.667\) &
\(1.76{\times}10^{-1}\) & \(4.491{\times}10^{-4}\) & \(0.00255\) \\
\(0.10\) & \(0.08\) & \(1257\) & \(1.000\) & \(1.000\) &
\(2.60{\times}10^{-2}\) & \(1.050{\times}10^{-5}\) & \(0.000404\) \\
\(0.10\) & \(0.16\) & \(1257\) & \(1.000\) & \(1.000\) &
\(4.91{\times}10^{-2}\) & \(4.483{\times}10^{-5}\) & \(0.000913\) \\
\(0.05\) & \(0.08\) & \(5027\) & \(1.000\) & \(1.000\) &
\(6.80{\times}10^{-3}\) & \(3.106{\times}10^{-7}\) & \(0.0000457\) \\
\(0.05\) & \(0.16\) & \(5027\) & \(1.000\) & \(1.000\) &
\(1.26{\times}10^{-2}\) & \(7.927{\times}10^{-7}\) & \(0.0000629\) \\
\bottomrule
\end{tabular}
\end{table}

Chen and Ling~\cite{ChenLing2020KernelHeatTransport} consider
\eqref{eq:results-pde} on
\[
  \left(\frac{x_1}{1+0.25\sin t}\right)^2+x_2^2+x_3^2=1,
  \qquad 0\le t\le4,
\]
with
\[
  \bm v(\bm x,t)
  =
  \frac{0.25\cos t}{1+0.25\sin t}\,x_1\bm e_1,
  \qquad
  c(\bm x,t)=e^{-t}x_1x_2,
  \qquad \nu=1.
\]
Their Table~3 reports \(L^\infty_tL^2_x\) errors in the surface gradient
and Laplace--Beltrami operator.  We evaluate the same quantities with our
surface quadrature and discrete surface-gradient and Laplace--Beltrami
operators.  The exact solution supplies the BDF3 startup history, and we retain
the published choice \(\Delta t=h\).  The two coarsest point clouds do not
resolve a \(\xi=6\) stencil, so we compare only \(h\le1/4\).

\begin{table}[t]
\centering
\caption{Derivative-norm comparison with the moving-ellipsoid benchmark of
Chen--Ling \cite{ChenLing2020KernelHeatTransport}.  The columns report
\(L^\infty_tH^1_x\) and \(L^\infty_tH^2_x\) seminorm errors; CL denotes
Chen--Ling.}
\scriptsize
\setlength{\tabcolsep}{3pt}
\label{tab:chen-ling-ellipsoid}
\begin{tabular}{c c c c c c c c}
\toprule
\(h=\Delta t\) & \(N\) & CL \(H^1\) & our \(H^1\) &
our/CL & CL \(H^2\) & our \(H^2\) & our/CL \\
\midrule
\(0.25\) & \(312\) &
\(1.83{\times}10^{-4}\) & \(1.33{\times}10^{-3}\) & \(7.30\) &
\(4.36{\times}10^{-4}\) & \(4.13{\times}10^{-3}\) & \(9.48\) \\
\(0.125\) & \(1206\) &
\(4.01{\times}10^{-5}\) & \(5.66{\times}10^{-5}\) & \(1.41\) &
\(9.61{\times}10^{-5}\) & \(1.41{\times}10^{-4}\) & \(1.47\) \\
\(0.0625\) & \(4836\) &
\(9.77{\times}10^{-6}\) & \(8.50{\times}10^{-6}\) & \(0.870\) &
\(2.34{\times}10^{-5}\) & \(2.00{\times}10^{-5}\) & \(0.856\) \\
\bottomrule
\end{tabular}
\end{table}

At every matched spacing, the \(\xi=6\) errors in
Table~\ref{tab:petras-ling-ellipsoid} are orders of magnitude smaller than the
published Petras--Ling--Piret--Ruuth errors.  The Chen--Ling derivative norms
are more stringent: our errors are larger on the two coarser resolved clouds,
then converge to and slightly improve upon the global-kernel errors at the
finest spacing.  These seminorm comparisons directly test the differentiated
solution and complement the final-time scalar-error results above.

\subsection{Conservation of total mass}
\label{subsec:mass-conservation-results}

The same surface quadrature used to evaluate the solution error defines the
discrete total mass
\[
  M_h(t)=Q_h^t[C_h(t)].
\]
Let \(\Delta M_\star\) denote the change in total mass prescribed by the
continuous conservation law over \([0,T]\).  Thus \(\Delta M_\star=0\) for a
source-free conservative problem, while a problem with reaction or forcing
has the corresponding nonzero prescribed change.  We measure the error in
the computed total-mass evolution by
\begin{equation}
  E_{\mathrm{cons}}(T)
  =
  \frac{|M_h(T)-M_h(0)-\Delta M_\star|}
       {\max\{|M_h(0)|+|\Delta M_\star|,10^{-14}\}}.
  \label{eq:results-relative-conservation-error}
\end{equation}
This quantity directly measures whether the numerical change in total mass
equals the change required by the PDE.  Table~\ref{tab:mass-conservation}
reports the largest value over every resolved \((\xi,N)\) pair for each
convergence study.  For the literature comparisons, the entries correspond
to the configurations in Tables~\ref{tab:literature-thresholds}--
\ref{tab:chen-ling-ellipsoid}; where a comparison contains several rows, we
again report the largest value.

\begin{table}[t]
\centering
\caption{Relative error in the computed evolution of total mass.  Entries for
the five convergence studies are maxima over all reported values of \(\xi\)
and \(N\); entries for literature comparisons are maxima over the
configurations reported in the corresponding comparison tables. For the latter, the label (C) indicates a conservative method, (NC) indicates a non-conservative method (many of which nevertheless have convergent mass errors).}
\label{tab:mass-conservation}
\begin{tabular}{l c}
\toprule
Problem & Maximum \(E_{\mathrm{cons}}(T)\) \\
\midrule
Contracting mean-curvature-flow sphere & \(5.42{\times}10^{-16}\) \\
Expanding inverse-mean-curvature-flow sphere & \(1.24{\times}10^{-15}\) \\
Rotating breathing sphere & \(1.33{\times}10^{-15}\) \\
Anisotropically deforming ellipsoid & \(6.68{\times}10^{-16}\) \\
Breathing torus & \(5.25{\times}10^{-16}\) \\
\midrule
Dziuk--Elliott stationary sphere (C) & \(9.35{\times}10^{-17}\) \\
Dziuk--Elliott moving ellipsoid (C) & \(0\) \\
Olshanskii--Xu translating sphere (N) & \(1.41{\times}10^{-16}\) \\
Olshanskii--Xu rotating sphere (NC) & \(0\) \\
Olshanskii--Xu shrinking sphere (NC) & \(1.41{\times}10^{-16}\) \\
Petras--Ling--Piret--Ruuth oscillating ellipsoid (NC) & \(1.45{\times}10^{-14}\) \\
Chen--Ling moving ellipsoid (NC) & \(2.32{\times}10^{-16}\) \\
\bottomrule
\end{tabular}
\end{table}

Every convergence and literature-comparison run satisfies its prescribed
total-mass evolution to roundoff.  The separate projection study in
Appendix~\ref{subsec:mass-correction-results} examines the size of the
correction and its effect on the pointwise solution error.

\subsection{Transport on a fluid-driven biconcave membrane}
\label{subsec:rbc-capstone}
The preceding experiments prescribe the surface motion analytically.  We
conclude with a one-way coupled experiment in which a three-dimensional
fluid--structure interaction calculation first generates the motion of a
biconcave membrane, after which a passive surface concentration is evolved on
the recorded trajectory.  The concentration does not feed back into the fluid
or membrane mechanics.

Let \(\bm u_f\) and \(p\) denote the Eulerian fluid velocity and pressure, and
let \(\bm X(\bm s,t)\) map the reference membrane \(\Gamma_0\) into its current
configuration \(\Gamma(t)\).  The fluid satisfies
\begin{equation}
  \rho\left(\partial_t\bm u_f+\bm u_f\!\cdot\nabla\bm u_f\right)
  =-\nabla p+\mu_f\Delta\bm u_f+\bm g+\bm f_{\Gamma},
  \qquad \nabla\!\cdot\bm u_f=0,
  \label{eq:rbc-capstone-fluid}
\end{equation}
where \(\bm g\) is the imposed channel forcing.  Fluid and membrane are coupled
by the regularized immersed-boundary operators
\begin{align}
  \bm f_{\Gamma}(\bm x,t)
  &=\int_{\Gamma_0}\bm F(\bm s,t)
   \delta_h\!\left(\bm x-\bm X(\bm s,t)\right)\,dA_0(\bm s),\label{eq:rbc-capstone-ib-coupling1}\\  
  \partial_t\bm X(\bm s,t)
  &=\int_{\Omega}\bm u_f(\bm x,t)
   \delta_h\!\left(\bm x-\bm X(\bm s,t)\right)\,d\bm x.
  \label{eq:rbc-capstone-ib-coupling2}
\end{align}

The second relation imposes no slip.  We discretize these equations with the
hybrid finite-difference/finite-element immersed-boundary formulation of
Griffith and Luo~\cite{Peskin2002IB,GriffithLuo2017HybridIB}, using a linear
finite-element surface mesh and regularized-delta interaction quadrature.

The membrane mechanics are defined by the discrete energy
\begin{equation}
\begin{split}
  E_h(\bm X)
  ={}&\frac{\kappa_s}{2}\sum_{e}(\ell_e-\ell_e^0)^2
  +\frac{\kappa_b}{2}\sum_e \ell_e^0(\theta_e-\theta_e^0)^2 \\
  &+\frac{\kappa_A}{2}\left(\frac{A-A_0}{A_0}\right)^2
  +\frac{\kappa_V}{2}\left(\frac{V-V_0}{V_0}\right)^2.
  \label{eq:rbc-capstone-membrane-energy}
\end{split}
\end{equation}
Here \(\ell_e\) and \(\theta_e\) are the current length and dihedral angle of
edge \(e\), a superscript zero denotes the stress-free biconcave configuration,
and \(A\) and \(V\) are the total membrane area and enclosed volume.  This
spring--hinge representation follows the standard decomposition of
triangulated red-cell energies into in-plane, bending, area, and volume terms
\cite{FedosovCaswellKarniadakis2010RBC}.  If \(A_i^0\) is the reference lumped
area at node \(i\), the force density supplied to the interaction operator is
\begin{equation}
  \widehat{\bm F}_i
  =\frac{\bm F_i-\overline{\bm F}}{A_i^0},
  \qquad
  \bm F_i=-\nabla_{\bm X_i}E_h,
  \qquad
  \overline{\bm F}=\frac{1}{N_f}\sum_{j=1}^{N_f}\bm F_j.
  \label{eq:rbc-capstone-membrane-force}
\end{equation}
Here \(N_f\) is the number of membrane nodes.  Subtracting
\(\overline{\bm F}\) removes any residual resultant introduced by the numerical
evaluation of the discrete energy gradient without tethering the membrane.  The reference
surface is the Evans--Fung biconcave profile
\cite{EvansFung1972ErythrocyteGeometry}.  Appendix~\ref{app:rbc-capstone-setup}
gives the profile, all physical and numerical parameters, and the precise
IBAMR configuration.

The recorded trajectory supplies the evolving surface \(\mathcal M(t)\) and
its material velocity \(\bm v\).  On this surface we solve
\begin{equation}
  D_t c+c\nabla_{\mathcal M(t)}\!\cdot\bm v
  =\mu\Delta_{\mathcal M(t)}c,
  \qquad
  \mu=2\times10^{-3},
  \qquad 0<t\le2,
  \label{eq:rbc-capstone-pde}
\end{equation}
with neither reaction nor forcing.  To initialize two localized tracer
patches, let \(\bm u\in\mathbb S^2\) denote the fixed material coordinate of
the spherical geometry model and define
\(d_j(\bm u)=\arccos(\bm u\cdot\widehat{\bm q}_j)\), where
\begin{equation}
  \widehat{\bm q}_1
  =\frac{(0.25,0.80,0.54)}{\lVert(0.25,0.80,0.54)\rVert_2},
  \qquad
  \widehat{\bm q}_2
  =\frac{(-0.55,-0.20,0.81)}{\lVert(-0.55,-0.20,0.81)\rVert_2}.
\end{equation}
The initial concentration is
\begin{equation}
  c(\bm u,0)
  =1+\exp\!\left[-\left(\frac{d_1(\bm u)}{0.28}\right)^2\right]
  +0.65\exp\!\left[-\left(\frac{d_2(\bm u)}{0.22}\right)^2\right].
  \label{eq:rbc-capstone-initial}
\end{equation}

The trajectory is reconstructed with the spherical geometry model of
Section~\ref{subsec:global-parametric-sbf-models}, and
\eqref{eq:rbc-capstone-pde} is advanced with the method developed in
Sections~\ref{sec:methods} and~\ref{sec:global-bdf-linear-algebra}.  No exact
solution is available, so we assess the calculation through geometric,
conservation, and solver diagnostics.  Over \(0\le t\le2\), the membrane
centroid travels 0.4332 channel units and its streamwise thickness increases
from 0.1180 to 0.2349.  Its relative area and volume changes are
\(8.29\times10^{-6}\) and \(-9.75\times10^{-5}\), respectively.  The marker
quality never exceeds 1.9149, below the rearrangement threshold, and hence no
rearrangement is triggered.

The final concentration remains positive in \([0.9997,1.1057]\).  Its relative
total-mass conservation error is \(3.80\times10^{-16}\), while the largest
nodal change made by the mass projection at any step is
\(2.06\times10^{-8}\).  Thus the calculation preserves the source-free mass
law to roundoff without materially changing the transported field.  The local
defect update occurs \(99.961\%\) of the moving stencils, and every global
linear system is solved purely by global defect correction without triggering the GMRES fallback; the
largest relative linear residual is \(4.42\times10^{-11}\).

\begin{figure}[pos=htbp]
\centering
\includegraphics[width=0.78\textwidth]{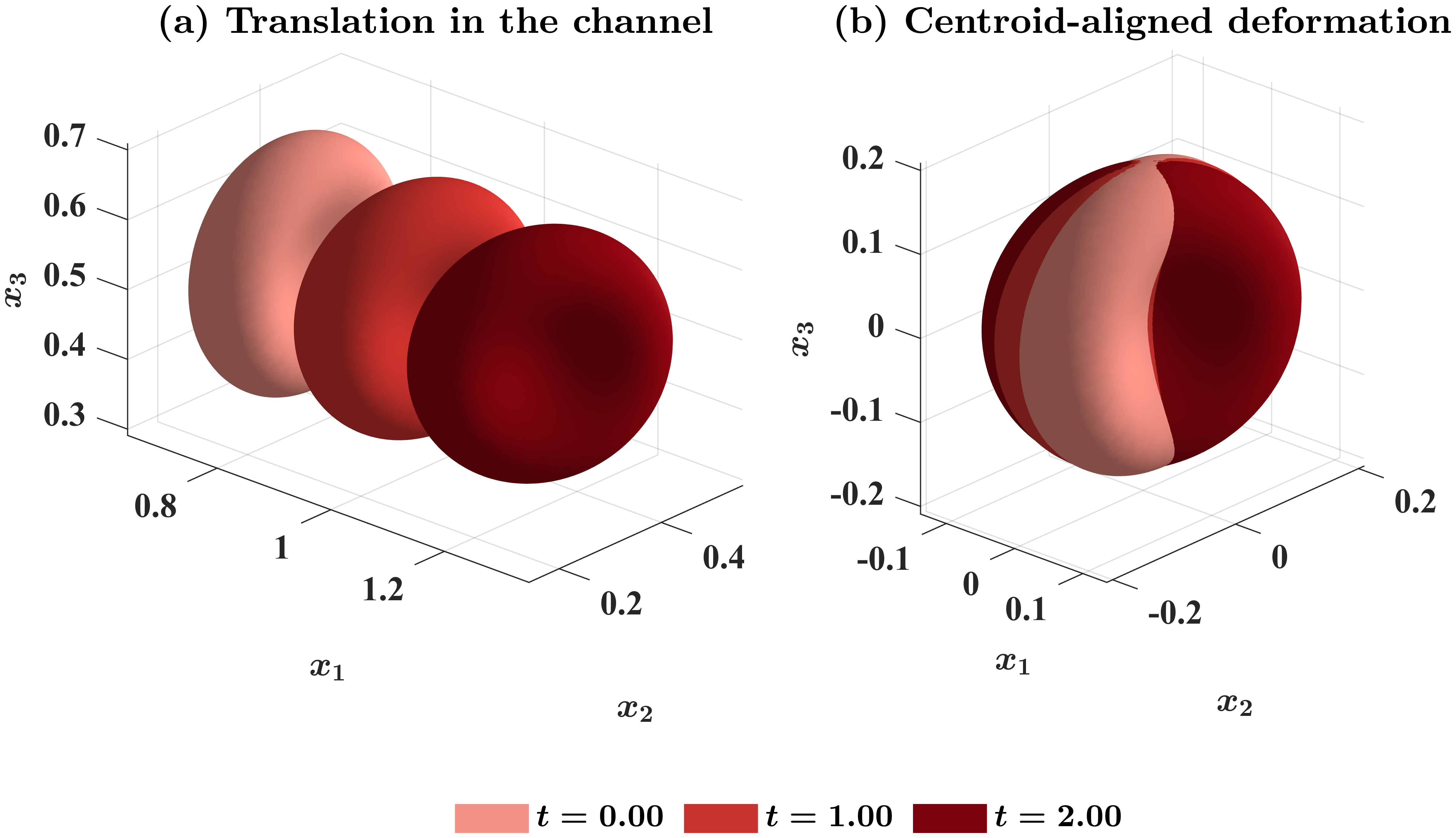}
\caption{Fluid-driven motion of the biconcave membrane at \(t=0,1,2\).
Panel (a) retains translation through the channel, while panel (b) removes the
centroid motion to isolate deformation.  The spherical geometry model is
evaluated at \(10\,000\) material sites for visualization.}
\label{fig:rbc-capstone-motion}
\end{figure}

\begin{figure}[pos=htbp]
\centering
\includegraphics[width=0.78\textwidth]{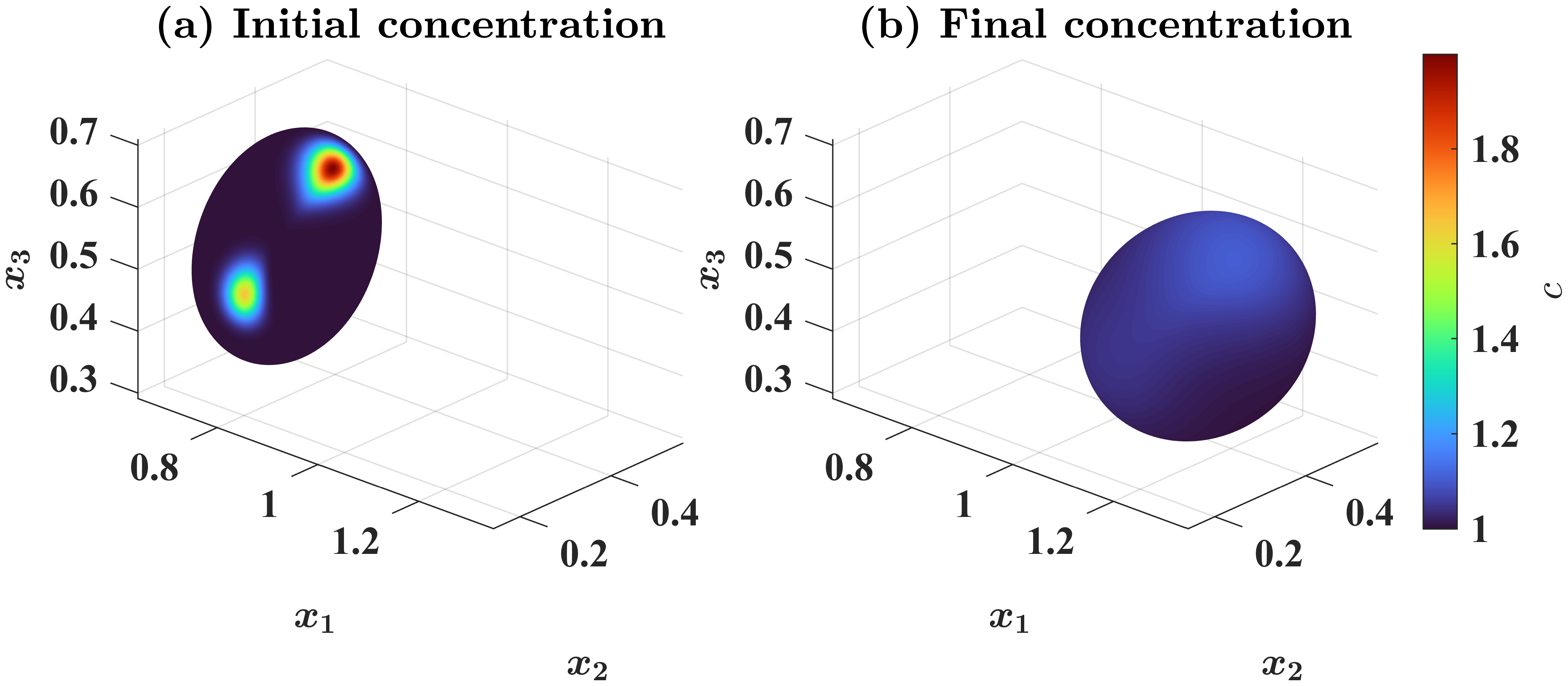}
\caption{Surface concentration at \(t=0\) and \(t=2\) for
\eqref{eq:rbc-capstone-pde}.  The panels use identical spatial limits and
color scales.  The solution at the \(N=2562\) PDE markers is interpolated in
material coordinates to the \(10\,000\)-site visualization cloud.}
\label{fig:rbc-capstone-transport}
\end{figure}

\section{Concluding remarks}
\label{sec:conclusions}
We presented a high-order accurate, globally conservative Lagrangian-Eulerian RBF-FD framework for the simulation of advection–diffusion problems on moving surfaces. Our method combined RBF-FD tangent plane discretizations of surface differential operators, a Lagrangian treatment of the material derivative, hyperviscosity stabilization of the surface divergence operator, and a parametric geometric model that played several support roles: supplying normal vectors (for RBF-FD), quadrature weights (for conservation), and surface samples (for Lagrangian rearrangement and semi-Lagrangian reconstruction). In addition, we leveraged the movement of the surface through hyperviscosity coefficient updates, local defect correction for updating the RBF-FD weights, and global defect correction as a preconditioner (but often a solver) for the sparse linear systems. We also presented theoretical results for tangent plane RBF-FD on moving domains.

Our results showed high-order convergence across five geometries involving contraction, expansion, rigid motion, anisotropic deformation, and
a surface with nonzero genus. Our results closely agreed with or exceeded our predicted theoretical convergence rates. We also showed physically plausible behavior and conservation on a more practical problem involving advection-diffusion on an elastic structure immersed in an incompressible fluid. Finally, our acceleration scheme showed speedups over direct updates, with the defect correction completely eliminating the need for GMRES.  We also matched or exceeded FEM and TraceFEM results on substantially coarser point clouds while comparing favorably with meshfree kernel results in their reported norms.

For future work, we will explore fluid flow on moving surfaces and coupled bulk-surface problems, both of which have widespread applications across engineering and the sciences.

\appendix

\section{Supporting numerical experiments}
\label{app:supporting-experiments}

This appendix collects experiments that validate the geometry, conservation,
and marker maintenance components used by the moving-surface solver.  These
components are important when a surface is supplied only through samples or
when a Lagrangian cloud becomes poorly distributed, but they are logically
separate from the convergence and performance claims in
Section~\ref{sec:numerical-experiments}.  We also record the complete numerical
configuration of the fluid-driven membrane experiment so that the external
fluid--structure trajectory and the subsequent surface transport calculation
can be reproduced independently.

\subsection{Estimated normals}
\label{subsec:normal-estimation-results}

The manufactured convergence studies use analytic normals so that the
differential discretization can be assessed independently of geometry error.
Here we replace those normals by estimates from the nodal cloud.  We compare
local PCA normals with normals obtained from the reduced parametric SBF model
of Section~\ref{subsec:global-parametric-sbf-models}.  The two test surfaces
are the anisotropically deforming ellipsoid and the breathing torus.

The PCA estimate at a node is the unit eigenvector associated with the smallest
eigenvalue of the covariance matrix of its 32 nearest neighbors.  Its sign is
oriented consistently with the analytic normal before the error is measured.
The SBF estimate is obtained by differentiating the reduced global geometry
model in material coordinates and normalizing the cross product of its two
tangent vectors.  Thus PCA uses only the instantaneous embedded cloud, whereas
the SBF construction also uses the known material parametrization.

For an order-eight geometry model, the nominal control-set size
\(M_{\rm th}\) is chosen from
\[
  H_M^8\le 0.1h^\xi,
\]
where \(H_M\) is the material fill distance of the selected centers.  The
computations use
\[
  M=\min\{N,\max\{48,\lceil M_{\rm th}/3\rceil\}\}.
\]
Thus the geometric approximation is kept below the target PDE error while
avoiding an \(N\times N\) global interpolation problem.

Let \(\bm n_i(t)\) and \(\bm n_i^\star(t)\) denote an estimated and analytic
unit normal.  For each run we record the largest-in-time RMS angle
\[
  \theta_{\rm rms}
  =
  \frac{180}{\pi}\max_n
  \left[
    \frac{1}{N}\sum_{i=1}^N
    \arccos^2\!\left(
      \min\{1,\max\{-1,\bm n_i(t_n)\!\cdot\bm n_i^\star(t_n)\}\}
    \right)
  \right]^{1/2}.
\]
The PDE-error ratio is
\(R_E=E_{\rm estimated}/E_{\rm exact}\), with all three normal choices using
the same nodes, time step, differentiation-matrix update, hyperviscosity, and
linear solver.  Consequently, \(R_E\) isolates the effect of the supplied
normal field.

\begin{table}[!htbp]
\centering
\caption{Effect of estimated normals on the PDE error.  The angle columns are
the largest values of \(\theta_{\rm rms}\), in degrees, over the tested
resolutions; the ratio columns report the largest value of \(R_E\).  The SBF
\(M\) entry gives the control-set range from the coarsest to the finest nodal
resolution.}
\scriptsize
\setlength{\tabcolsep}{3pt}
\label{tab:final-normal-source-accuracy}
\begin{tabular}{l c c c c c c}
\toprule
Geometry & \(\xi\) & SBF \(M\) & SBF angle &
SBF ratio & PCA angle & PCA ratio \\
\midrule
Anisotropic ellipsoid & \(2\) & \(48\text{--}48\) &
\(1.81{\times}10^{-3}{}^\circ\) & \(1.002\) & \(3.05^\circ\) & \(10.4\) \\
Anisotropic ellipsoid & \(4\) & \(48\text{--}48\) &
\(1.81{\times}10^{-3}{}^\circ\) & \(1.038\) & \(3.05^\circ\) & \(103.5\) \\
Anisotropic ellipsoid & \(6\) & \(59\text{--}131\) &
\(8.15{\times}10^{-4}{}^\circ\) & \(1.000\) & \(3.05^\circ\) & \(931.9\) \\
Breathing torus & \(2\) & \(48\text{--}48\) &
\(3.68{\times}10^{-2}{}^\circ\) & \(1.031\) & \(1.92^\circ\) & \(2.33\) \\
Breathing torus & \(4\) & \(98\text{--}147\) &
\(8.01{\times}10^{-4}{}^\circ\) & \(1.000\) & \(1.92^\circ\) & \(5.43\) \\
Breathing torus & \(6\) & \(356\text{--}482\) &
\(5.88{\times}10^{-6}{}^\circ\) & \(1.000\) & \(1.92^\circ\) & \(5.32\) \\
\bottomrule
\end{tabular}
\end{table}

Table~\ref{tab:final-normal-source-accuracy} shows that the PCA geometry error
is essentially independent of the requested PDE order.  It is therefore a
small perturbation for \(\xi=2\), but increasingly dominates the truncation
error as \(\xi\) increases.  On the ellipsoid, the resulting PDE error is more
than nine hundred times the exact-normal error for \(\xi=6\).  In contrast,
the SBF normal error decreases as the control set is enlarged and the SBF and
exact-normal PDE errors are indistinguishable to the reported precision on the
highest-order runs.

This accuracy does not require a full \(N\)-center geometry solve.  For
\(\xi=6\), only 59--131 SBF controls are used on the ellipsoid and 356--482 on
the torus.  The corresponding mean per-timestep costs are \(176.1\) ms with
either exact or SBF normals on the ellipsoid, and \(626.4\) ms with exact
normals versus \(641.2\) ms with SBF normals on the torus.  The latter is a
2.4\% overhead.  These timings include control selection, factorization,
coefficient recovery, and normal evaluation, and not merely the final tangent
cross products.

\begin{figure}[pos=htbp]
\centering
\includegraphics[width=0.55\textwidth]{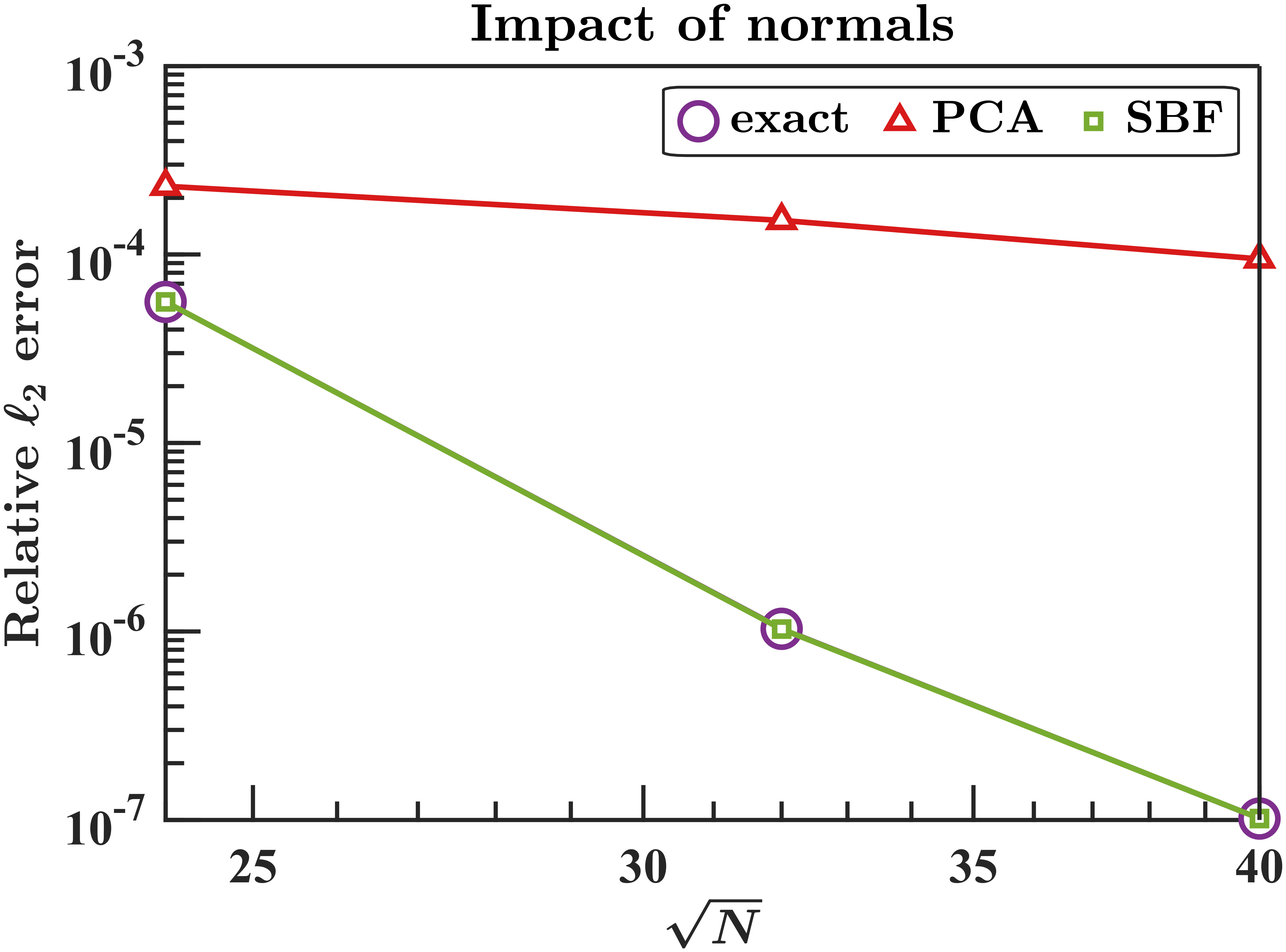}
\caption{Impact of normals on the \(\xi=6\) ellipsoid solve.  SBF normals
track the exact-normal curve, whereas PCA normal error limits the PDE
accuracy.}
\label{fig:final-normal-ellipsoid-xi6-error}
\end{figure}

\begin{figure}[pos=htbp]
\centering
\includegraphics[width=0.55\textwidth]{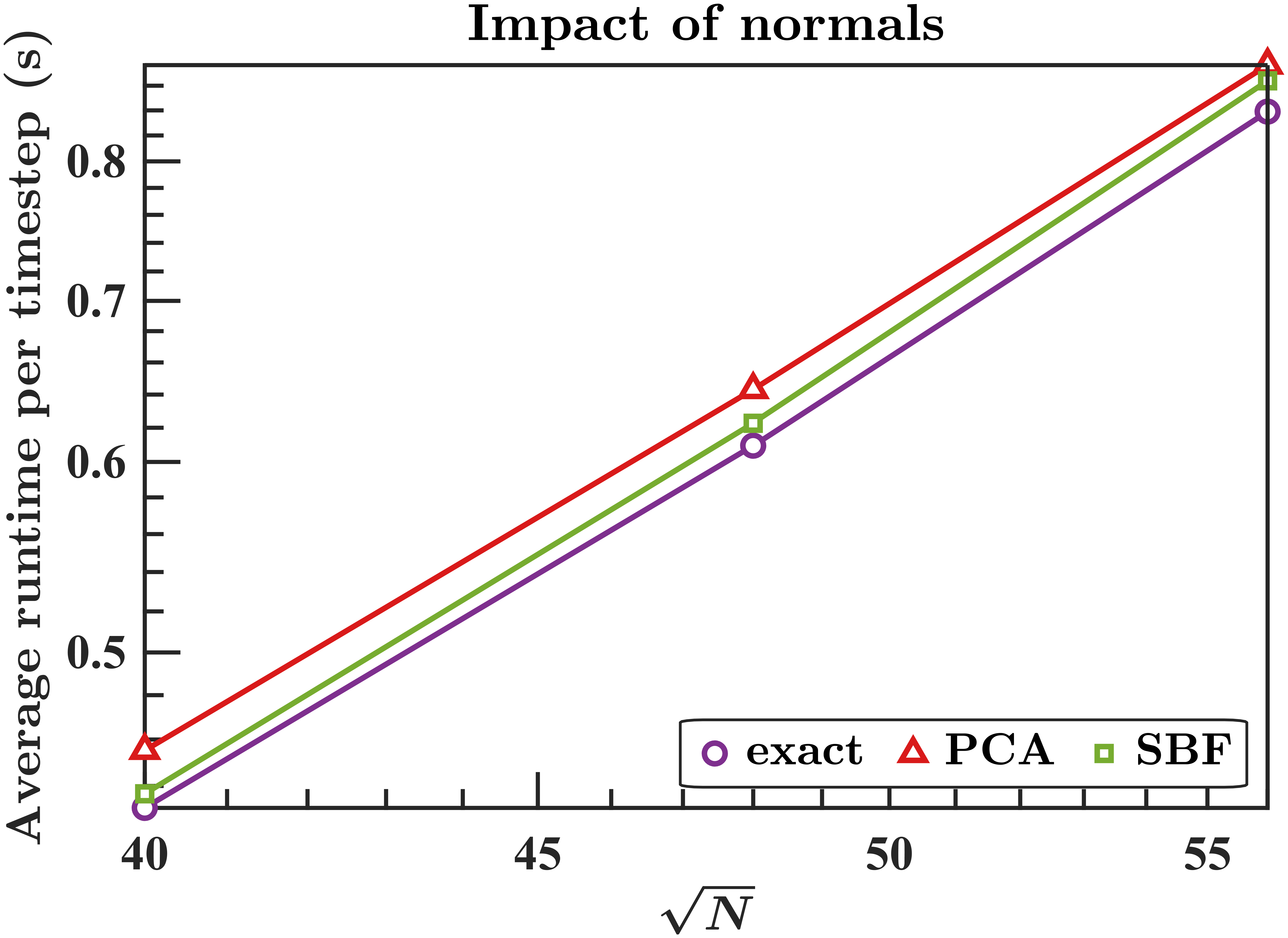}
\caption{Average runtime per timestep for the \(\xi=6\) breathing-torus
normal comparison.}
\label{fig:final-normal-torus-xi6-timing}
\end{figure}

%\clearpage
\subsection{Geometry quadrature and mass projection}
\label{subsec:mass-correction-results}

We next isolate the scalar mass projection from
Section~\ref{subsec:mass-conservation-projection}.  Projected and unprojected
runs use identical SBF normals, differentiation matrices, hyperviscosity, and
global linear algebra.  The projection target is the continuous mass law, not
the discrete integral of every term in the manufactured forcing.  In
particular,
\[
  \int_{\mathcal M(t)}\Delta_{\mathcal M(t)}c\,dS=0
\]
on a closed surface, so diffusion does not contribute to the prescribed mass
source even if its discrete quadrature is nonzero.

The ellipsoid weights are obtained by multiplying equal-area spherical
material weights by the surface Jacobian of the anisotropic linear map.  The
torus weights are obtained by scaling initial area-uniform material weights
by the time-dependent toroidal Jacobian.  The same weights define the mass
functional and the reported surface norms.

For either geometry, let
\(\widetilde M_h^n=(\bm w^n)^T\widetilde{\bm c}^{n}\) be the mass after the
unprojected BDF solve and let \(M_\star^n\) be the target obtained from the
continuous mass law.  The projection adds a spatially constant shift,
\[
  \bm c^n
  =
  \widetilde{\bm c}^{n}
  +\frac{M_\star^n-\widetilde M_h^n}
         {(\bm w^n)^T\bm 1}\bm 1.
\]
We report both the relative conservation error
\[
  \varepsilon_M^n
  =
  \frac{|M_h^n-M_\star^n|}
       {\max\{|M_\star^n|,\epsilon\}}
\]
and the relative size of the requested correction,
\[
  \eta_M^n
  =
  \frac{|M_\star^n-\widetilde M_h^n|}
       {\max\{|M_\star^n|,|\widetilde M_h^n|,\epsilon\}}.
\]
Here \(\epsilon=10^{-14}\) is used only to protect the normalization when the
target mass vanishes.
The ellipsoid sweep uses \(N=576,1024,1600\), and the torus sweep uses
\(N=1600,2304,3136\).  Each projected run is paired with an otherwise
identical unprojected calculation.

\begin{table}[!htbp]
\centering
\caption{Effect of the scalar mass projection.  The unprojected and projected
columns report the largest \(\varepsilon_M^n\) over each refinement sweep.  The
error range is the corrected relative \(\ell_2\) solution error divided by the
uncorrected error, and the final column is the largest \(\eta_M^n\).}
\scriptsize
\setlength{\tabcolsep}{3pt}
\label{tab:mass-correction-results}
\begin{tabular}{l c c c c c}
\toprule
Geometry & \(\xi\) & unprojected & projected &
error ratio & max correction \\
\midrule
Anisotropic ellipsoid & \(2\) &
\(2.57{\times}10^{-6}\) & \(3.34{\times}10^{-16}\) &
\(0.994\text{--}0.999\) & \(2.65{\times}10^{-7}\) \\
Anisotropic ellipsoid & \(4\) &
\(2.36{\times}10^{-6}\) & \(6.68{\times}10^{-16}\) &
\(0.975\text{--}0.997\) & \(5.49{\times}10^{-8}\) \\
Anisotropic ellipsoid & \(6\) &
\(7.32{\times}10^{-6}\) & \(1.67{\times}10^{-16}\) &
\(0.955\text{--}0.990\) & \(4.78{\times}10^{-8}\) \\
Breathing torus & \(2\) &
\(2.41{\times}10^{-5}\) & \(1.75{\times}10^{-16}\) &
\(0.963\text{--}1.000\) & \(2.12{\times}10^{-6}\) \\
Breathing torus & \(4\) &
\(7.22{\times}10^{-5}\) & \(5.25{\times}10^{-16}\) &
\(0.656\text{--}0.737\) & \(1.20{\times}10^{-6}\) \\
Breathing torus & \(6\) &
\(3.19{\times}10^{-4}\) & \(7.00{\times}10^{-16}\) &
\(0.811\text{--}0.926\) & \(1.28{\times}10^{-6}\) \\
\bottomrule
\end{tabular}
\end{table}

The projection reduces the discrepancy between the quadrature mass and its
target to roundoff without degrading the pointwise solution.  In the refreshed
sweep, the corrected error is no larger than the uncorrected error in any row.
The corrections remain small compared with the solution magnitude, with a
maximum relative change of \(2.12\times10^{-6}\).

The distinction between \(\varepsilon_M^n\) and \(\eta_M^n\) is important.
The first measures whether the accepted solution satisfies the prescribed mass
law; the second measures how much the unconstrained BDF update had to be
changed to enforce it.  Roundoff-level values of \(\varepsilon_M^n\) alone
would be uninformative because the projection enforces that condition by
construction.  The uniformly small values of \(\eta_M^n\), together with the
solution-error ratios at or below one, show that the projection removes only a
small accumulated scalar drift and does not conceal a spatial discretization
error.

\begin{figure}[pos=htbp]
\centering
\includegraphics[width=0.70\textwidth]{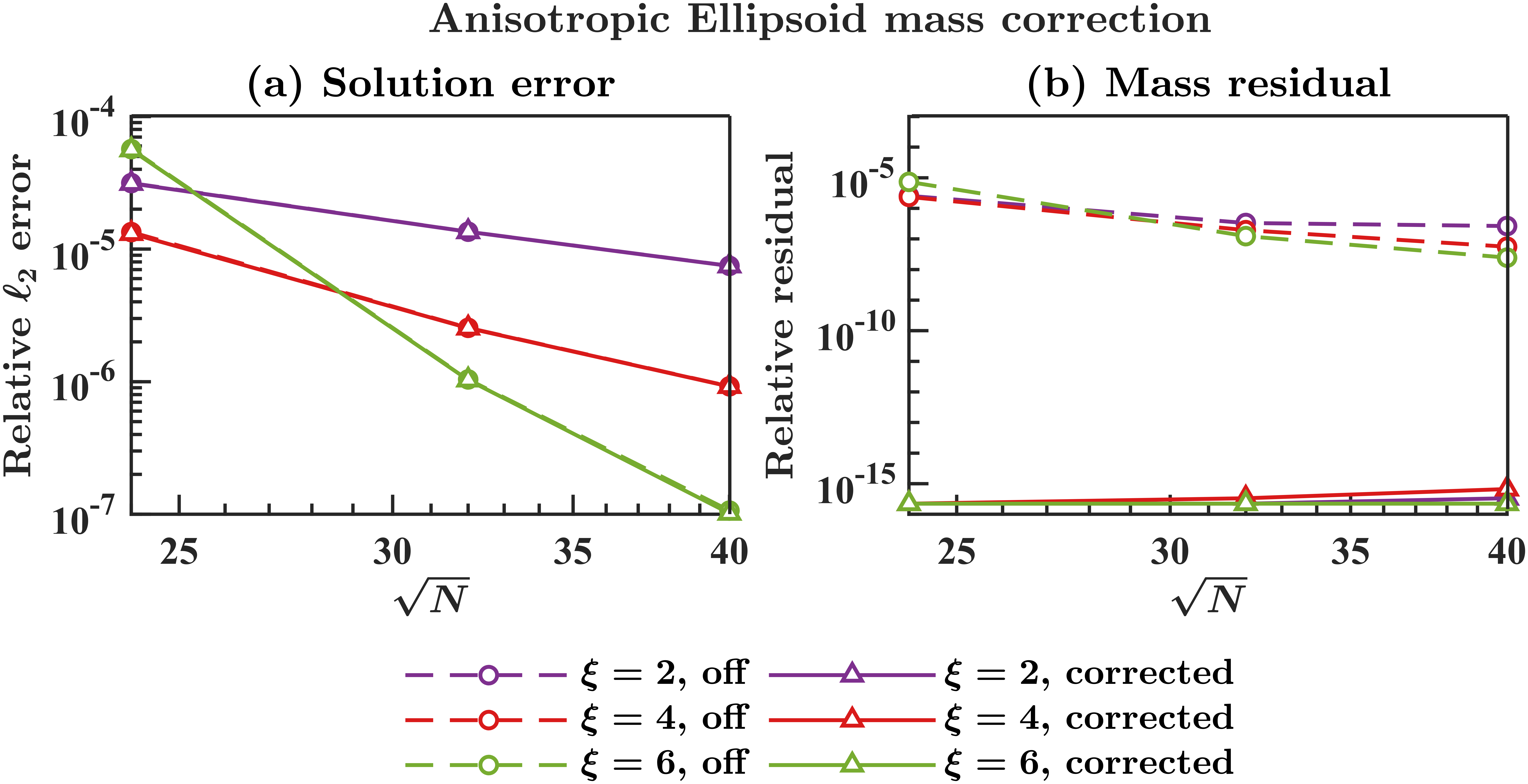}
\caption{Relative \(\ell_2\) solution error and relative mass-conservation
error with and without the scalar projection on the anisotropically deforming
ellipsoid.}
\label{fig:mass-correction-ellipsoid}
\end{figure}

\begin{figure}[pos=htbp]
\centering
\includegraphics[width=0.70\textwidth]{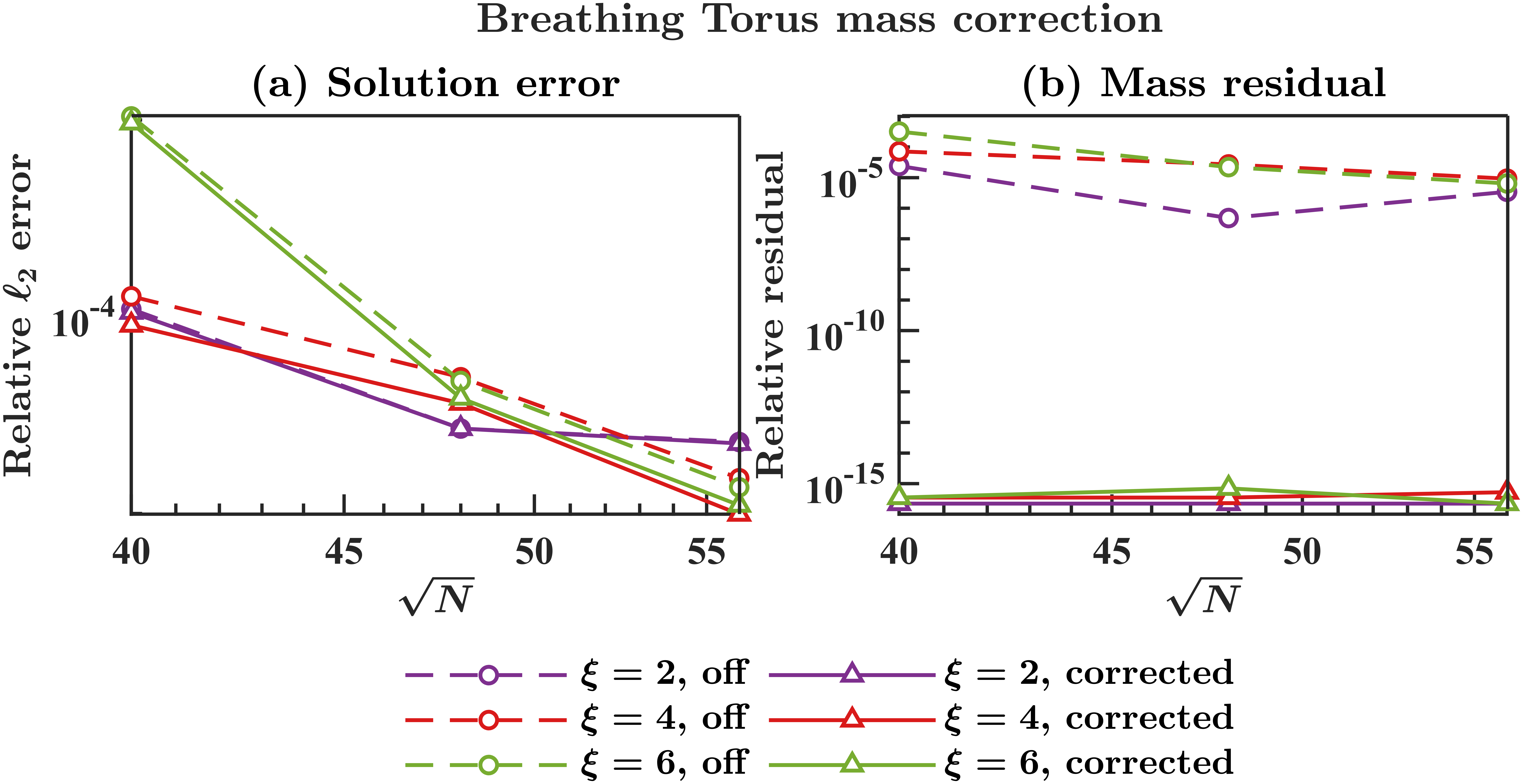}
\caption{Relative \(\ell_2\) solution error and relative mass-conservation
error with and without the scalar projection on the breathing torus.}
\label{fig:mass-correction-torus}
\end{figure}

%\clearpage
\subsection{Long-time and non-manufactured diagnostics}
\label{subsec:long-time-showcase}

The main convergence tests are intentionally short so that spatial error can
be separated cleanly from accumulated time-integration error.  To test whether
the geometry, differentiation-matrix, hyperviscosity, and global-solver updates
remain accurate over a longer sequence of surface configurations, the rotating
breathing sphere is also integrated to \(T=0.40\).  We retain the manufactured
solution, order-dependent timestep rule, and exact BDF startup used in the main
study.  If \(c_h(T)\) and \(c^\star(T)\) are the numerical and exact solutions,
the reported error is the quadrature-weighted relative norm, with
\(t_{n_T}=T\),
\[
  E_N(T)
  =
  \frac{\bigl(Q_h^{n_T}[|c_h-c^\star|^2]\bigr)^{1/2}}
       {\bigl(Q_h^{n_T}[|c^\star|^2]\bigr)^{1/2}}.
\]

\begin{table}[!htbp]
\centering
\caption{Long-time rotating breathing sphere at \(T=0.40\).  The step count
is reported at \(N=1600\).}
\scriptsize
\setlength{\tabcolsep}{3pt}
\label{tab:long-time-results}
\begin{tabular}{c c c c c c}
\toprule
\(\xi\) & \(E_{576}\) & \(E_{1024}\) & \(E_{1600}\) &
 max conservation error & steps \\
\midrule
\(2\) & \(2.73{\times}10^{-4}\) & \(2.21{\times}10^{-4}\) &
\(1.16{\times}10^{-4}\) & \(3.36{\times}10^{-16}\) & \(41\) \\
\(4\) & \(1.29{\times}10^{-5}\) & \(3.34{\times}10^{-6}\) &
\(5.07{\times}10^{-7}\) & \(1.12{\times}10^{-16}\) & \(203\) \\
\(6\) & \(9.24{\times}10^{-6}\) & \(8.19{\times}10^{-7}\) &
\(1.29{\times}10^{-7}\) & \(1.12{\times}10^{-16}\) & \(1019\) \\
\bottomrule
\end{tabular}
\end{table}

The error decreases with refinement for every order through the longest run,
which requires 1019 steps for \(\xi=6\) at \(N=1600\).  The conservation error
remains at roundoff throughout.  These observations rule out a secular loss of
accuracy from repeatedly updating the geometry and discrete operators over the
tested interval.  The two snapshot figures supplement the final-time norm by
showing the transported structure and the spatial distribution of the error.

\begin{figure}[pos=htbp]
\centering
\includegraphics[width=0.55\textwidth]{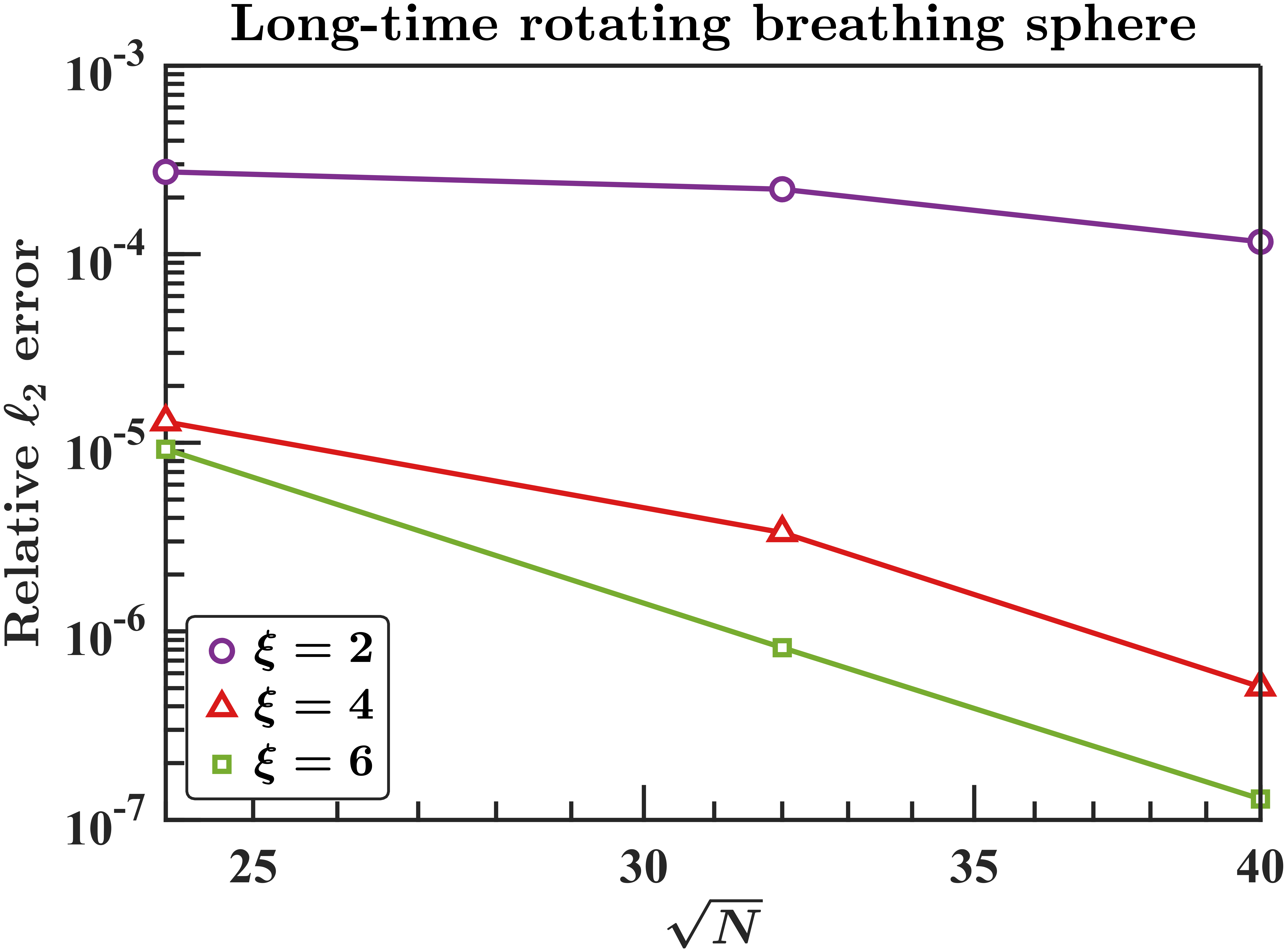}
\caption{Relative \(\ell_2\) error versus \(\sqrt N\) for the rotating
breathing sphere at \(T=0.40\).}
\label{fig:long-time-rotating-breathing-sphere}
\end{figure}

\begin{figure}[pos=htbp]
\centering
\includegraphics[width=0.78\textwidth]{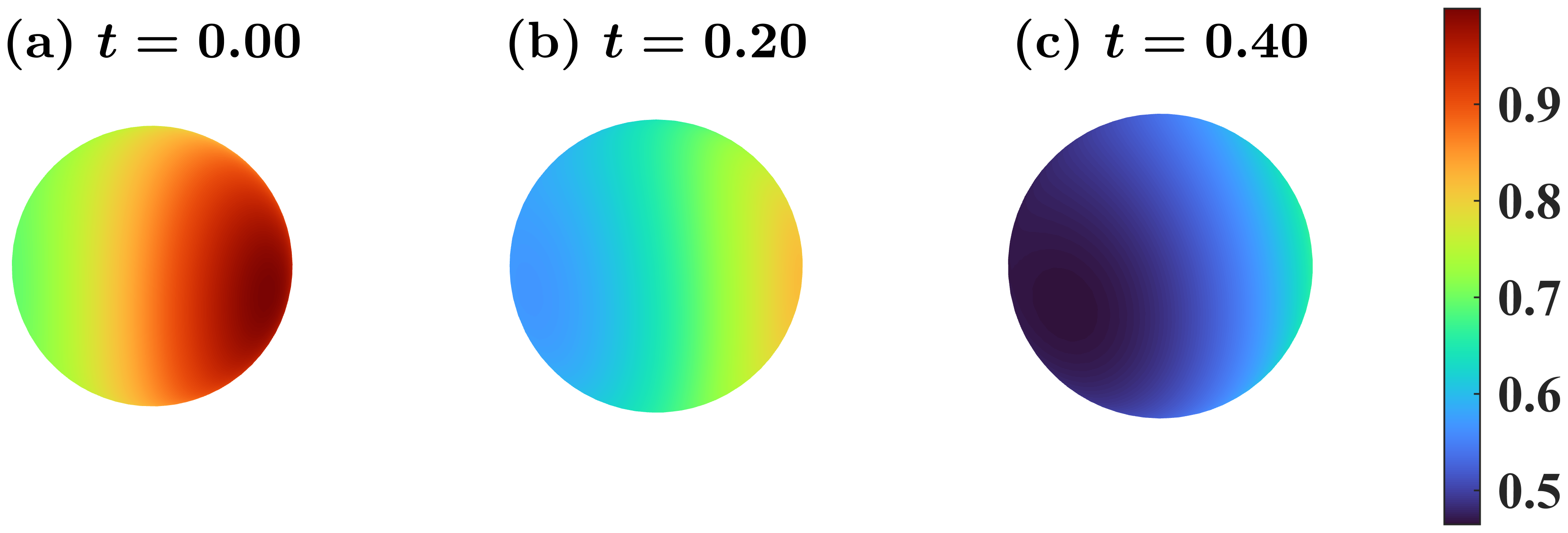}
\caption{Long-time concentration snapshots at \(t=0,0.20,0.40\) for the
rotating breathing sphere.}
\label{fig:long-time-rotating-breathing-sphere-showcase}
\end{figure}

\begin{figure}[pos=htbp]
\centering
\includegraphics[width=0.78\textwidth]{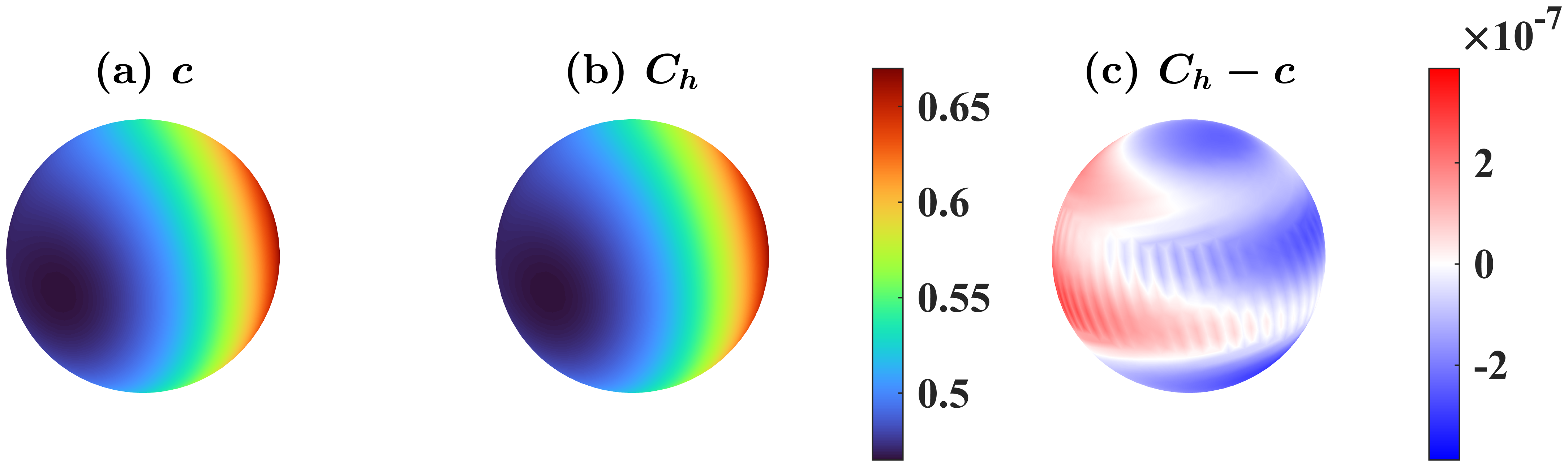}
\caption{Exact concentration, numerical concentration, and pointwise error at
\(T=0.40\) for the rotating breathing sphere with \(\xi=6\) and \(N=1600\).}
\label{fig:long-time-rotating-breathing-sphere-residual}
\end{figure}

We next remove the manufactured forcing and analytic normals.  Let
\(\bm U\in\mathbb S^2\), set \(\bm Y=R_z(0.8t)\bm U\), and let
\((\theta,\phi)\) be the spherical angles of \(\bm Y\).  The embedded surface is
\[
  \bm X(\bm U,t)=r(\theta,\phi,t)\bm Y,
\]
where
\begin{align*}
  r(\theta,\phi,t)
  ={}&1+0.10\sin(1.2t)
  +0.16\sin(1.7t+0.25)\sin^4\theta\cos(5\phi)\\
  &+0.08\cos(1.1t)\sin^3\theta\sin(3\phi).
\end{align*}
The initial concentration consists of two smooth material patches,
\[
  c(\bm U,0)
  =e^{-18(1-\bm U\cdot\bm a)}
  +0.65e^{-28(1-\bm U\cdot\bm b)},
\]
with
\(\bm a=(0.45,-0.25,0.8573214099741123)\) and
\(\bm b=(-0.55,0.60,0.5809475019311126)\).  Both vectors have unit length.
We solve source-free diffusion with \(\mu=0.02\), \(N=1600\), and \(\xi=4\).
No analytic normals are supplied to the solver; the evolving normals are
recovered from the global parametric SBF geometry model.  The quadrature
weights are formed independently from the analytic Jacobian of the radial map.

Because this problem has no exact solution, it is not used to infer a
convergence rate.  Instead, it visualizes diffusion of the two patches on a
non-spherical moving geometry and tests the source-free mass law.  The relative
mass-conservation errors are \(2.21\times10^{-16}\) at \(T=0.20\) and
\(4.42\times10^{-16}\) at \(T=0.40\).

\begin{figure}[pos=htbp]
\centering
\includegraphics[width=0.78\textwidth]{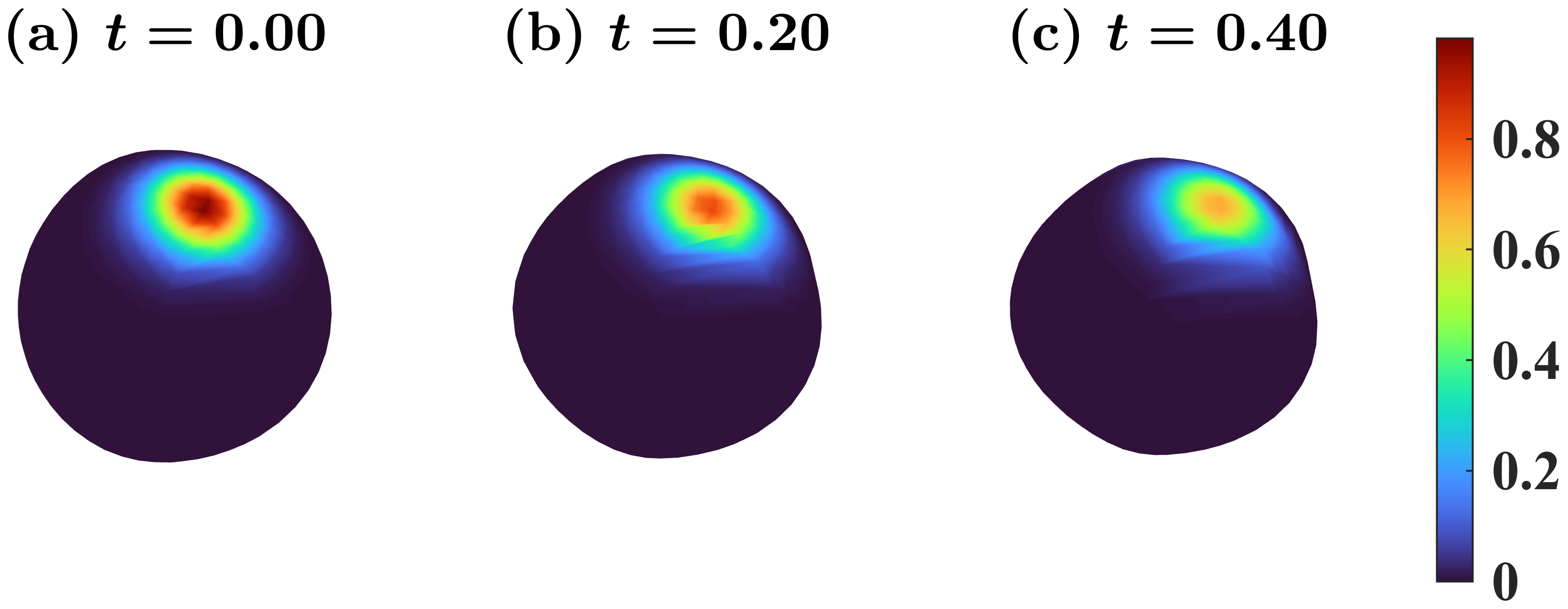}
\caption{Source-free diffusion on a rotating, breathing, bumpy sphere using
SBF normals.  The snapshots are shown at \(t=0,0.20,0.40\).}
\label{fig:bumpy-sphere-showcase}
\end{figure}

%\clearpage
\subsection{Fluid-driven membrane configuration}
\label{app:rbc-capstone-setup}

The membrane experiment in Section~\ref{subsec:rbc-capstone} is one-way
coupled.  A three-dimensional immersed-boundary/finite-element calculation is
first used to generate the membrane trajectory; the stored positions and
velocities are then treated as prescribed data by the surface transport
solver.  The coupling follows the regularized immersed-boundary formulation of
Peskin~\cite{Peskin2002IB} and the hybrid finite-difference/finite-element
surface discretization of Griffith and Luo
\cite{GriffithLuo2017HybridIB}.  All quantities reported below are
nondimensional.  The fluid and membrane equations, the regularized
Lagrangian--Eulerian coupling, and the discrete membrane energy are given in
\eqref{eq:rbc-capstone-fluid}--\eqref{eq:rbc-capstone-membrane-force}.

The reference membrane is obtained by refining an icosahedron four times and
mapping its vertices to the Evans--Fung biconcave profile
\cite{EvansFung1972ErythrocyteGeometry}.  For
\(\bm q=(q_1,q_2,q_3)\in\mathbb S^2\), let \(r^2=q_1^2+q_2^2\).  The map is
\begin{equation}
  \bm X_0(\bm q)
  =\bm x_c+R
  \begin{pmatrix}
    \frac{1}{2}\operatorname{sgn}(q_3)\sqrt{1-r^2}
    \left(c_0+c_1r^2+c_2r^4\right)\\
    q_1\\
    q_2
  \end{pmatrix},
  \label{eq:rbc-reference-profile}
\end{equation}
where \((c_0,c_1,c_2)=(0.207,2.002,-1.123)\).  This construction gives 2562
vertices and 5120 linear triangular elements.  The spherical material labels
used by the geometry model and tracer initial condition are rotated by
\(R_z(0.23)R_y(0.37)\).  This rotation changes only the labels and therefore
does not alter the physical reference surface.

The edge-stretching, area, and volume contributions to
\eqref{eq:rbc-capstone-membrane-energy} are differentiated analytically.  For
each interior edge, the four-node dihedral-energy gradient is evaluated by a
centered difference with coordinate perturbation \(10^{-6}\ell_e^0\).  The
membrane force is evaluated at the midpoint configuration.  Before conversion
to reference-area force density, the mean nodal force is removed as in
\eqref{eq:rbc-capstone-membrane-force}; consequently, numerical differentiation
of the bending energy cannot introduce a spurious net force.  The decomposition
into in-plane, hinge-bending, global-area, and global-volume energies is the
standard triangulated-membrane construction used in red-cell models
\cite{FedosovCaswellKarniadakis2010RBC}.

\begin{table}[p]
\centering
\caption{Fluid--structure configuration used to generate the biconcave
membrane trajectory.}
\label{tab:rbc-fsi-parameters}
\small
\setlength{\tabcolsep}{5pt}
\begin{tabular}{p{0.37\textwidth} p{0.49\textwidth}}
\toprule
Quantity & Value or method \\
\midrule
Fluid domain & \([0,4]\times[0,1]\times[0,1]\) \\
Boundary conditions & periodic in \(x_1,x_3\); no slip at \(x_2=0,1\) \\
Density and viscosity & \(\rho=1\), \(\mu_f=2.0\times10^{-2}\) \\
Body force & \(\bm g=4.0\times10^{-2}\bm e_1\) \\
Initial velocity & \(\bm u_f(\bm x,0)=
  [4.0\times10^{-2}/(2\mu_f)]x_2(1-x_2)\bm e_1\) \\
Base Eulerian grid & \(128\times32\times32\) staggered cells \\
Adaptive hierarchy & at most \(3\) levels; refinement ratio \(2\) \\
Fluid advection & advective form; Adams--Bashforth with PPM \\
IB time integration & explicit midpoint rule \\
Lagrangian--Eulerian coupling & four-point kernel \(\delta_h\)
  (\texttt{IB\_4}); interaction-point density \(2\) \\
Surface finite elements & linear triangles; consistent mass matrix \\
Fluid timestep and final time & \(\Delta t_f=10^{-3}\), \(T=2\) \\
Reference center and radius & \(\bm x_c=(0.75,0.35,0.50)\), \(R=0.18\) \\
Membrane mesh & \(2562\) vertices; \(5120\) faces \\
Energy coefficients & \(\kappa_s=0.5\),
  \(\kappa_b=2.5\times10^{-4}\), \(\kappa_A=5\), \(\kappa_V=20\) \\
Trajectory sampling & every \(5\) fluid steps; \(401\) stored frames \\
\bottomrule
\end{tabular}
\end{table}

At each stored time, the membrane is represented by the reduced spherical SBF
model of Section~\ref{subsec:global-parametric-sbf-models}.  The model uses a
degree-seven polyharmonic spline, spherical harmonics through degree five, and
854 farthest-point control sites.  A diagonal regularization of
\(10^{-13}\max(1,\max_{ij}|K_{ij}|)\) is added to the SBF block before its LU
factorization.  Cubic Hermite interpolation of the stored positions and
velocities supplies a continuously differentiable trajectory between frames.
Surface positions, velocity, tangents, normals, and the Jacobian-weighted
quadrature rule are all evaluated from this same model.

Table~\ref{tab:rbc-surface-parameters} lists the settings used to solve
\eqref{eq:rbc-capstone-pde}.  The time step is the final-time-adjusted value
obtained from the order-dependent restriction \(\Delta t^3\le h^\xi\).  The
rearrangement machinery remains active, although the quality indicator stays
below its trigger value throughout this trajectory.

\begin{table}[p]
\centering
\caption{Geometry and surface-transport parameters for the fluid-driven
membrane calculation.}
\label{tab:rbc-surface-parameters}
\small
\setlength{\tabcolsep}{5pt}
\begin{tabular}{p{0.37\textwidth} p{0.49\textwidth}}
\toprule
Quantity & Value or method \\
\midrule
Geometry model & degree-\(7\) spherical PHS-SBF; harmonics through degree \(5\);
  \(854\) control sites \\
Geometry regularization &
  \(10^{-13}\max(1,\max_{ij}|K_{ij}|)\) \\
Temporal geometry interpolation & cubic Hermite interpolation of positions
  and velocities \\
Surface diffusion coefficient & \(\mu=2.0\times10^{-3}\) \\
PDE resolution and order & \(N=2562\), \(\xi=6\) \\
PDE timestep and final time & \(\Delta t=9.98302885095\times10^{-5}\),
  \(T=2\), \(20{,}034\) steps \\
Stencil-size factor & \(1\) \\
Local operator update & defect tolerance \(10^{-6}\); at most \(4\) corrections \\
Neighbor update & reuse with a check every \(5\) steps \\
Hyperviscosity update & adaptive geometry predictor; selected power \(k=4\) \\
Global linear algebra & \(4\) defect sweeps followed, if needed, by GMRES with
  frozen ILU; ILU drop tolerance \(10^{-4}\) \\
Rearrangement criterion & quality threshold \(2.3\); \(3\)-step look-ahead; at least
  \(8\) steps between rearrangements \\
Rearrangement transfer & local tangent-plane interpolation \\
Conservation & geometry-quadrature mass projection after every step \\
Visualization & \(10{,}000\) material sites \\
\bottomrule
\end{tabular}
\end{table}

\subsection{Marker rearrangement and history backfill}
\label{subsec:rearrangement-results}

This experiment isolates the accuracy of Lagrangian rearrangement and
semi-Lagrangian history backfill.  The physical surface is the fixed oblate
spheroid
\[
  \bm X(\lambda,\eta)
  =
  \left(
    \sqrt{1-\eta^2}\cos\lambda,\,
    \sqrt{1-\eta^2}\sin\lambda,\,
    0.62\eta
  \right).
\]
The material labels satisfy
\[
  \dot\lambda=0.6+1.92\eta,
  \qquad
  \dot\eta=0.28(1-\eta^2),
\]
which generates a smooth tangential velocity with nonzero surface divergence
and progressively degrades the marker distribution even though the embedded
surface does not change.  This separates marker distortion from errors caused
by normal motion or changing curvature.

We solve forced advection to \(T=0.35\), with no diffusion or reaction.  The
manufactured concentration is
\[
  c(\bm x,t)
  =
  e^{-t}
  \left[
    1+0.12x+0.08yz+0.06(x^2-y^2)+0.05xyz
  \right],
\]
and the forcing is evaluated analytically as
\[
  f=D_t c+c\nabla_{\mathcal M}\!\cdot\bm v.
\]
Exact derivatives are used only to evaluate this forcing and the final error;
all spatial derivatives in the numerical update are produced by the
tangent-plane RBF-FD discretization.

At the prescribed rearrangement time \(t_r=0.175\), new markers are generated
by the area-weighted
farthest-point procedure
\eqref{eq:surface-area-density}--\eqref{eq:fps-remap}.  The BDF history at the
new markers is reconstructed by the semi-Lagrangian backfill
\eqref{eq:sl-backtrace}--\eqref{eq:history-transfer}.  This procedure repairs
a degraded Lagrangian cloud; it does not replace the Lagrangian
time-discretization of the PDE.

To quantify marker degradation, let \(d_i\) be the Euclidean distance from
marker \(i\) to its nearest distinct neighbor and define
\[
  q_X(t)=\frac{\max_i d_i(t)}{\min_i d_i(t)}.
\]
A value near one indicates uniform local separation, while a large value
signals simultaneous crowding and gap formation.  The error gain in
Table~\ref{tab:spheroid-adaptive-remap} is
\(E_{\rm Lagrangian}/E_{\rm rearranged}\).  Thus a gain greater than one means
that rearrangement improves the final solution as well as the point cloud.

\begin{table}[!htbp]
\centering
\caption{Pure Lagrangian and rearranged solutions at \(N=1152\).  The quality
columns report the maximum nearest-neighbor quality ratio over the run.}
\scriptsize
\setlength{\tabcolsep}{3pt}
\label{tab:spheroid-adaptive-remap}
\begin{tabular}{c c c c c c}
\toprule
\(\xi\) & Lagr. error & remap error & gain &
Lagr. quality & remap quality \\
\midrule
\(2\) & \(7.65{\times}10^{-5}\) & \(6.41{\times}10^{-5}\) & \(1.19\) & \(2.68\) & \(1.98\) \\
\(4\) & \(2.04{\times}10^{-5}\) & \(1.76{\times}10^{-5}\) & \(1.16\) & \(2.68\) & \(1.98\) \\
\(6\) & \(1.98{\times}10^{-5}\) & \(1.80{\times}10^{-5}\) & \(1.10\) & \(2.68\) & \(1.98\) \\
\bottomrule
\end{tabular}
\end{table}

Across \(N=512,768,1024,1152\), the rearranged solution retains the
convergence trend of the pure Lagrangian calculation while keeping the quality
ratio near two.  At the finest tested resolution it also reduces the final
error for all three values of \(\xi\).

\begin{figure}[pos=htbp]
\centering
\includegraphics[width=\textwidth]{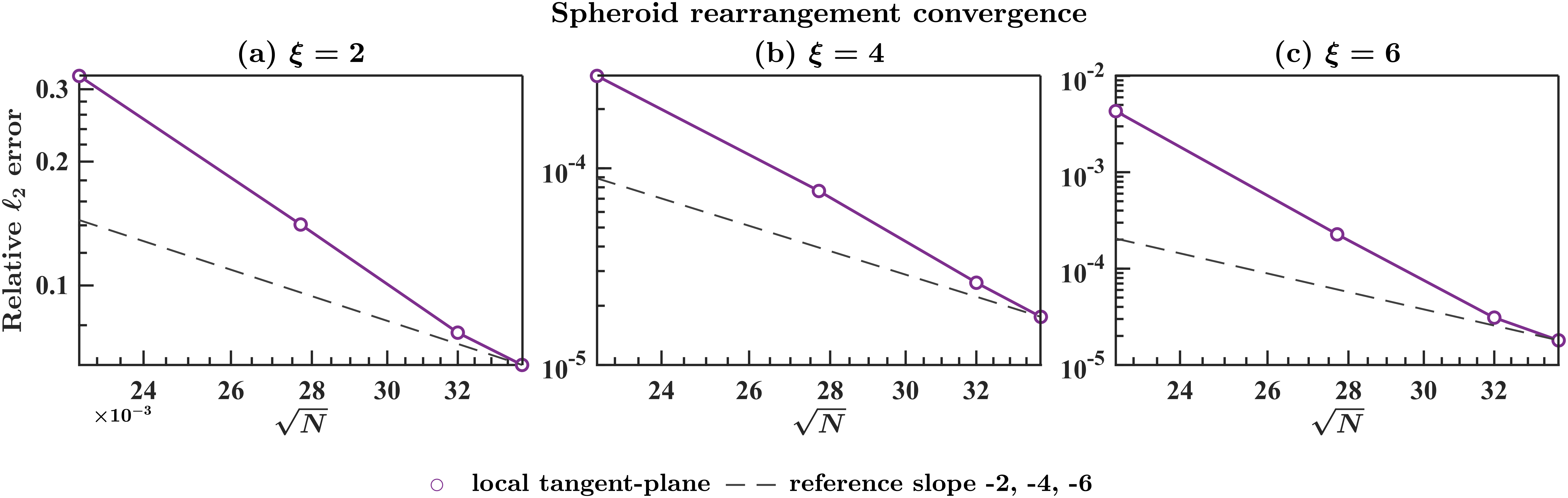}
\caption{Convergence with controlled Lagrangian rearrangement and local
tangent-plane transfer on the spheroid.}
\label{fig:spheroid-rearrangement-convergence}
\end{figure}

The comparison is deliberately made over a complete time integration rather
than a single interpolation step.  After rearrangement, the reconstructed BDF
history is used by every subsequent step, so the final error includes the
transfer error, backtrace error, and any downstream error propagated by the
multistep scheme.  The results therefore test the actual restart mechanism
used by the moving-surface solver.

To isolate transfer error, an additional study compares exact backfill, local
tangent-plane interpolation, and global SBF interpolation.  Exact backfill is
a diagnostic lower bound rather than a numerical transfer method.

\begin{table}[!htbp]
\centering
\caption{Controlled-remap transfer comparison.  Slopes are least-squares fits
of final-time relative error versus \(\sqrt N\) over the transfer study.}
\scriptsize
\setlength{\tabcolsep}{3pt}
\label{tab:spheroid-transfer-rates}
\begin{tabular}{l c c c}
\toprule
Backfill / transfer & \(\xi=2\) & \(\xi=4\) & \(\xi=6\) \\
\midrule
exact backfill & \(-3.52\) & \(-6.09\) & \(-10.41\) \\
local tangent-plane & \(-3.55\) & \(-6.20\) & \(-10.23\) \\
global SBF & \(-3.54\) & \(-6.18\) & \(-10.23\) \\
\bottomrule
\end{tabular}
\end{table}

Both numerical transfers preserve the convergence behavior of the
exact-backfill calculation.  The local interpolant is preferable in this setting
because it uses the tangent-plane basis already available on each stencil and
avoids a separate global transfer solve.

The fitted slopes in Table~\ref{tab:spheroid-transfer-rates} are not used as
standalone estimates of the formal order; their role is comparative.  At each
\(\xi\), the local tangent-plane and global SBF slopes agree with the
exact-backfill slope, showing that neither numerical transfer changes the
observed refinement trend.  Since the local transfer also avoids constructing
and solving an additional global interpolation system, it is used in the
remaining experiments.

\subsection{Additional literature controls}
\label{app:literature-controls}

Table~\ref{tab:literature-matched-xi2} removes the high-order/coarse-resolution
advantage from Table~\ref{tab:literature-thresholds}.  The RBF-FD point count
is chosen as
\[
  N=\operatorname{round}\!\left(
    \frac{|\mathcal M(0)|}{h_{\rm lit}^2}
  \right),
\]
so that the nominal spacing
\(h=\sqrt{|\mathcal M(0)|/N}\) matches the published spacing
\(h_{\rm lit}\).  The published timestep is retained up to the rounding needed
to reach the same final time with an integer number of steps.  We then use
\(\xi=2\), SBF normals, exact BDF startup, and the same stabilization,
rearrangement, mass-projection, and linear-algebra procedures as in the other
experiments.

The comparison metric is taken from the corresponding paper rather than
replaced by a common endpoint norm.  The Dziuk--Elliott entries use the absolute
\(L^\infty(0,T;L^2(\mathcal M(t)))\) error, while the Olshanskii--Xu entries use
the absolute \(L^2(0,T;L^2(\mathcal M(t)))\) error.  Consequently, ratios are
meaningful within each row, but the numerical values should not be compared
across the two groups.

\begin{table}[!htbp]
\centering
\caption{Resolution- and timestep-matched \(\xi=2\) control against the
Dziuk--Elliott \cite{DziukElliott2007EvolvingSurfaces} and Olshanskii--Xu
\cite{OlshanskiiXu2017TraceFEM} manufactured tests.}
\scriptsize
\setlength{\tabcolsep}{3pt}
\label{tab:literature-matched-xi2}
\begin{tabular}{l c c c c}
\toprule
Benchmark & \(N\) & lit. error & our error & our/lit. \\
\midrule
Dziuk--Elliott sphere & \(9610\) &
\(1.225{\times}10^{-3}\) & \(2.391{\times}10^{-6}\) & \(0.00195\) \\
Dziuk--Elliott ellipsoid & \(4153\) &
\(2.295{\times}10^{-3}\) & \(4.173{\times}10^{-6}\) & \(0.00182\) \\
Olshanskii--Xu translating sphere & \(3217\) &
\(1.040{\times}10^{-2}\) & \(2.786{\times}10^{-4}\) & \(0.0268\) \\
Olshanskii--Xu rotating sphere & \(3217\) &
\(7.360{\times}10^{-3}\) & \(4.082{\times}10^{-4}\) & \(0.0555\) \\
Olshanskii--Xu shrinking sphere & \(12868\) &
\(3.038{\times}10^{-3}\) & \(5.711{\times}10^{-5}\) & \(0.0188\) \\
\bottomrule
\end{tabular}
\end{table}

The matched control remains below the published error in every case.  The
largest ratio is 0.0555 for the rotating-sphere test, and hence even the least
favorable matched run has an error about eighteen times smaller than the
published value.  This is a resolution-and-timestep control, not a claim of
equal computational cost: the trial spaces, quadrature rules, nonlinear
geometry work, and software implementations differ.  Its purpose is to show
that the high-order/coarse-resolution comparisons in the main text are not
created solely by evaluating the RBF-FD method at a different nominal spacing.

Olshanskii--Xu also report a non-manufactured conservation diagnostic on the
reference surface
\[
  (x_1-x_3^2)^2+x_2^2+x_3^2=1
\]
under the ambient velocity
\[
  \bm v(\bm x,t)
  =
  \bigl(0.1x_1\cos t,\,0.2x_2\sin t,\,0.2x_3\cos t\bigr).
\]
Equivalently, if \(\bm X_0\) lies on the reference surface, then
\[
  \bm X(t)
  =
  \operatorname{diag}\!\left(
    e^{0.1\sin t},
    e^{0.2(1-\cos t)},
    e^{0.2\sin t}
  \right)\bm X_0.
\]
The source-free initial field is
\(c(\bm x,0)=1+x_1x_2x_3\), so the continuous total mass is constant.  At the
finest published spacing, a \(\xi=2\), \(N=3484\) run over \(0\le t\le6\)
uses SBF normals, geometry-derived quadrature, mass projection, and two marker
rearrangements.  The quadrature-weighted mass is preserved to roundoff.

\begin{table}[!htbp]
\centering
\caption{Olshanskii--Xu complex moving-manifold mass diagnostic
\cite{OlshanskiiXu2017TraceFEM}.  The error is
\(|M(T)-M(0)|\) for the source-free problem.}
\scriptsize
\setlength{\tabcolsep}{3pt}
\label{tab:ox-complex-mass}
\begin{tabular}{l c c c c c}
\toprule
Benchmark & \(\xi\) & \(N\) & \(h\) &
lit. mass error & our mass error \\
\midrule
OX complex manifold & \(2\) & \(3484\) & \(0.0625\) &
\(1.006{\times}10^{-1}\) & \(0\) \\
\bottomrule
\end{tabular}
\end{table}

This experiment should not be read as an accuracy comparison: no exact
solution is available, and the reported quantity is precisely the invariant
enforced by the scalar projection.  Nor does the zero in
Table~\ref{tab:ox-complex-mass} imply that the pointwise solution is exact.  Its
purpose is narrower: it verifies that geometry-model quadrature, mass
projection, and two changes of the Lagrangian marker set remain mutually
consistent during a long, strongly deforming calculation.  The published
Olshanskii--Xu mass error is included only to place the conservation diagnostic
on the scale reported in the source paper.

\bibliographystyle{cas-model2-names}
\bibliography{main}

@article{reeger_sphere,
author = {Reeger, Jonah A. and Fornberg, Bengt},
title = {Numerical Quadrature over the Surface of a Sphere},
journal = {Studies in Applied Mathematics},
volume = {137},
number = {2},
pages = {174-188},
doi = {https://doi.org/10.1111/sapm.12106},
url = {https://onlinelibrary.wiley.com/doi/abs/10.1111/sapm.12106},
eprint = {https://onlinelibrary.wiley.com/doi/pdf/10.1111/sapm.12106},
year = {2016}
}

@article{reeger_smooth,
	title = {Numerical quadrature over smooth surfaces with boundaries},
	volume = {355},
	issn = {0021-9991},
	url = {https://www.sciencedirect.com/science/article/pii/S0021999117308483},
	doi = {https://doi.org/10.1016/j.jcp.2017.11.010},
	journal = {Journal of Computational Physics},
	author = {Reeger, Jonah A. and Fornberg, Bengt},
	year = {2018},
	pages = {176--190},
}

@article{reeger_closed,
author = {Reeger, Jonah and Fornberg, B. and Watts, M.},
year = {2016},
month = {10},
pages = {20160401},
title = {Numerical quadrature over smooth, closed surfaces},
volume = {472},
journal = {Proceedings. Mathematical, physical, and engineering sciences},
doi = {10.1098/rspa.2016.0401}
}

@article{petras_meshfree_2022,
	title = {Meshfree {Semi}-{Lagrangian} {Methods} for {Solving} {Surface} {Advection} {PDEs}},
	volume = {93},
	issn = {1573-7691},
	url = {https://doi.org/10.1007/s10915-022-01966-w},
	doi = {10.1007/s10915-022-01966-w},
	number = {1},
	journal = {Journal of Scientific Computing},
	author = {Petras, Argyrios and Ling, Leevan and Ruuth, Steven J.},
	month = aug,
	year = {2022},
	pages = {11},
}

@article{shankar_mesh-free_2018,
	title = {Mesh-free semi-{Lagrangian} methods for transport on a sphere using radial basis functions},
	volume = {366},
	issn = {0021-9991},
	url = {https://www.sciencedirect.com/science/article/pii/S0021999118302213},
	doi = {https://doi.org/10.1016/j.jcp.2018.04.007},
	journal = {Journal of Computational Physics},
	author = {Shankar, Varun and Wright, Grady B.},
	year = {2018},
	pages = {170--190},
}

@article{shankar_rbf-loi_2018,
	title = {{RBF}-{LOI}: {Augmenting} {Radial} {Basis} {Functions} ({RBFs}) with {Least} {Orthogonal} {Interpolation} ({LOI}) for solving {PDEs} on surfaces},
	volume = {373},
	issn = {0021-9991},
	url = {https://www.sciencedirect.com/science/article/pii/S0021999118304765},
	doi = {https://doi.org/10.1016/j.jcp.2018.07.015},
	journal = {Journal of Computational Physics},
	author = {Shankar, Varun and Narayan, Akil and Kirby, Robert M.},
	year = {2018},
	pages = {722--735},
}

@article{shankar_hyperviscosity-based_2018,
	title = {Hyperviscosity-based stabilization for radial basis function-finite difference ({RBF}-{FD}) discretizations of advection–diffusion equations},
	volume = {372},
	issn = {0021-9991},
	url = {https://www.sciencedirect.com/science/article/pii/S0021999118304145},
	doi = {https://doi.org/10.1016/j.jcp.2018.06.036},
	journal = {Journal of Computational Physics},
	author = {Shankar, Varun and Fogelson, Aaron L.},
	year = {2018},
	pages = {616--639},
}

@article{ShankarWrightNarayan2020,
  author = {Shankar, Varun and Wright, Grady B. and Narayan, Akil},
  title = {A Robust Hyperviscosity Formulation for Stable {RBF-FD} Discretizations of Advection-Diffusion-Reaction Equations on Manifolds},
  journal = {SIAM Journal on Scientific Computing},
  volume = {42},
  number = {4},
  pages = {A2371--A2401},
  year = {2020},
  doi = {10.1137/19M1288747}
}

@article{PetrasLingPiretRuuth2019LSRBFCPM,
  author = {Petras, A. and Ling, L. and Piret, C. and Ruuth, S. J.},
  title = {A Least-Squares Implicit {RBF-FD} Closest Point Method and Applications to {PDEs} on Moving Surfaces},
  journal = {Journal of Computational Physics},
  volume = {381},
  pages = {146--161},
  year = {2019},
  doi = {10.1016/j.jcp.2018.12.031}
}

@article{petras_rbf-fd_2018,
	title = {An {RBF}-{FD} closest point method for solving {PDEs} on surfaces},
	volume = {370},
	issn = {0021-9991},
	url = {https://www.sciencedirect.com/science/article/pii/S002199911830322X},
	doi = {https://doi.org/10.1016/j.jcp.2018.05.022},
	journal = {Journal of Computational Physics},
	author = {Petras, A. and Ling, L. and Ruuth, S. J.},
	year = {2018},
	pages = {43--57},
}

@article{ChenLing2020KernelHeatTransport,
  author = {Chen, Meng and Ling, Leevan},
  title = {Kernel-Based Collocation Methods for Heat Transport on Evolving Surfaces},
  journal = {Journal of Computational Physics},
  volume = {405},
  pages = {109166},
  year = {2020},
  doi = {10.1016/j.jcp.2019.109166}
}

@article{LubichMansourVenkataraman2013BDFEvolving,
  author = {Lubich, Christian and Mansour, Dhia and Venkataraman, Chandrasekhar},
  title = {Backward Difference Time Discretization of Parabolic Differential Equations on Evolving Surfaces},
  journal = {IMA Journal of Numerical Analysis},
  volume = {33},
  number = {4},
  pages = {1365--1385},
  year = {2013},
  doi = {10.1093/imanum/drs044}
}

@article{DziukElliott2007EvolvingSurfaces,
  author = {Dziuk, Gerhard and Elliott, Charles M.},
  title = {Finite Elements on Evolving Surfaces},
  journal = {IMA Journal of Numerical Analysis},
  volume = {27},
  number = {2},
  pages = {262--292},
  year = {2007},
  doi = {10.1093/imanum/drl023}
}

@article{DziukElliott2007SurfaceFiniteElements,
  author  = {Dziuk, Gerhard and Elliott, Charles M.},
  title   = {Surface Finite Elements for Parabolic Equations},
  journal = {Journal of Computational Mathematics},
  volume  = {25},
  number  = {4},
  pages   = {385--407},
  year    = {2007},
  url     = {https://www.global-sci.com/JCM/article/view/11834}
}

@article{dziuk_fully_2012,
	title = {A Fully Discrete Evolving Surface Finite Element Method},
	volume = {50},
	url = {https://doi.org/10.1137/110828642},
	doi = {10.1137/110828642},
	number = {5},
	journal = {SIAM Journal on Numerical Analysis},
	author = {Dziuk, Gerhard and Elliott, Charles M.},
	year = {2012},
	note = {\_eprint: https://doi.org/10.1137/110828642},
	pages = {2677--2694},
}

@article{DziukElliott2013, 
        title={Finite element methods for surface PDEs}, volume={22}, DOI={10.1017/S0962492913000056}, journal={Acta Numerica}, author={Dziuk, Gerhard and Elliott, Charles M.}, year={2013}, pages={289–396}}

@article{elliott_evolving_2015,
	title = {Evolving surface finite element method for the {Cahn}–{Hilliard} equation},
	volume = {129},
	issn = {0945-3245},
	url = {https://doi.org/10.1007/s00211-014-0644-y},
	doi = {10.1007/s00211-014-0644-y},
	number = {3},
	journal = {Numerische Mathematik},
	author = {Elliott, Charles M. and Ranner, Thomas},
	month = mar,
	year = {2015},
	pages = {483--534},
}

@article{olshanskii_eulerian_2014,
	title = {An {Eulerian} {Space}-{Time} {Finite} {Element} {Method} for {Diffusion} {Problems} on {Evolving} {Surfaces}},
	volume = {52},
	url = {https://doi.org/10.1137/130918149},
	doi = {10.1137/130918149},
	number = {3},
	journal = {SIAM Journal on Numerical Analysis},
	author = {Olshanskii, Maxim A. and Reusken, Arnold and Xu, Xianmin},
	year = {2014},
	note = {\_eprint: https://doi.org/10.1137/130918149},
	pages = {1354--1377},
}

@article{kovacs_convergent_2019,
	title = {A convergent evolving finite element algorithm for mean curvature flow of closed surfaces},
	volume = {143},
	issn = {0945-3245},
	url = {https://doi.org/10.1007/s00211-019-01074-2},
	doi = {10.1007/s00211-019-01074-2},
	number = {4},
	journal = {Numerische Mathematik},
	author = {Kovács, Balázs and Li, Buyang and Lubich, Christian},
	month = dec,
	year = {2019},
	pages = {797--853},
}

@article{olshanskii_finite_2009,
	title = {A {Finite} {Element} {Method} for {Elliptic} {Equations} on {Surfaces}},
	volume = {47},
	url = {https://doi.org/10.1137/080717602},
	doi = {10.1137/080717602},
	number = {5},
	journal = {SIAM Journal on Numerical Analysis},
	author = {Olshanskii, Maxim A. and Reusken, Arnold and Grande, Jörg},
	year = {2009},
	note = {\_eprint: https://doi.org/10.1137/080717602},
	pages = {3339--3358},
}

@article{olshanskii_stabilized_2014,
	title = {A stabilized finite element method for advection–diffusion equations on surfaces},
	volume = {34},
	issn = {0272-4979},
	url = {https://doi.org/10.1093/imanum/drt016},
	doi = {10.1093/imanum/drt016},
	number = {2},
	journal = {IMA Journal of Numerical Analysis},
	author = {Olshanskii, Maxim A. and Reusken, Arnold and Xu, Xianmin},
	month = apr,
	year = {2014},
	note = {\_eprint: https://academic.oup.com/imajna/article-pdf/34/2/732/2003745/drt016.pdf},
	pages = {732--758},
}

@article{grande_analysis_2018,
	title = {Analysis of a {High}-{Order} {Trace} {Finite} {Element} {Method} for {PDEs} on {Level} {Set} {Surfaces}},
	volume = {56},
	url = {https://doi.org/10.1137/16M1102203},
	doi = {10.1137/16M1102203},
	number = {1},
	journal = {SIAM Journal on Numerical Analysis},
	author = {Grande, Jörg and Lehrenfeld, Christoph and Reusken, Arnold},
	year = {2018},
	note = {\_eprint: https://doi.org/10.1137/16M1102203},
	pages = {228--255},
}

@article{elliott_unified_2021,
	title = {A unified theory for continuous-in-time evolving finite element space approximations to partial differential equations in evolving domains},
	volume = {41},
	issn = {0272-4979},
	url = {https://doi.org/10.1093/imanum/draa062},
	doi = {10.1093/imanum/draa062},
	number = {3},
	journal = {IMA Journal of Numerical Analysis},
	author = {Elliott, C M and Ranner, T},
	month = jul,
	year = {2021},
	note = {\_eprint: https://academic.oup.com/imajna/article-pdf/41/3/1696/38983520/draa062.pdf},
	pages = {1696--1845},
}

@inproceedings{olshanskii_trace_2017_static,
	address = {Cham},
	title = {Trace {Finite} {Element} {Methods} for {PDEs} on {Surfaces}},
	isbn = {978-3-319-71431-8},
	booktitle = {Geometrically {Unfitted} {Finite} {Element} {Methods} and {Applications}},
	publisher = {Springer International Publishing},
	author = {Olshanskii, Maxim A. and Reusken, Arnold},
	editor = {Bordas, Stéphane P. A. and Burman, Erik and Larson, Mats G. and Olshanskii, Maxim A.},
	year = {2017},
	pages = {211--258},
}

@article{OlshanskiiXu2017TraceFEM,
  author = {Olshanskii, Maxim A. and Xu, Xianmin},
  title = {A Trace Finite Element Method for {PDEs} on Evolving Surfaces},
  journal = {SIAM Journal on Scientific Computing},
  volume = {39},
  number = {4},
  pages = {A1301--A1319},
  year = {2017},
  doi = {10.1137/16M1099388}
}

@article{LehrenfeldOlshanskiiXu2018TraceFEM,
  author = {Lehrenfeld, Christoph and Olshanskii, Maxim A. and Xu, Xianmin},
  title = {A Stabilized Trace Finite Element Method for Partial Differential Equations on Evolving Surfaces},
  journal = {SIAM Journal on Numerical Analysis},
  volume = {56},
  number = {3},
  pages = {1643--1672},
  year = {2018},
  doi = {10.1137/17M1148633}
}

@article{Stone1990,
  author    = {Stone, Howard A.},
  title     = {A Simple Derivation of the Time-Dependent Convective-Diffusion Equation for Surfactant Transport along a Deforming Interface},
  journal   = {Physics of Fluids A: Fluid Dynamics},
  volume    = {2},
  number    = {1},
  pages     = {111--112},
  year      = {1990},
  doi       = {10.1063/1.857686},
  publisher = {AIP Publishing}
}

@article{XuZhaoLowengrub2006,
  author  = {Xu, Jian-Jun and Li, Zhilin and Lowengrub, John and Zhao, Hongkai},
  title   = {A Level-Set Method for Interfacial Flows with Surfactant},
  journal = {Journal of Computational Physics},
  volume  = {212},
  number  = {2},
  pages   = {590--616},
  year    = {2006},
  doi     = {10.1016/j.jcp.2005.07.016}
}

@article{ElliottStinner2010,
  author  = {Elliott, Charles M. and Stinner, Bj{\"o}rn},
  title   = {Modeling and Computation of Two Phase Geometric Biomembranes Using Surface Finite Elements},
  journal = {Journal of Computational Physics},
  volume  = {229},
  number  = {18},
  pages   = {6585--6612},
  year    = {2010},
  doi     = {10.1016/j.jcp.2010.05.014}
}

@article{RaetzRoeger2014,
  author  = {R{\"a}tz, Andreas and R{\"o}ger, Matthias},
  title   = {Symmetry Breaking in a Bulk-Surface Reaction-Diffusion Model for Signalling Networks},
  journal = {Nonlinearity},
  volume  = {27},
  number  = {8},
  pages   = {1805--1827},
  year    = {2014},
  doi     = {10.1088/0951-7715/27/8/1805}
}

@article{BarreiraElliottMadzvamuse2011,
  author  = {Barreira, Raquel and Elliott, Charles M. and Madzvamuse, Anotida},
  title   = {The Surface Finite Element Method for Pattern Formation on Evolving Biological Surfaces},
  journal = {Journal of Mathematical Biology},
  volume  = {63},
  pages   = {1095--1119},
  year    = {2011},
  doi     = {10.1007/s00285-011-0401-0}
}

@article{EilksElliott2008,
  author  = {Eilks, Carsten and Elliott, Charles M.},
  title   = {Numerical Simulation of Dealloying by Surface Dissolution via the Evolving Surface Finite Element Method},
  journal = {Journal of Computational Physics},
  volume  = {227},
  number  = {23},
  pages   = {9727--9741},
  year    = {2008},
  doi     = {10.1016/j.jcp.2008.07.023}
}

@article{Pudykiewicz2006,
  author  = {Pudykiewicz, Janusz A.},
  title   = {Numerical Solution of the Reaction-Advection-Diffusion Equation on the Sphere},
  journal = {Journal of Computational Physics},
  volume  = {213},
  number  = {1},
  pages   = {358--390},
  year    = {2006},
  doi     = {10.1016/j.jcp.2005.08.021}
}

@article{fuselier2013highorder,
  author  = {Fuselier, Edward J. and Wright, Grady B.},
  title   = {A High-Order Kernel Method for Diffusion and Reaction-Diffusion Equations on Surfaces},
  journal = {Journal of Scientific Computing},
  volume  = {56},
  number  = {3},
  pages   = {535--565},
  year    = {2013},
  doi     = {10.1007/s10915-013-9688-x}
}

@article{lehto2017rbfcompact,
  author  = {Lehto, Erik and Shankar, Varun and Wright, Grady B.},
  title   = {A Radial Basis Function ({RBF}) Compact Finite Difference ({FD}) Scheme for Reaction-Diffusion Equations on Surfaces},
  journal = {SIAM Journal on Scientific Computing},
  volume  = {39},
  number  = {5},
  pages   = {A2129--A2151},
  year    = {2017},
  doi     = {10.1137/16M1095457}
}

@article{piret2012orthogonal,
  author  = {Piret, C{\'e}cile},
  title   = {The Orthogonal Gradients Method: A Radial Basis Functions Method for Solving Partial Differential Equations on Arbitrary Surfaces},
  journal = {Journal of Computational Physics},
  volume  = {231},
  number  = {14},
  pages   = {4662--4675},
  year    = {2012},
  doi     = {10.1016/j.jcp.2012.03.007}
}

@article{shankar2015rbffdSurfaces,
  author  = {Shankar, Varun and Wright, Grady B. and Kirby, Robert M. and Fogelson, Aaron L.},
  title   = {A Radial Basis Function ({RBF})-Finite Difference ({FD}) Method for Diffusion and Reaction-Diffusion Equations on Surfaces},
  journal = {Journal of Scientific Computing},
  volume  = {63},
  number  = {3},
  pages   = {745--768},
  year    = {2015},
  doi     = {10.1007/s10915-014-9914-1}
}

@article{RuuthMerriman2008,
  author  = {Ruuth, Steven J. and Merriman, Barry},
  title   = {A Simple Embedding Method for Solving Partial Differential Equations on Surfaces},
  journal = {Journal of Computational Physics},
  volume  = {227},
  number  = {3},
  pages   = {1943--1961},
  year    = {2008},
  doi     = {10.1016/j.jcp.2007.10.009}
}

@article{MacdonaldRuuth2010,
  author  = {Macdonald, Colin B. and Ruuth, Steven J.},
  title   = {The Implicit Closest Point Method for the Numerical Solution of Partial Differential Equations on Surfaces},
  journal = {SIAM Journal on Scientific Computing},
  volume  = {31},
  number  = {6},
  pages   = {4330--4350},
  year    = {2010},
  doi     = {10.1137/080740003}
}

@article{MarzMacdonald2012,
  author  = {M{\"a}rz, Thomas and Macdonald, Colin B.},
  title   = {Calculus on Surfaces with General Closest Point Functions},
  journal = {SIAM Journal on Numerical Analysis},
  volume  = {50},
  number  = {6},
  pages   = {3303--3328},
  year    = {2012},
  doi     = {10.1137/120865537}
}

@article{ChenLing2020Extrinsic,
  author  = {Chen, Meng and Ling, Leevan},
  title   = {Extrinsic Meshless Collocation Methods for {PDEs} on Manifolds},
  journal = {SIAM Journal on Numerical Analysis},
  volume  = {58},
  number  = {2},
  pages   = {988--1007},
  year    = {2020},
  doi     = {10.1137/17M1158641}
}

@article{LiangZhao2013,
  author  = {Liang, Jian and Zhao, Hongkai},
  title   = {Solving Partial Differential Equations on Point Clouds},
  journal = {SIAM Journal on Scientific Computing},
  volume  = {35},
  number  = {3},
  pages   = {A1461--A1486},
  year    = {2013},
  doi     = {10.1137/120869730}
}

@article{TraskKuberry2020,
author = {Trask, Nathaniel and Kuberry, Paul},
year = {2020},
title = {Compatible meshfree discretization of surface {PDEs}},
journal = {Computational Particle Mechanics},
volume = {7},
pages = {271--277},
doi = {10.1007/s40571-019-00251-2}
}

@article{GrossTraskKuberryAtzberger2020,
  author  = {Gross, B. J. and Trask, N. and Kuberry, P. and Atzberger, P. J.},
  title   = {Meshfree Methods on Manifolds for Hydrodynamic Flows on Curved Surfaces: A Generalized Moving Least-Squares ({GMLS}) Approach},
  journal = {Journal of Computational Physics},
  volume  = {409},
  pages   = {109340},
  year    = {2020},
  doi     = {10.1016/j.jcp.2020.109340}
}

@article{GriffithLuo2017HybridIB,
  author  = {Griffith, Boyce E. and Luo, Xiaoyu},
  title   = {Hybrid Finite Difference/Finite Element Immersed Boundary Method},
  journal = {International Journal for Numerical Methods in Biomedical Engineering},
  volume  = {33},
  number  = {12},
  pages   = {e2888},
  year    = {2017},
  doi     = {10.1002/cnm.2888}
}

@article{JonesBoslerKuberryWright2023,
  author  = {Jones, Andrew M. and Bosler, Peter A. and Kuberry, Paul A. and Wright, Grady B.},
  title   = {Generalized Moving Least Squares vs. Radial Basis Function Finite Difference Methods for Approximating Surface Derivatives},
  journal = {Computers \& Mathematics with Applications},
  volume  = {147},
  pages   = {1--13},
  year    = {2023},
  doi     = {10.1016/j.camwa.2023.07.015}
}

@article{DavydovSchaback2018,
  author  = {Davydov, Oleg and Schaback, Robert},
  title   = {Minimal Numerical Differentiation Formulas},
  journal = {Numerische Mathematik},
  volume  = {140},
  number  = {3},
  pages   = {555--592},
  year    = {2018},
  doi     = {10.1007/s00211-018-0973-3}
}

@article{wright2023mgm,
  title={{MGM}: A meshfree geometric multilevel method for systems arising from elliptic equations on point cloud surfaces},
  author={Wright, Grady B and Jones, Andrew and Shankar, Varun},
  journal={SIAM Journal on Scientific Computing},
  volume={45},
  number={2},
  pages={A312--A337},
  year={2023},
  doi={10.1137/22M1490338},
  publisher={SIAM}
}

@article{Peskin2002IB,
  author = {Peskin, Charles S.},
  title = {The Immersed Boundary Method},
  journal = {Acta Numerica},
  volume = {11},
  pages = {479--517},
  year = {2002},
  doi = {10.1017/S0962492902000077}
}

@article{EvansFung1972ErythrocyteGeometry,
  author = {Evans, Evan A. and Fung, Yuan-Cheng},
  title = {Improved Measurements of the Erythrocyte Geometry},
  journal = {Microvascular Research},
  volume = {4},
  number = {4},
  pages = {335--347},
  year = {1972},
  doi = {10.1016/0026-2862(72)90069-6}
}

@article{FedosovCaswellKarniadakis2010RBC,
  author = {Fedosov, Dmitry A. and Caswell, Bruce and Karniadakis, George Em},
  title = {A Multiscale Red Blood Cell Model with Accurate Mechanics,
           Rheology, and Dynamics},
  journal = {Biophysical Journal},
  volume = {98},
  number = {10},
  pages = {2215--2225},
  year = {2010},
  doi = {10.1016/j.bpj.2010.02.002}
}

@article{hangelbroek2026sla,
author = {Hangelbroek, Thomas and Rieger, Christian and Wright, Grady B.},
title = {A Semi-{L}agrangian Scheme on Embedded Manifolds Using Generalized Local Polynomial Reproductions},
journal = {SIAM Journal on Numerical Analysis},
volume = {64},
number = {3},
pages = {853-875},
year = {2026},
doi = {10.1137/25M1788476},
}

\end{document}